\documentclass[11pt]{amsart}

\usepackage[margin=1.2in]{geometry}

\usepackage[parfill]{parskip}    % Activate to begin paragraphs with an empty line rather than an indent
\DeclareRobustCommand{\SkipTocEntry}[4]{}  %Needed to fixed amsart bug with table of contents
\usepackage{amsmath,amsthm,amsfonts,amssymb,graphicx}

\usepackage{latexsym,amssymb}
\usepackage{amsthm}
\newcommand{\ov}{\overline}

\newtheorem{lem}{Lemma}[subsection]
\newtheorem{thm}[lem]{Theorem}
\newtheorem{prop}[lem]{Proposition}
\newtheorem{cor}[lem]{Corollary}

\newtheorem{defn}[lem]{Definition}

\theoremstyle{definition}

\newtheorem{example}[lem]{Example}
\newtheorem{remark}[lem]{Remark}

\numberwithin{figure}{subsection}
\numberwithin{equation}{subsection}

\newcommand{\MS}{{\medskip}}
\newcommand{\NI}{{\noindent}}

\DeclareMathOperator{\supp}{supp}
\DeclareMathOperator{\im}{Im}

\DeclareMathOperator{\id}{id}

\DeclareMathOperator{\Int}{Int\,}

\DeclareMathOperator{\aff}{aff}

\DeclareMathOperator{\Ham}{Ham}

\renewcommand{\Tilde}{\widetilde}

\newcommand{\INT}{{\rm Int\,}}
\newcommand{\ip}[1]{\langle #1 \rangle}

\newcommand{\abs}[1]{\left| #1 \right|}
\newcommand{\norm}[1]{\left\| #1 \right\|}

\def\co{\colon\thinspace}

\def\C{\mathbb{C}}

\def\R{\mathbb{R}}
\def\Z{\mathbb{Z}}
\def\T{\mathbb{T}}
\def\CP{\mathbb{CP}}
\def\bP{\mathbb{P}}
\def\bD{\mathbb{D}}
\def\bA{\mathbb{A}}

\def\cDP{\mathcal{FP}}

\def\cSP{\mathcal{SP}}
\def\cE{\mathcal{E}}
\def\cF{\mathcal{F}}
\def\cH{\mathcal{H}}
\def\cL{\mathcal{L}}

\def\cN{\mathcal{N}}
\def\cO{\mathcal{O}}
\def\cP{\mathcal{P}}
\def\cU{\mathcal{U}}

\def\mf{\mathfrak}
\def\ft{{\bf\mf t}}

\def\a{\alpha}
\def\b{\beta}
\def\g{\gamma}
\def\d{\partial}
\def\De{\Delta}
\def\de{\delta}
\def\e{\epsilon}
\def\eps{\epsilon}

\def\ka{\kappa}
\def\l{\lambda}
\def\la{\lambda}

\def\r{\rho}

\def\th{\theta}

\def\vp{\varphi}

\def\w{\omega}
\def\z{\zeta}
\def\vp{\varphi}

\def\om{\omega}

\def\La{\Lambda}

\def\G{\Gamma}

\begin{document}

%--------------------------------------------------------------
\title[Extended probes]{Displacing Lagrangian toric fibers by extended probes}

\author[M. Abreu]{Miguel Abreu}
\address{Centro de An\'{a}lise Mathem\'{a}tica, Geometria e Sistemas Din\^{a}micos, Departamento de Mathem\'{a}tica\\
Instituto Superior T\'{e}cnico}
\email{mabreu@math.ist.utl.pt}

\author[M. S. Borman]{Matthew Strom Borman}
\address{Department of Mathematics\\ 
University of Chicago}
\email{borman@math.uchicago.edu}

\author[D. McDuff]{Dusa McDuff}
\address{Department of Mathematics, Barnard College, Columbia University}
\email{dusa@math.columbia.edu}

%\thanks{The first author is partially supported by Funda\c{c}\~{a}o para a Ci\^{e}ncia e a Tecnologia
%(FCT/Portugal). The second author is partially supported by NSF-grant DMS 1006610, and the third by 
%NSF-grant DMS 0905191.}
%\date{\today}

%\keywords{Differential geometry, algebraic geometry}

%\dedicatory{}

\begin{abstract}
	In this paper we introduce a new way of displacing Lagrangian
	fibers in toric symplectic manifolds, 
	a generalization of McDuff's original method of probes.
	Extended probes are formed by deflecting one 
	 probe by another auxiliary probe.  
	Using them, we are able to displace 
	all fibers in Hirzebruch surfaces 
	except those already known to be nondisplaceable, 
	and can also displace an open dense set of
	fibers in the weighted projective space 
	$\bP(1,3,5)$ after resolving the singularities.
	We also investigate the displaceability question 
	in sectors and their resolutions.  There are still many cases 
	in which there is an open set of fibers whose displaceability status is unknown.
\end{abstract} 

\maketitle

%--------------------------------------------------------------

%\addtocontents{toc}{\SkipTocEntry}  %Needed so Contents appear correctly
%\section*{Contents}				%when using the parskip package

\tableofcontents

%%%%%%%%%%%%%%%%%%%%%%%%%%%%%%%%%%%%
%%%%%%%%%%%%%%%%%%%%%%%%%%%%%%%%%%%%
%%%%%%%%%%%%%%%%%%%%%%%%%%%%%%%%%%%%
\section{Introduction and main results}

%%%%%%%%%%%%%%%%%%%%%%%%%%%%%
%%%%%%%%%%%%%%%%%%%%%%%%%%%%%
\subsection{Introduction}

Let $(M, \w)$ be a 
connected, but possibly noncompact, symplectic manifold without boundary.
A subset $X \subset M$ is said to be \emph{displaceable} if there is a compactly supported
Hamiltonian diffeomorphism $\phi \in \Ham(M, \w)$ such that $\phi(X) \cap \ov{X} = \emptyset$,
and if no such $\phi$ exists, then $X$ is said to be \emph{nondisplaceable}. 
Ever since Arnold conjectured that certain Lagrangian submanifolds are nondisplaceable, a central theme 
in symplectic topology has been to determine what subsets are displaceable and what subsets are not.

We will work with symplectic toric manifolds, where a $2n$-dimensional symplectic manifold
$(M^{2n}, \w)$ is \emph{toric} if it is equipped with an effective Hamiltonian action
of an $n$-torus $\T^{n}$.
Associated to a toric manifold is a 
moment map
$$	
	\Phi\co M^{2n} \to \R^{n} \quad\mbox{given by}\quad \Phi(x) = (\Phi_{1}(x), \dots, \Phi_{n}(x))
$$
where
 the Hamiltonian flow of $\Phi_{k}$ generates the action of the $k$-th component circle in
$\T^{n}$, and if 
 $\Phi$ is  proper the image of $\Phi$ is a polytope $\Delta = \Phi(M)$,
called the \emph{moment polytope}.
Symplectic toric manifolds come with a natural family of Lagrangian tori, namely for each 
$u \in \Int \Phi(M)$ the fiber $L_{u} = \Phi^{-1}(u)$ is a Lagrangian torus and an orbit of the $\T^{n}$-action.
In this paper we will restrict ourselves to the study of the displaceability of the Lagrangian toric fibers in 
symplectic toric manifolds, and will assume unless explicit mention is made to contrary
that the moment map is proper. 

Determining which Lagrangian fibers are nondisplaceable 
involves two complementary tasks: 
building and computing invariants that obstruct displaceability, and finding general sufficient criteria for 
when Lagrangian fibers are displaceable.  There are now many well-developed Floer-theoretic tools that can 
be used to prove that certain Lagrangian fibers are nondisplaceable: quasi-states from Hamiltonian Floer 
homology \cite{Bo11, EP06, EP09,FOOO11a}, Lagrangian Floer homology \cite{BC09, Ch08, FOOO10a, FOOO11,FOOO10}, 
and quasi-map Floer homology \cite{Wd11, WW11}.  In contrast the only known general method 
for proving that a Lagrangian fiber is displaceable was introduced by McDuff in \cite{Mc11} and
involves the affine geometric notion of a probe in the moment polytope. 
%Strom
Chekanov--Schlenk \cite{CS10} also used this method of displacement in a slightly different context.
The method of probes was later reinterpreted in \cite{AbM11} in terms of symplectic reduction.  However the main contribution of that paper was to the other side of the problem in that it allowed 
one to deduce many nondisplaceability results from a few basic examples.

For very simple examples the method of probes perfectly complements the proven nondisplaceability results. For instance if $(M^{2n}, \w, \Phi)$ is closed and monotone with $2n \leq 6$, then the method of probes displaces everything except an identified fiber $L_{u_{0}}$ \cite[Theorem 1.1]{Mc11}, and $L_{u_{0}}$ is known to be nondisplaceable.  

However, in general  the method of probes does not perfectly complement the proven nondisplaceability results.  The simplest such example is a Hirzebruch surface $F_{2k+1}$ for $k\ge 1$, which is
the projectivization $\bP(\cO_{2k+1}\oplus\C)$ where $\cO_{2k+1}\to \C P^1$ is a line bundle of Chern class $2k+1$.   Here these methods leave a line segment 
of points with unknown displaceability properties.
The next basic example is (a blow up of) the weighted projective space $\bP(1,3,5)$, the quotient of $\C^3\smallsetminus \{0\}$ by the group action
$$
e^{2\pi i t}\cdot\bigl(z _0, z_{1} ,z_2\bigr)= \bigr(e^{2\pi i t} z_0,\, e^{2\pi i 3t}z_1,\, e^{2\pi i 5t}z_2\bigr).
$$
It was pointed out in \cite{Mc11} that it is possible to resolve the singularities of 
$\bP(1,3,5)$ by small blow ups in such a way that there is an open set of points not displaceable by probes.  On the other hand for smooth toric $4$-manifolds, the nondisplaceable fibers detected by the Floer theoretic methods of \cite{FOOO10} lie on a finite number of line segments (cf.\ the proof of 
Proposition~\ref{p:noghost}).
One expects that all fibers with vanishing invariants are displaceable.
Our current methods do give better results than standard probes in many cases, but still are not powerful enough to prove this even in four dimensions.

As we show in Proposition~\ref{p:noghost}, the method of quasi-map Floer homology, developed by Woodward \cite{Wd11}, 
gives no more information than 
standard Floer theoretic methods in the closed smooth case.   However, it applies also in the orbifold and noncompact cases and in those cases can give open sets of fibers that are 
nondisplaceable because they have nonvanishing quasi-map invariants (called qW invariants, for short).
See also the orbifold version of the standard approach by Cho and Poddar \cite{Ch11}.  
Figure \ref{f:135} illustrates the current knowledge about the displaceability of points in $\bP(1,3,5)$.  Here the displaceable points are displaced by standard probes.\footnote
{As remarked in \S\ref{ss:gen}  below,
extended probes do not help in triangles.}
As we explain in \S\ref{ss:pot} the open set of
nondisplaceable points
comes from varying the position of certain \lq\lq ghost" facets; 
cf.\ also the proof of Theorem~\ref{t:DND}(i).
These facets are precisely the ones that can be used to resolve the singular points, and once one has used them for this purpose, so that their position is fixed, the numbers of points with nonvanishing invariants decreases.  Thus as one resolves singularities by blowing up, the set of points with nonvanishing qW invariants tends to decrease. 
%%%%%%%%%%%%%%%%%%%%%%%%%%%%%
%%%%%%%%%%%%%%%%%%%%%%%%%%%%%

%%%%%%%%%%%%%%%%%%%%%%%%%%%%%
%%%%%%%%%%%%%%%%%%%%%%%%%%%%%
\subsection{Main results}

In this paper we will introduce a technique for extending a given probe by deflecting it by an auxiliary probe.
In contrast to the nondisplaceability results explained above, this technique
gives no new information for very simple orbifolds such as $\bP(1,3,5)$.  Instead it starts to displace more fibers as we resolve the singularities by blow up.  Because it is a geometric method, it is very sensitive to the exact choice of blow up, i.e.\ to the choice of support constant that determines where the 
new facet is in relation to the others.  
 In higher dimensions one could use this technique in more elaborate ways, for example
 by deflecting a probe several times in different directions. 
However, we will restrict to the $2$ dimensional case since, even in this simple case, 
the results are quite complicated to work out precisely.
The next paragraph describes our results rather informally.  More complete definitions 
and statements are given later.

%%%%%%%%%%%%%%%%%%%%%%%
\subsubsection{The method of extended probes}

We  say that two integral vectors in $\R^2$ are {\bf complementary} if they form a basis for the integral lattice $\Z^2$.  Let $\De$ be a rational polygon in $\R^2$, i.e. the direction vectors $d_F$ of the edges $F$ are integral.  We call it {\bf smooth} 
if the direction vectors  
 at each vertex are complementary.
A {\bf probe} $P$ in $\De$  is a line segment in $\De$ starting at an interior point of some edge $F$ (called its {\bf base facet}), whose direction $v_P$ is integral and complementary to $d_F$. 
  By \cite{Mc11}, if $u\in \De$ lies less than halfway along a probe $P$ then the corresponding fiber $L_u$ is displaceable.  For short, we will say that the point $u$ itself is displaceable.

A probe  $Q$ is said to be {\bf symmetric} if it is also a probe when its direction is reversed. That means that  its exit point also lies at an interior point of an edge $F'$ and also that $d_{F'}$ is complementary to $v_Q$.  All points other than the midpoint of a symmetric probe are displaceable.    Moreover there is an affine reflection $A_Q$ of a neighborhood of $Q$ in $\De$ that reverses its direction. 

In this paper we show how to lengthen a probe $P$  so that it still has the property that points less than 
halfway 
along are displaceable. There are three basic methods whose effects are described in the following theorems.
\begin{itemize}
\item[(a)] Theorem~\ref{t:noflag}: deflecting $P$ via a {\bf symmetric probe}.
\item[(b)] Theorem~\ref{t:pdeflected}: deflecting $P$ by a {\bf parallel probe}  $Q$, i.e.\ one whose base facet 
$F_Q$ is parallel to the direction $v_P$  of $P$; these probes have  flags that are parallelograms.
\item[(c)] Theorem~\ref{t:deflected}:  deflecting $P$ by an arbitrary probe $Q$; these probes have  trapezoidal flags.
\end{itemize}
In  case (a), the extended probe $\cP$ (often denoted $\cSP$ for clarity) is a union of  line segments,  and can be used to displace points less than halfway along it, whether these points lie before or after the 
 intersection with $Q$.
In  cases (b) and (c), the extended probe $\cP$ (or more precisely $\cDP$)
 is the union of the initial segment of $P$ together with a \lq\lq flag"  emanating from $Q$, a parallelogram in case (b) 
 (cf. Figure~\ref{f:extendedprobe}) or trapezoid in case (c)  (cf. Figure~\ref{f:epNP}).  One can 
 only displace points less than halfway along $\cDP$  that also lie before  the 
 intersection with $Q$.
  Another difficulty with case (c) is that the flag may taper to a point, which severely restricts its length.  In particular, if 
  the ray in direction $P$ 
   meets the base facet $F_Q$ of $Q$ at $y_{PQ}$
  then
   $\cDP$ cannot be longer than the line segment  from the initial point $b_P$ 
   of $P$ to its intersection $y_{PQ}$ with $F_Q$.
     Therefore this variant is less useful, though it does displace some new points 
     in certain bounded regions: see Proposition~\ref{p:fvAnR}.
%%%%%%%%%%%%%%%%%%%%%%%%

%%%%%%%%%%%%%%%%%%%%%%%
\subsubsection{Examples}

\begin{itemize}
\item {\it Symmetric extended probes solve the displaceability problem for   
Hirzebruch surfaces.}
Previous results show that all Hirzebruch surfaces $F_n$ for $n\ge 0$, with the exception of certain surfaces $F_1$ that are (small) blow ups of $\C P^2$, have precisely one fiber with nonvanishing Floer homology.  Proposition~\ref{p:H} shows that all the other fibers are displaceable by symmetric extended probes.  
(The case $n=1$ can be fully understood using standard probes.)

\begin{figure}[ht]
  
  \centering
  \includegraphics[width=5in]{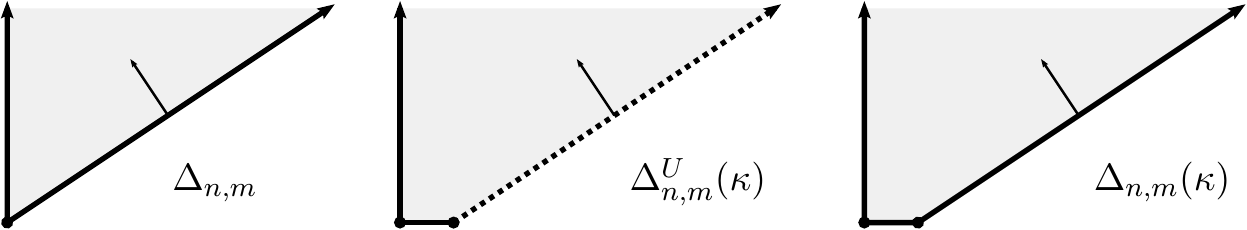} 
   \caption{Three basic moment polytopes; in each case, the slant edge has normal vector $(-n,m)$.}
   \label{fig:X1}
\end{figure}

\item {\it Using parallel extended probes.}
 We next consider open regions of two elementary types.
Section~\ref{s:DinC2} 
considers the regions
$$
	U_{n,m}(\ka) =\bigl \{z \in \C^{2} : -n\abs{z_{1}}^{2} + m\abs{z_{2}}^{2} + \ka > 0\bigr\}
$$
where $m > n \geq 1$ are relatively prime integers and $\ka > 0$, with moment polytope 
$$
\De^U_{n,m}(\ka)=\bigl\{(x_1,x_2): x_1, x_{2} \ge 0, -nx_1+mx_2+\ka>0\bigr\}
$$ 
as in  Figure~\ref{fig:X1}. 
Lemmas~\ref{l:DOsecP} and \ref{l:openEP}  
specify the points that can be displaced by probes and extended probes.   
Figure~\ref{f:DXnm} shows 
 that there are open regions of points in the moment polytope of $U_{n,m}(\ka)$ where extended probes are needed.

The second basic region represents the quotient $M_{n,m}= \C^2/\Gamma$,  where the generator
$\zeta = e^{2\pi i/m}$ of  $\Gamma =\Z/m\Z$ acts via
$$
\zeta\cdot\bigl(z_1,z_2\bigr) = \bigl(\zeta^{n}z_1,\zeta z_2\bigr).
$$
Its moment polytope is the sector
$$
\De_{n,m}: = \bigl\{(x_1,x_2): x_1,x_2\ge 0,  -nx_1+mx_2\ge0\bigr\};
$$
cf.\ Figure~\ref{fig:X1} and \S\ref{s:Mnm}.
Here we contrast the set of probe displaceable points with
the set of points that are nondisplaceable because they 
have nonvanishing qW invariants.  
Theorem~\ref{t:DND} states the precise result.
There is an open set of points with unknown behavior whenever the Hirzebruch--Jung continued fraction 
\begin{equation}\label{eq:HJk}
\frac nm: =\frac{1}{E_1 - \frac 1{E_2-\dots\frac 1{E_k}}} =: (E_1,\dots,E_k)
\end{equation}
has at least one of $E_1,E_k$ greater than  $2$.
\MS

\item  {\it Resolving a singular vertex: unbounded case.} \,\,
As one resolves the singular point of $\De_{n,m}$ by blowing up, the set of probe displaceable points increases while that of points with nontrivial invariants decreases.
Proposition~\ref{p:asympresolve} explains what happens after a single blow up, with 
exceptional divisor corresponding to the horizontal edge $x_2=\ka$. The moment polytope is then $\De_{n,m}(\ka)$, the closure of $\De_{n,m}^U(\ka)$.   
Remark ~\ref{r:resolveP} (ii) points out that
usually there is an open set of unknown points.

In Subsection~\ref{s:EMR} we look at some complete minimal resolutions 
$\ov\De_{n,m}$.
In this case all qW invariants vanish, while 
typically there is still an open subset of points that we do not know how to displace; cf. 
Figure~\ref{f:5-8TR}.  As we show in Corollary~\ref{c:An}, even in the easy case of an $A_n$ singularity (which has all $E_i=2$) there are $k-1$ lines of  points of unknown status,
where $k$ is as in equation~\eqref{eq:HJk}.  
\MS

\item {\it Special cases of $\De_{n,m}$ and $\ov\De_{n,m}$.}\,\,
The first, the case  $n=1$, is treated  in \S\ref{ss:Om}.    Here the status of all interior points can be determined by our methods.
In fact, Lemmas~\ref{l:DNDm} and \ref{l:O-m} show that 
all points are displaceable except in the case of $\De_{1,m}$ with $m$ odd, in which case there is a ray of points with nontrivial qW invariants: see Figure~\ref{f:1m}.

The second is when $m/n = (E_1,E_2)$  
 in the notation of equation~\eqref{eq:HJk}.  Then, although there may be an open set of 
 unknown points of $\De_{n,m}$, Corollary~\ref{c:k1} shows that there is at 
 most one line of
 such points in the resolution $\ov\De_{n,m}$, namely the (affine) bisector of the angle between the two new edges: see Figure~\ref{f:k1}.
\MS

\item {\it Using general extended probes in
regions of finite area.}\,\, 
 The above results only use extended probes of types (a) and (b), and it is easy to see that 
extended probes of type (c) would give nothing new.
  To demonstrate the use of this kind of extension, Proposition~\ref{p:fvAnR} considers 
  the resolution of  an open finite volume $A_n$-singularity,
  i.e. one whose moment polytope has an open upper boundary $x_2<K$.  
  As shown in Figure~\ref{f:fvAnR}, there are some points that can be reached only
  by these new probes; however there is still an open set of
  unknown points.\MS
\end{itemize}

%%%%%%%%%%%%%%%%%%%%%%%
\subsubsection{Displaceability in compact toric orbifolds and their resolutions}
%%%%%%%%%%%%%%%%%%%%%%%

 Finally, we discuss a simple family of examples with closed moment polytope,
namely the weighted projective spaces $\bP(1,p,q)$  (with $1< p<q$ relatively prime)
and its resolutions.
The moment polytope of $\bP(1,p,q)$ is a triangle with vertices at $(0,0), (p,0)$ and 
$(0,q)$.  Thus it has two singular vertices, the one at $(p,0)$ modelled on
$\De_{p,q}$
and the one at $(0,q)$ modelled on $\De_{q-kp,p}$ where $0<q-kp < p$.
  If these sectors (or their resolutions) have unknown points, and the new edges coming from the resolutions are sufficiently short, one would expect
there to be corresponding  unknown points  in  the compactifcation
$\bP(1,p,q)$ and its resolutions.
This is the case for $\bP(1,p,q)$  itself, since there are no symmetric or parallel extended probes and there are not enough edges for there to be any 
useful probes of type (c): cf.\ \S\ref{ss:gen} below.   
Figure~\ref{f:135} shows the situation for $\bP(1,3,5)$: there is an open set of points with nonzero qW invariants, as well as an open set of points with unknown behavior.

One can always displace more points in  the  resolutions.   Some of these newly displaceable points can be displaced  by the same parallel extended probes that are used in  the resolved sectors $\ov{\De}_{n,m}$.  
However, there is another kind of extended probe,
constructed using a symmetric probe $Q$, that always displaces 
at least a few more points, though typically there is still an open set of
 points with unknown behavior; cf. Proposition~\ref{p:openset}.

The case $\bP(1,3,5)$ is special since in this case one can reverse 
the direction of the deflected 
probes formed from $Q$.
As illustrated in Figure~\ref{f:135R03DaND}, when 
the singularity   of $\bP(1,3,5)$  at $(3,0)$  is resolved,
one can displace an open dense set of fibers by probes or extended probes, leaving just a line segment of points 
 together with one more  point that
cannot be determined by our methods.  
 In particular, all points near $(3,0)$ can be displaced 
 with the help of the symmetric probes $Q$, though they are not all displaceable in 
 the corresponding resolved sector $\ov{\De}_{3,5}$.
The picture does not change significantly when one fully resolves both singular vertices 
in $\bP(1,3,5)$: cf.\ Proposition~\ref{p:135TRDaND}  and Figure~\ref{f:135TRDaND}.

%%%%%%%%%%%%%%%%%%%%%%%
\subsubsection{General results}\label{ss:gen}
%%%%%%%%%%%%%%%%%%%%%%%

 Proposition~\ref{p:noghost} shows that qW invariants give no more information than
standard Floer homology in the
case where the moment polytope $\De$ is a
smooth closed polytope of $\R^{n}$.
If in addition $\De$ is compact and $2$-dimensional,
it is easy to understand geometrically when the qW invariant $qW(u)$ does not vanish.
  For a typical point with $qW(u)\ne 0$,  the set of facets that are closest to it has at least three elements.
 The exception is when there are two closest facets that are parallel.
 Thus the set of
  points with nontrivial invariants 
is the union of a finite  set together with at most one line segment.

By Remarks~\ref{r:nohelp} and \ref{r:extprob},  one 
cannot lengthen a probe $P$ by deflecting it 
by a  probe $Q$ that starts from the facet at which $P$ exits $\De$.  
Thus extended probes do not displace extra points in the sectors $\De_{n,m}$, 
since these have only two edges.   It is also not hard to check that they also do 
not help in triangles such as $\bP(1,p,q)$, though they do help with the blow up 
$\De_{n,m}(\ka)$ of $\De_{n,m}$.

%%%%%%%%%%%%%%%%%%%%%%%%%%%%%
%%%%%%%%%%%%%%%%%%%%%%%%%%%%%
\subsubsection{Organization of the paper}

After a brief introduction to affine geometry, we describe probes and symmetric extended probes.  Theorem~\ref{t:noflag} explains which points can be displaced by these new 
probes, and the section ends by illustrating their use  in the Hirzebruch surfaces $F_n$. 
Theorem~\ref{t:pdeflected} describes the points that can be 
displaced by parallel  extended probes.  
The rest of  \S\ref{s:3}   illustrates how to use
this result
to understand the displaceability of  points in the open sectors
$\De_{n,m}^U(\ka)$.   Next, in \S\ref{s:DinTOR}  we describe the qW invariants, and use them to prove the results stated above about  closed sectors $\De_{n,m}$, weighted projective planes $\bP(1,p,q)$,  and their resolutions.  Finally in \S\ref{s:EPwFNP} we describe  extended
probes with trapezoidal flags, and use them in an open polytope of finite area.
 The last section \S\ref{s:proofs} contains the proofs of all the main theorems about probes.

\subsubsection{Acknowledgements}  We thank Andrew Fanoe, Yael Karshon, Egor Shelukhin, and Chris Woodward for useful discussions.
We would also like to thank the referee for their comments and corrections.
The first named author was partially supported by Funda\c{c}\~{a}o para a Ci\^{e}ncia e a Tecnologia
(FCT/Portugal), the second named author by NSF-grant DMS 1006610, and the third named author by 
NSF-grant DMS 0905191.

%%%%%%%%%%%%%%%%%%%%%%%%%%%%%%%%%%%%
%%%%%%%%%%%%%%%%%%%%%%%%%%%%%%%%%%%%
%%%%%%%%%%%%%%%%%%%%%%%%%%%%%%%%%%%%

%%%%%%%%%%%%%%%%%%%%%%%%%%%%%%%%%%%%
%%%%%%%%%%%%%%%%%%%%%%%%%%%%%%%%%%%%
%%%%%%%%%%%%%%%%%%%%%%%%%%%%%%%%%%%%
\section{Symmetric extended probes}\label{s:two}

%%%%%%%%%%%%%%%%%%%%%%%%%%%%%
%%%%%%%%%%%%%%%%%%%%%%%%%%%%%
\subsection{Moment polytopes and integral affine geometry}

Let $(M^{2n}, \om, \T)$ be a toric symplectic manifold with 
moment map $\Phi\co M \to \ft^{*}$, where
$\ft^{*}$ is the dual of the Lie algebra $\ft$ of the torus $\T$. 
We will identify $\ft$ together with its 
integer lattice $\ft_\Z$ with $(\R^n,\Z^n)$, and, using the natural pairing $\langle\cdot,\cdot \rangle$,
will also identify the pair $(\ft^*,\ft^*_\Z)$ with $(\R^n,\Z^n)$.  Thus we write
$$
\ft\times \ft^*: = \R^n\times \R^n\to \R,\quad (\eta,x)
\mapsto \langle\eta,x\rangle.
$$
For clarity, we use Greek letters for elements in $\ft$ and Latin letters for elements of $\ft^*$.

A polytope $\De \subset \ft^*\equiv\R^{n}$ is \emph{rational} if it is the
finite
intersection of half-spaces 
$$ \De = \bigcap_{i=1}^{N} \{x \in \R^{n} \mid \ip{\eta_{i}, x} + \ka_{i}\geq 0\}$$
where $\eta_{i} \in \Z^{n}$ are primitive vectors and are the interior conormals for the half-spaces.
A rational polytope $\De \subset \R^{n}$ is \emph{simple} if  
each codimension $k$ face of $\De$ meets exactly $k$ facets.
A rational simple polytope $\De \subset \R^{n}$ is \emph{smooth},
if at each codimension $k$ face, the $k$ conormal vectors for the facets meeting at the face can be extended to an integral basis of $\Z^{n}\equiv \ft_\Z$. 
Moment polytopes for symplectic toric manifolds are smooth.  Delzant \cite{D88} proved that the moment map gives a bijective correspondence between closed symplectic toric manifolds of dimension $2n$, up to equivariant symplectomorphism, and smooth 
compact 
polytopes in $\R^{n}$, up to integral affine equivalences by elements of $\R^{n} \rtimes GL_{n}(\Z)$.	

\begin{remark} \label{r:non-compact-orbifolds}
Lerman and Tolman \cite{LT97} generalized Delzant's classification result to closed symplectic toric orbifolds, 
showing that these are uniquely determined by the image of their moment map, a rational simple 
compact
polytope 
together with a positive integer \emph{label} attached to each of its facets. 
Equivalently, we can allow the conormal vectors $\eta_i$ to the facets to be nonprimitive, interpreting the label as the g.c.d.\ of their entries.
Using the convexity and connectedness results in \cite{LMTW98}, 
Karshon and Lerman \cite{KL09} have further extended this classification to the non-compact setting, provided one assumes that the moment map is \emph{proper} 
as a map $\Phi\co M^{2n} \to U \subset \R^{n}$, where $U$ is a convex open set.
In this case $\Phi(M^{2n})$ is an open polytope as in \cite[Definition~3.1]{WW11};
cf.\ the polytope $\De^U_{m,n}(\ka)$ in Figure~\ref{fig:X1} above.
\end{remark}

An affine hyperplane $A \subset \ft^*\equiv \R^{n}$ 
is \emph{rational} if it has a primitive conormal vector $\eta \in \Z^{n}$, in
which case $A = \{x \in \R^{n} \mid \ip{\eta, x} + \ka = 0 \}$ for some $\ka \in \R$.
The \emph{affine distance} between a rational affine hyperplane $A$, as above, and a point $x \in \R^{n}$ is
$$
	d_{\aff}(x, A) := \abs{\ip{\eta, x} + \ka}.
$$
For a rational hyperplane $A$, as above, a vector $v \in \R^{n}$ is \emph{parallel} to $A$
if $\ip{\eta, v} = 0$ where $\eta$ is 
the defining conormal for $A$.
An integral vector $v$
is \emph{integrally transverse} to $A$ if there is an integral basis of $\Z^{n}\subset\ft^*$ consisting of 
$v$ and vectors $w \in \Z^{n}$ parallel to $A$, or equivalently if
$\abs{\ip{\eta, v}} = 1$.
Note also that if
$v$ is integrally transverse to $A$, then there is
an integral affine equivalence of $\R^n$
taking $A$ and $v$ to 
$ \{x_{1} = 0\}$ and $ (1, 0, \dots, 0)$.

An affine line $L = z + \R v$ in $\R^{n}$ is \emph{rational} if the direction vector $v$ can be taken to be a primitive integral vector in $\Z^{n}$.  Given a rational line $L$ with 
primitive direction vector $v \in \Z^{n}$, the \emph{affine distance} $d_{\aff}(x,y)$ between two points $x,y \in L$ is defined by
$$
	d_{\aff}(x,y) := \abs{t} \quad\mbox{where $t \in \R$ is such that}\quad x-y = tv \in \R^{n}.
$$
For a primitive vector $v \in \Z^{n}$ and a rational hyperplane $A$, 
the \emph{affine distance along $v$} between a point 
$x \in \R^{n}$ and $A$ is defined as
$$
	d_{v}(x, A): = d_{\aff}(x,y) \quad\mbox{if $y \in A$ is on the rational ray $x + \R_{\geq 0}v$.}
$$
If $x + \R_{\geq 0}v$ does not meet $A$, then $d_{v}(x, A): = \infty$.  If $v_{0}$ is integrally transverse to $A$ and the ray $x + \R_{\geq 0}v_{0}$ meets $A$, then $d_{v_{0}}(x, A) = d_{\aff}(x, A)$,
but otherwise, somewhat paradoxically,  we have
$$
d_{v_{0}}(x, A) < d_{\aff}(x, A).
$$
  For example, if $A = \{x_{1} = 0\}$, $x = (1, 0, \dots, 0)$ and $v_0 = (-2,3)$, then $d_{\aff}(x, A)=1$ while $d_{v_{0}}(x, A)=\frac 12$.
%%%%%%%%%%%%%%%%%%%%%%%%%%%%%
%%%%%%%%%%%%%%%%%%%%%%%%%%%%%

%%%%%%%%%%%%%%%%%%%%%%%%%%%%%
%%%%%%%%%%%%%%%%%%%%%%%%%%%%%
\subsection{Probes and extended probes}
We now recall the definition of a probe in a rational polytope from \cite{Mc11} and state the method
of probes.

\begin{defn}
	A { \bf probe} $P$ in a rational polytope $\De \subset \ft^{*} \equiv \R^{n}$, is a 
	directed rational line 
	segment contained
	in $\De$ whose initial point $b_{P}$ lies
	in the interior of a facet $F_{P}$ of $\De$ 
	and whose  {\bf direction vector} $v_{P} \in \ft^{*}_{\Z} \equiv \Z^{n}$
	 is primitive and integrally transverse to the {\bf base facet} $F_{P}$.
	If $e_{P}$ is the endpoint of $P$, then the {\bf length}
	 $\ell(P)$ of $P$ is defined as the affine
	distance $d_{\aff}(e_{P}, b_{P})$.
\end{defn}

Note that if $\eta_{F_{P}}$ is the interior conormal for the facet $F_{P}$, then the definition requires
that $\ip{\eta_{F_{P}}, v_{P}} = 1$ since $v_{P}$ needs to be integrally transverse and be
inward pointing into $\De$.

\begin{lem}[{\cite[Lemma 2.4]{Mc11}}]\label{l:probe}
	Let $P$ be a probe in a moment polytope $\De$ %= \Phi(M)$ 
	for a toric symplectic orbifold 
	$(M^{2n}, \w, \Phi)$.
	If a point $u$ on the probe $P$ is less than halfway along $P$, meaning that
	$d_{\aff}(u, F_{P}) < \tfrac{1}{2}\ell(P)$, then the Lagrangian fiber $L_{u} = \Phi^{-1}(u)$ 
	is displaceable.
\end{lem}

Since the inverse image of $P$ by the moment map is symplectomorphic to the product of a disc $\bD$ of area $\ell(P)$ with $\T^{n-1}$, it is easy to construct 
a proof of this lemma from the remarks after 
equation~\eqref{e:FE} below.
To generalize this method, let us first introduce the notion of a symmetric probe and the associated
reflections.

\begin{defn}
	A probe $Q$ in a rational polytope $\De \subset \ft^{*} \equiv \R^{n}$ is
	{\bf symmetric} if the endpoint $e_{Q}$ lies on the interior of
	a facet $F_{Q}'$ that is integrally transverse to $v_{Q}$.  
\end{defn}

Associated to a symmetric probe $Q$ is an affine reflection $A_{Q}\co\ft^{*} \to \ft^{*}$
	\begin{equation}\label{e:Areflect}
		A_{Q}(x) =x + \ip{\eta_{F'_{Q}}-\eta_{F_{Q}}, x} v_{Q} + (\kappa' - \kappa)v_{Q}
	\end{equation}
that swaps the two facets
$$
	F_{Q} = \{x \in \ft^{*} \mid \ip{\eta_{F_{Q}}, x} + \kappa = 0\}
	\quad\mbox{and}\quad
	F_{Q}' = \{x \in \ft^{*} \mid \ip{\eta_{F_{Q}'}, x} + \kappa' = 0\},
$$
since $\ip{\eta_{F_{Q}}, v_{Q}} = 1$ and $\ip{\eta_{F_{Q}'}, v_{Q}} = -1$.	
Let $\widehat{A}_{Q}\co\ft^{*} \to \ft^{*}$ be the
associated linear reflection 
\begin{equation}\label{e:reflect}
	\widehat{A}_{Q}(x) = x + \ip{\eta_{F_{Q}'} - \eta_{F_{Q}}, x} v_{Q}.
\end{equation}
Observe that $A_{Q}$ preserves $Q$ set-wise, $\widehat{A}_{Q}(v_{Q}) = -v_{Q}$,
and $\widehat{A}_{Q}^{*}(\eta_{F_{Q}}) = \eta_{F_{Q}'}$.

Note that $A_{Q}$ need not preserve the polytope $\De$, but it does preserve a small neighborhood of
$Q$.  Our first generalization of the method of probes 
deflects a given probe by a symmetric probe.

\begin{defn}	
	Let $Q$ be a symmetric probe in a rational polytope $\De \subset \ft^{*} \equiv \R^{n}$,
	and let $P$ be another probe with direction $v_{P}$ such that $P$ ends at the point 
	$x_{PQ}$ in the interior of $Q$.  
	A {\bf symmetric extended probe} $\cSP$
	formed by deflecting $P$ with $Q$ is a union
	$$
		\cSP = P \cup Q \cup P' \subset \De,
	$$
	where the extension $P'$ is a rational line segment in $\De$ with direction $v_{P'}$ starting at
	 $x_{PQ}'$,
	where
	$$
		v_{P'}: = \widehat{A}_{Q}(v_{P}), \quad
		x_{PQ}' := A_{Q}(x_{PQ}) .  
	$$
	We define the {\bf  length} of the extended probe $\cSP$ to be
	$\ell(\cSP): = \ell(P) + \ell(P')$. Thus the {\bf endpoint} $e_{P'}$ of $P'$ is
	$e_{P'} := x'_{PQ} + \ell(P')\,v_{P'}$. 
\end{defn}

\begin{remark}
	A visual description of $x_{PQ}'$ is that it is the unique point
	on $Q$ so that 
	$$
	d_{\aff}(x_{PQ}, F_{Q}) = d_{\aff}(x_{PQ}', F_{Q}');
	$$
	cf.\ Figure~\ref{f:epNF}.
	Note that if $v_{P}$ is parallel to $F_{Q}$, then 
	$\widehat{A}_{Q}(v_{P})$ is parallel to $F_{Q}'$, being the projection of $v_{P}$ to
	the linear hyperplane $\{x \in \ft^{*} \mid \ip{\eta_{F_{Q}'}, x} = 0\}$ along $-v_{Q}$.
\end{remark}

\begin{figure}[h]

	\begin{center} 
	\leavevmode 
	\includegraphics[width=4in]{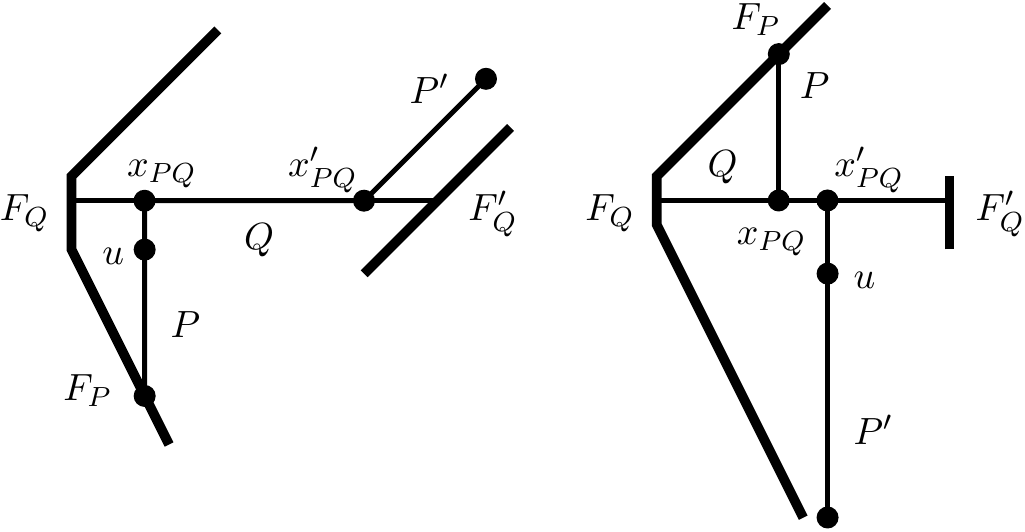}
	\end{center} 

	\caption{Two ways of using Theorem~\ref{t:noflag}.  
	Left: $u$ is on $P$.  Right: $u$ is on $P'$.}
	\label{f:epNF}
\end{figure}

\begin{thm}\label{t:noflag} 
Let $\cSP$ be a symmetric extended probe formed by deflecting the probe $P$
with the symmetric probe $Q$, in a moment polytope $\De = \Phi(M)$ for a toric symplectic 
orbifold $(M^{2n}, \w, \T, \Phi)$.  For a point $u$ in the moment polytope $\De$: 
\begin{itemize}
\item If $u$ is in the interior of $P$ and $d_{\aff}(u, F_{P}) < \tfrac{1}{2}\,\ell(\cSP)$, or 
\item If $u$ is in the interior of $P'$ and $\ell(P) + d_{\aff}(x'_{PQ}, u) < \tfrac{1}{2}\,\ell(\cSP)$,
\end{itemize}
then the Lagrangian torus fiber $L_{u} = \Phi^{-1}(u)$ is displaceable in $(M, \w)$.
\end{thm}

See Section~\ref{s:proofs} for the proof. 
The idea is to join $P$ to $P'$ using a symplectomorphism of   
$\Phi^{-1}({\rm nbhd}\, Q)\subset M$ that equals the identity on one boundary component and the lift of the reflection $A_Q$ on the other.

\begin{figure}[h]

	\begin{center} 
	\leavevmode 
	\includegraphics[width=2.5in]{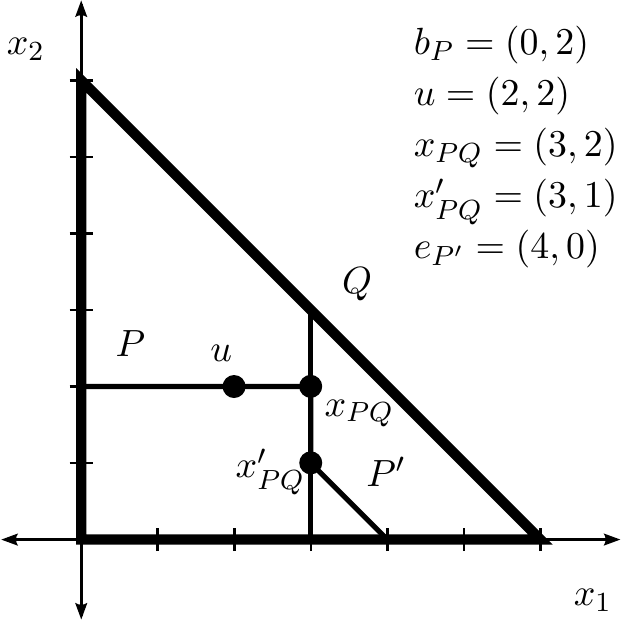}
	\end{center} 

	\caption{Illustration of Remark~\ref{r:nohelp}.  If $P$ were extended past
	$x_{PQ}$ it would exit the polytope at the point $(4,2)$, so its length is at most $4$.
	This is not enough to displace $u$ since
	$d_{\aff}(b_{P}, u) = 2$.  Using $Q$ as a deflecting probe does not help;
	the length of the resulting extended probe is equal to $4$.  Of course
	the fiber $L_{u}$ is nondisplaceable since it is the Clifford torus in $\CP^{2}$.}
	\label{f:nohelp}
	\end{figure}

\begin{remark}\label{r:nohelp}
	Let $\cSP = P \cup Q \cup P'$ be a symmetric extended probe.
	Because $e_{P'}$ must stay in $\De$, which is convex, it follows that
	$\ell(P') \leq d_{v_{P'}}(x_{PQ}', F_{Q}')$.  Also because $A_{Q}$ is an integral affine equivalence in 
	$\ft^{*} \rtimes GL(\ft^{*}_{\Z})$ it follows that
	$d_{v_{P}}(x_{PQ}, F_{Q}) = d_{v_{P'}}(x_{PQ}', F_{Q}').$
	Combining these two we see that
	$$
		\ell(\cSP)  \leq \ell(P) + d_{v_{P}}(x_{PQ}, F_{Q}) = d_{v_{P}}(b_{P}, F_{Q})
	$$
	the length of the extended probe $\ell(\cSP)$ is less than $d_{v_{P}}(b_{P}, F_{Q})$,
	which is the maximum length the probe $P$ can have before it hits the affine hyperplane $F_{Q}$.
	In particular, this means that one cannot make a probe $P$ displace more points by 
	deflecting it with a symmetric probe $Q$ that is based at the facet on which $P$ would exit 
	$\De$.  See Figure~\ref{f:nohelp}.
\end{remark}

In the next section we will generalize the notion of extended probes to the case where $Q$ is not a symmetric probe.  But before that let us first explain how
Theorem~\ref{t:noflag} suffices to settle the question
of which Lagrangian fibers in Hirzebruch surfaces are displaceable.
%%%%%%%%%%%%%%%%%%%%%%%%%%%%%
%%%%%%%%%%%%%%%%%%%%%%%%%%%%%

%%%%%%%%%%%%%%%%%%%%%%%%%%%%%
%%%%%%%%%%%%%%%%%%%%%%%%%%%%%
\subsection{Hirzebruch surfaces}\label{s:H}

Let $m \geq 0$ be an integer and $\ka$ a real number so that $\ka > m$.  
The moment polytope for the $m$th Hirzebruch surface is
$$ 
		\De_{1,m}(\ka) = \Big\{(x_{1}, x_{2}) \in \R^{2} \mid x_{1} \geq 0\,,\,\, x_{2} \geq 0\,,\,\,
		-x_{2}+2 \geq 0\,,\,\, -x_{1} - mx_{2}  +\ka + m \geq 0\Big\}.
$$

\begin{figure}[h]

\begin{center} 
\leavevmode 
\includegraphics[width=4.5in]{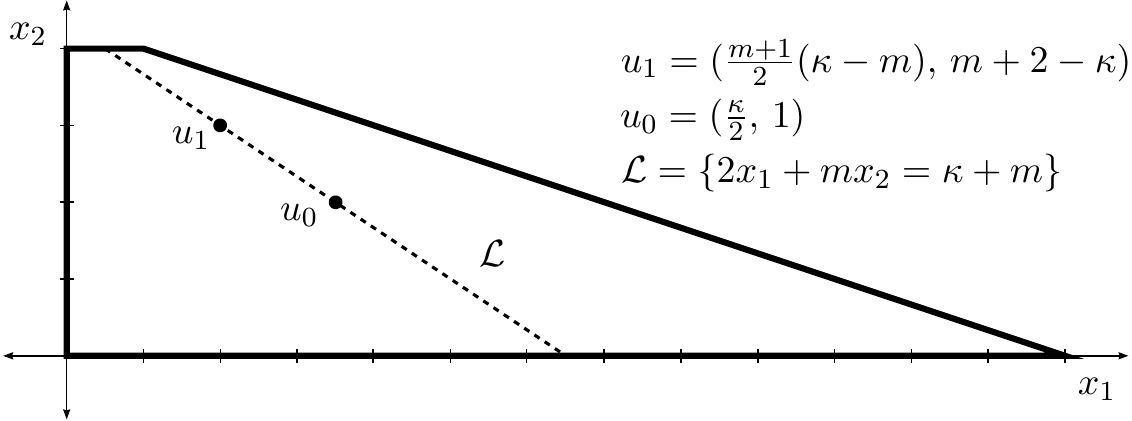}
\end{center} 
\caption{The moment polytope $\Delta_{1,m}(\ka)$ for $m=3$ and $\ka = \tfrac{7}{2}$.  Points
not on the line $\cL$ 
can be displaced by horizontal probes.  
When $m$ is even, probes
with directions $(\mp\tfrac{m}{2}, \pm1)$ based on the two horizontal facets displace everything on
$\cL$ except $u_{0}$. }
\label{f:H3}
\end{figure}
	
When $m\not=1$, the only known nondisplaceable fiber is $u_{0} = (\tfrac{\ka}{2}, 1)$
\cite{AbM11, FOOO10, Wd11} and when $m$ is even, standard probes displace every other fiber.
However for odd $m$ and $\ka < m+1$, \cite[Lemma 4.1]{Mc11} proves that the point 
$$
u_{1} = \big(\tfrac{m+1}{2}(\ka - m),\, m+2-\ka\big),
$$
which is different from $u_{0}$, cannot be displaced by probes,
so in these cases probes do not perfectly complement
the known nondisplaceablility results.  In fact, in these cases one cannot
use probes to displace any of the points on a line segment that starts at 
$u_{1}$ and runs through and a bit past $u_{0}$.  
We will now show that these unknown fibers can all be displaced using Theorem~\ref{t:noflag}.

\begin{prop}\label{p:H}
	When $m\not=1$, then every fiber in the polytope $\De_{m}(\ka)$ of
	the $m$-th Hirzebruch surface is displaceable
	except for
	$$
	u_{0} = (\tfrac{\ka}{2}, 1) \in \De_{m}(\ka).
	$$
	Hence $u_{0}$ is a stem and therefore is nondisplaceable.
\end{prop}	
Recall that a fiber of a moment polytope is called a \emph{stem} if every other fiber is displaceable.  It is proven in \cite[Theorem 2.1]{EP06} using the theory of quasi-states that every stem is nondisplaceable.

Proposition~\ref{p:H} stands in contrast to the well-studied case of $m = 1$, which corresponds to a toric blow-up of $\CP^{2}$, where standard probes do complement the known nondisplaceability results. 
Methods in \cite{AbM11,Bo11,Ch08,FOOO10,Wd11} prove that when $1 < \ka < 2$,
the fibers 
$$u_{0} = (\tfrac{\ka}{2}, 1) \quad \mbox{and}\quad u_{1} = (\ka -1, 3-\ka)$$ 
in $\De_{1}(\ka)$ are nondisplaceable, and when $\ka \geq 2$ only the fiber $u_{0}$ in $\De_{1}(\ka)$ is nondisplaceable.

\begin{proof}[Proof of Proposition~\ref{p:H}]
	The vector $(1, 0)$ is integrally transverse to %the facet 
	$\{x_{1} = 0\}$
	and $(-1,0)$ is integrally transverse to $\{x_{1} + mx_{2} = \ka + m\}$.
	It is easy to check that probes in these directions  displace every point $x \in \Int \De_{m}(\ka)$ 
	not on the median 
	$$
		\cL = \{2x_{1} + mx_{2} = \ka + m\}.
	$$
	When $m$ is even, one can use a probe with direction $(\tfrac{m}{2}, -1)$
	based on the facet $\{x_{2} = 2\}$ and a probe with direction 
	$(-\tfrac{m}{2}, 1)$ based on the facet $\{x_{2} = 0\}$ to displace all the
	points on $\cL$ except $u_{0}$.  
	
	When $m \geq 3$ is odd, we will use symmetric extended probes to show
	that every point on $\cL$, except $u_{0}$, is displaceable.
	Such points $w$ can be written as 
	\begin{equation}\label{e:w}
		w = \big(\tfrac{1}{2}(\ka + m(1-w_{2})),\, w_{2}\big),\quad w_2\ne 1
	\end{equation}
	and we will divide this into two cases:
	$1< w_{2} < 2$ and $0 < w_{2} < 1$.
	In both cases we will use the symmetric deflecting probe $Q$ where
	$$ 
	b_{Q} = (\tfrac{1}{2}(\ka - m),\, 2)\,, \quad v_{Q} = (1, -1)\,, \quad
	F_{Q} = \{x_{2} = 2\}\,, \quad \mbox{and} \quad F_{Q}' = \{x_{2} = 0\}.
	$$
	We will use probes with direction $v_{P} = (\pm 1, 0)$, which
	are parallel to both $F_{Q}$ and $F_{Q}'$, so it follows that $v_{P'} = \widehat{A}_{Q}(v_{P}) = v_{P}$.
	
	For $1 < w_{2} < 2$, cf. Figure~\ref{f:H3ep}, let $P$ be the probe with
	$$ 
	b_{P} = \big(\ka + m(1-w_{2}),\, w_{2}\big)\,, \quad 
	v_{P} = (-1, 0)\,, \quad\mbox{and}\quad 
	F_{P} = \{x_{1} + mx_{2} = \ka +m\},
	$$
	and let $\cSP = P \cup Q \cup P'$ be the associated extended probe, where
	\begin{align*}
		x_{PQ} &= \big(\tfrac{1}{2}(\ka -m) + 2-w_{2},\, w_{2}\big) &&
		&\ell(P) &= \tfrac{1}{2}(\ka+3m)+ (1-m)w_{2}-2\\
		x'_{PQ} &=  \big(\tfrac{1}{2}(\ka-m) + w_{2},\,2-w_{2}\big)&&
		&\ell(P')&= \tfrac{1}{2}(\ka-m) + w_{2}\\
		e_{P'} &= (0, 2-w_{2}) && 
		&\ell(\cSP)& = \ka + (m-2)(1-w_{2}).
	\end{align*}
	
	\begin{figure}[h]

	\begin{center} 
	\leavevmode 
	\includegraphics[width=4.5in]{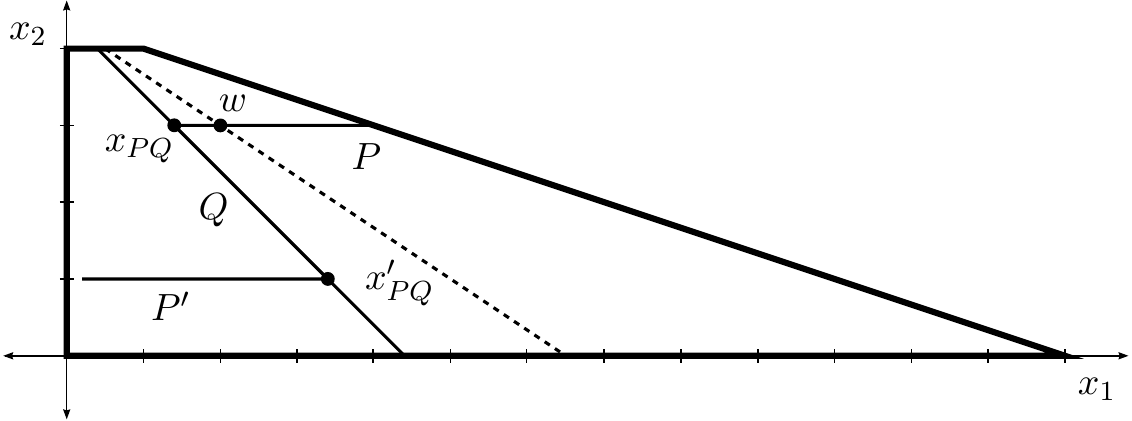}
	\end{center} 

	\caption{The symmetric extended probe $\cSP$ that displaces
	$w$ when $1 < w_{2} < 2$.}
	\label{f:H3ep}
	\end{figure}
	
	The point $w$ from \eqref{e:w} lies on $P$ and therefore using $1 < w_{2}$ we have
	\begin{align*}
		d_{\aff}(w, F_{P}) &= \tfrac{1}{2}\Big(\ka + m(1-w_{2})\Big) < \tfrac{1}{2}\,\ell(\cSP).
	\end{align*}
	Hence $L_w$ is displaceable by Theorem~\ref{t:noflag}.
	
	For $0 < w_{2} < 1$, cf. Figure~\ref{f:H3epNF}, let $P$ be the probe with
	$$
	b_{P} = \big(0,\, 2 - w_{2}\big)\,, \quad v_{P} = (1,0)\,, \quad\mbox{and}\quad
	F_{P} = \{x_{1} = 0\}
	$$
	and let $\cSP = P \cup Q \cup P'$ be the associated extended probe, where
	\begin{align*}
		x_{PQ} &= \big(\tfrac{1}{2}(\ka -m) + w_{2},\, 2-w_{2}\big) && 
		&\ell(P) &= \tfrac{1}{2}(\ka - m) + w_{2}\\
		x'_{PQ} &= \big(\tfrac{1}{2}(\ka-m) + 2-w_{2},\,w_{2}\big) &&
		&\ell(P') &= \tfrac{1}{2}(\ka +3m) + (1-m)w_{2} - 2\\
		e_{P'} &= \big(\ka + m(1-w_{2}),\,w_{2}\big) &&
		&\ell(\cSP) &= \ka + m(1-w_{2}) + 2(w_{2} -1)
	\end{align*}
	and the point $w$ is on $P'$.  
	\begin{figure}[h]

	\begin{center} 
	\leavevmode 
	\includegraphics[width=4.5in]{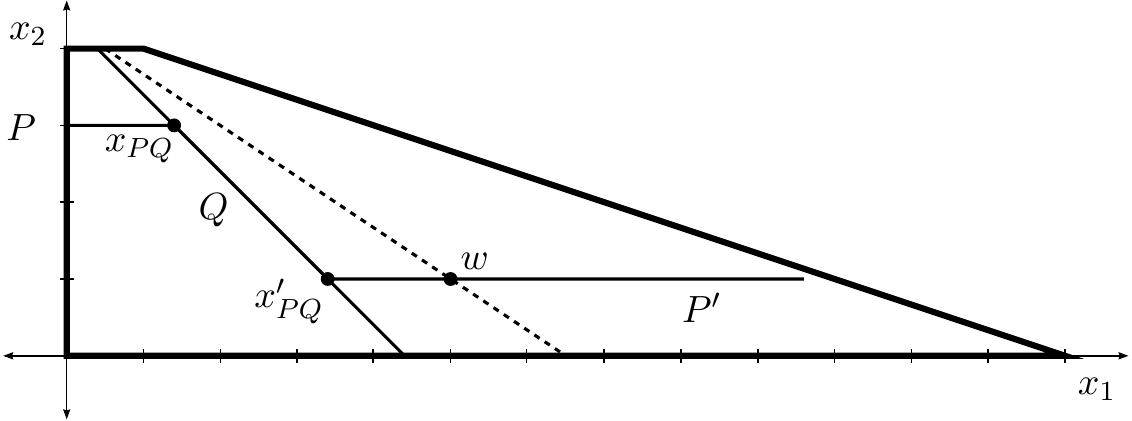}
	\end{center} 

	\caption{The symmetric extended probe $\cSP$ that displaces $w$ when $0 < w_{2} < 1$.}
	\label{f:H3epNF}
	\end{figure}
	
	Since $w_{2} < 1$ we have
	\begin{align*}
		\ell(s_{1}) + d_{\aff}(x'_{PQ}, w) &= \tfrac{1}{2}\Big(\ka + m(1-w_{2})\Big) + 2(w_{2} -1)
		< \tfrac{1}{2} \,\ell(\cSP),
	\end{align*}
	and therefore Theorem~\ref{t:noflag} implies that  $L_{w}$ is displaceable.
\end{proof}
%%%%%%%%%%%%%%%%%%%%%%%%%%%%%
%%%%%%%%%%%%%%%%%%%%%%%%%%%%%

%%%%%%%%%%%%%%%%%%%%%%%%%%%%%%%%%%%%
%%%%%%%%%%%%%%%%%%%%%%%%%%%%%%%%%%%%
%%%%%%%%%%%%%%%%%%%%%%%%%%%%%%%%%%%%

%%%%%%%%%%%%%%%%%%%%%%%%%%%%%%%%%%%%
%%%%%%%%%%%%%%%%%%%%%%%%%%%%%%%%%%%%
%%%%%%%%%%%%%%%%%%%%%%%%%%%%%%%%%%%%
\section{Extended probes with flags: parallel case}\label{s:3}

The use of symmetric extended probes is fairly restrictive since a symmetric probe $Q$ represents a torus bundle over $S^{2}$.  In cases where $Q$ does not exit the polytope (or does so non-transversally) then the following construction can be used with $Q$ to deflect probes.

%%%%%%%%%%%%%%%%%%%%%%%%%%%%%
%%%%%%%%%%%%%%%%%%%%%%%%%%%%%
\subsection{The definition and the displaceability method}

\begin{defn}\label{d:PEPwF}
Let $P$ and $Q$ be probes in a rational polytope $\De \subset \R^{n}$ where 
the probe $P$ ends at the point $x_{PQ}$ on $Q$, and suppose that $v_{P}$ is parallel to the base facet 
$F_{Q}$  of $Q$.
The {\bf parallel extended probe with flag} $\cDP$ formed by
deflecting $P$ with $Q$ is the subset
$$
	\cDP = P \cup Q \cup \cF\  \subset\ \De.
$$
Here the flag $\cF$ is the convex hull of the points $\{x_{\cF},\, x_{\cF}',\,e_{\cF},\, e'_{\cF}\}$, where
$x_{\cF}$ and $x_{\cF}'$ lie on $Q$, and the vector
$
x_{\cF} -e_{\cF}  =  x'_{\cF} -e'_{\cF} 
$
is parallel to $P$ (see Figure~\ref{f:extendedprobe}).  
The {\bf length} of the flag $\ell(\cF)$ is $d_{\aff}(x_{\cF} , e_{\cF})$
so that 
\begin{equation}\label{e:PFL}
x_{\cF} -e_{\cF}  =  x'_{\cF} -e'_{\cF} = \ell(\cF)\, v_{P}
\end{equation}
The {\bf  length} of the extended probe 
is $\ell(\cDP) = \ell(P) + \ell(\cF)$.
\end{defn}

\begin{figure}[h]
	
	\begin{center} 
	\leavevmode 
	\includegraphics[width=2.5in]{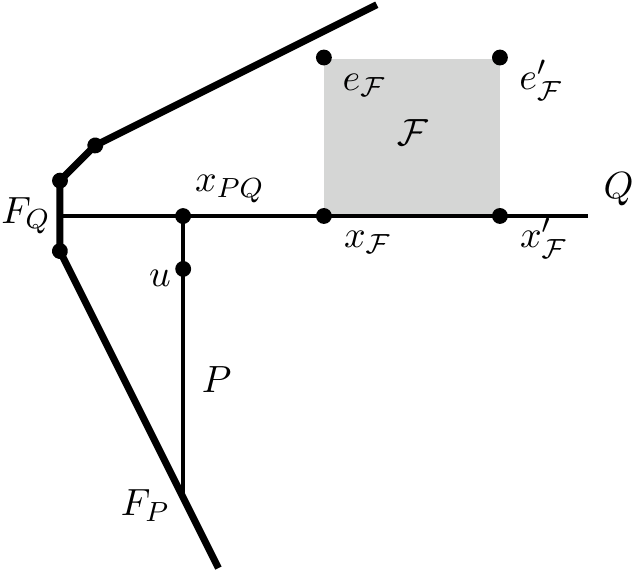}
	\end{center} 

	\caption{An extended probe with flag $\cDP = P \cup Q \cup \cF$.}
	\label{f:extendedprobe}
\end{figure}

The following theorem explains how one can use parallel extended probes  to displace Lagrangian torus fibers. 
 
\begin{thm}\label{t:pdeflected}
	Let $\cDP = P \cup Q \cup \cF \subset \De$ be a parallel extended probe
	in the moment polytope $\De = \Phi(M)$ of the toric symplectic orbifold
	$(M^{2n}, \w, \T, \Phi)$, and let $u\in \INT (P) \subset \INT(\De)$.  
	Then the Lagrangian fiber $L_{u} = \Phi^{-1}(u)$ is displaceable if the following conditions both hold:
	\begin{itemize} \item the affine distance from $u$ to the facet $F_{P}$
	satisfies
	\begin{equation}\label{e:p1half}
		d_{\aff}(u, F_{P}) < \tfrac{1}{2}\,\ell(\cDP)\,,
	\end{equation}
\item and the flag $\cF$ satisfies the inequality
	\begin{equation}\label{e:pflag}
	d_{\aff}(x_{PQ}, F_{Q}) < d_{\aff}(x_{\cF}, x_{\cF}')\,.
	\end{equation}
	\end{itemize}
	\end{thm} 

In Section~\ref{s:EPwFNP} we will further generalize this to the case where $P$ is not parallel to the facet $F_{Q}$. 
See Section~\ref{s:proofs} for the proofs.

%%%%%%%%%%%%%%%%%%%%%%%%%%%%%
%%%%%%%%%%%%%%%%%%%%%%%%%%%%%

%%%%%%%%%%%%%%%%%%%%%%%%%%%%%
%%%%%%%%%%%%%%%%%%%%%%%%%%%%%
\subsection{Example: displaceability in the open region $U_{n,m}(\ka)$}\label{s:DinC2}

Consider the standard toric structure $(\C^{N}, \w_{0}, (S^{1})^{N},\Phi_{0})$.  
The symplectic form is $\w_{0} = \tfrac{1}{\pi}\, \sum_{k=1}^{N} dx_{k} \wedge dy_{k}$,
the torus action is
$$
	(t_{1}, \dots, t_{N}) \cdot (z_{1}, \dots, z_{N}) = (e^{2\pi i t_{1}}z_{1}, \dots, e^{2\pi i t_{N}}z_{N})
$$
where $t_{k} \in S^{1} = \R/\Z$, the moment map is
$$
	\Phi_{0}\co\C^{N} \to \R^{N} \quad\mbox{where}\quad
	\Phi_{0}(z_{1}, \dots, z_{N}) = (\abs{z_{1}}^{2}, \dots, \abs{z_{N}}^{2})
$$
and the moment polytope is $\R^{N}_{+} \subset \R^{N}$.

For ease of notation let us now specialize to the case $\C^{2}$.  Let $m > n \geq 1$ be relatively prime integers
and consider the open subset of $\C^{2}$ 
$$
	U_{n,m}(\ka) = \big\{z = (z_{1}, z_{2}) \in \C^{2} : -n\abs{z_{1}}^{2} + m\abs{z_{2}}^{2} + \ka > 0\big\}
$$
where $\ka > 0$ is any positive constant.  
Its image under the moment map is
$$
	\De^{U}_{n,m} := \Phi_{0}(U_{n,m}) = \big\{x = (x_{1}, x_{2}) \in \R^{2} : x_{1} \geq 0\,,\,\, 
	x_{2} \geq 0\,,\,\, -n x_{1} + mx_{2} + \ka > 0\big\}.	
$$
We will now turn to the investigation of 
 the displaceability of Lagrangian toric fibers in
$(U_{n,m}(\ka), \w_{0}, \De^{U}_{n,m}(\ka))$.

\begin{figure}[h]

	\begin{center} 
	\leavevmode 
	\includegraphics[width=2.2in]{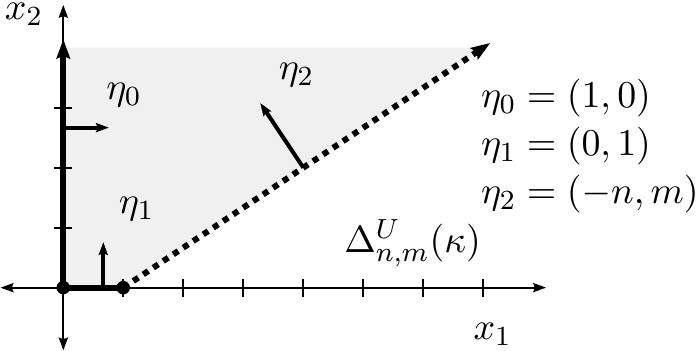}	
	\end{center} 

	\caption{The moment polytope $\De_{n,m}^{U}(\ka)$ for $(n,m) = (2,3)$
	and $\ka = 2$.}
	\label{f:Xnm}
\end{figure}

Let us first explain what  is displaceable in $(U_{n,m}, \w_{0}, \De^{U}_{n,m})$
 by standard probes.

\begin{lem}\label{l:DOsecP}
	The following points $x \in \De^{U}_{n,m}$ can be displaced by probes
	based on the facet $\{x_{1} = 0\} \subset \De_{n,m}^{U}$:
	\begin{itemize}
	\item 
	points in $\{x_{1} < x_{2}\}$ by probes with direction $(1,1)$.			
	\item 
	points in $\{x_{1} < \tfrac{m}{2n}\,x_{2} + \tfrac{\ka}{2n}\}$ by probes with 
	direction $(1,0)$.
	\end{itemize}
	The following points $x \in \De^{U}_{n,m}$ can be displaced by probes based
	on %the facet 
	$\{x_{2} = 0\} \subset \De_{n,m}^{U}$:
	$$
		\{cx_{2} < x_{1} < cx_{2} + \ka/n\} \mbox{ by probes with direction $(c, 1)$ for 
		$c=0, 1, \dots, \lceil m/n \rceil-1$}.
	$$
\end{lem}
\begin{proof} An elementary calculation.
\end{proof}

\begin{figure}[h]
	
	\begin{center} 
	\leavevmode 
	\includegraphics[width=4.5in]{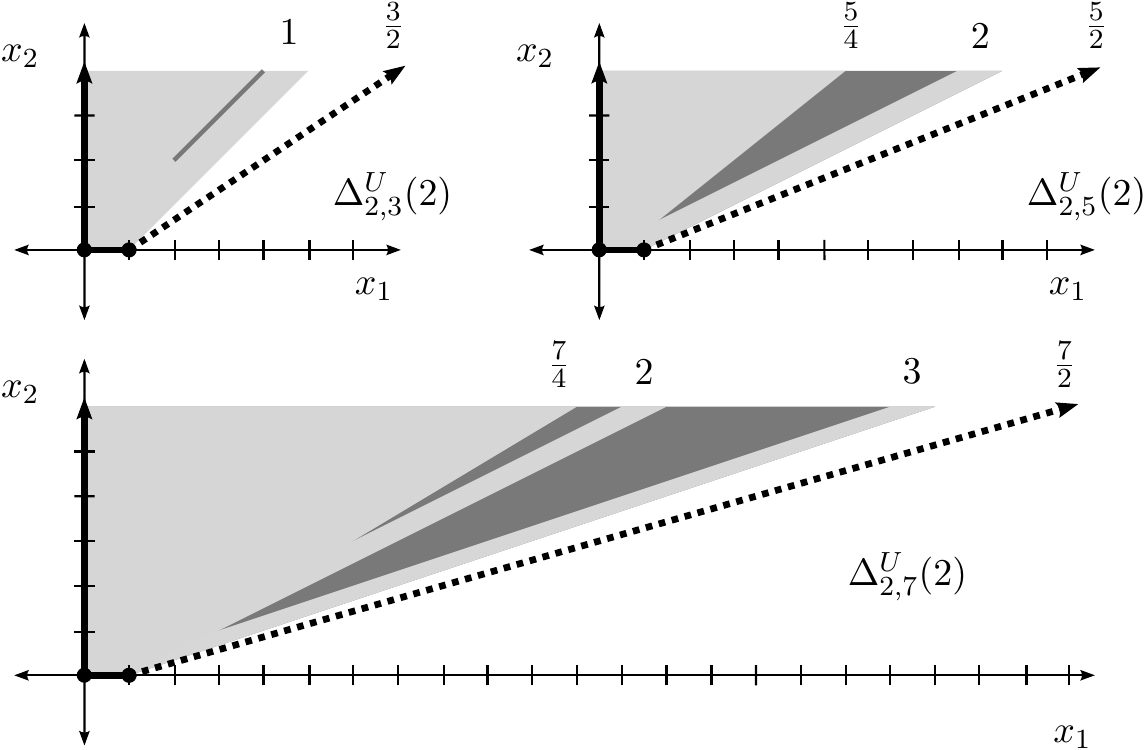}
	\end{center} 

	\caption{Comparing Lemma~\ref{l:DOsecP} and Lemma~\ref{l:openEP}.
	The points in the gray regions are displaceable, but extended probes
	are needed for points in the dark gray regions.  See Remark~\ref{r:EPvsP}.
	%Strom
	In this figure and others the numbers $1, \tfrac{3}{2}, \tfrac{5}{4}, 2, \ldots$ refer to the
	slope $\lambda$ for the line $x_{1} = \lambda x_{2} + b$ defining the boundary between differently shaded regions.
	}
	\label{f:DXnm}
\end{figure}

Here is what we can do with extended probes. 

\begin{lem}\label{l:openEP}
	Let $x= (x_{1}, x_{2}) \in \De^{U}_{n,m}$.  If
	$x$ is to the left of the line passing through $(\ka/n, 0)$ with slope $1/(\lceil m/n \rceil - 1)$,
	that is 
	\begin{equation}\label{e:knd}
		x_{1} < (\lceil m/n \rceil - 1)\,x_{2} + \ka/n\,,
	\end{equation} 
	then $x$ can be displaced by a parallel extended probe with flag in $\De^{U}_{n,m}$.
\end{lem}
\begin{proof}
	Let $d = \lceil m/n \rceil - 1$ and let $w = (w_{1}, w_{2}) \in \De^{U}_{n,m}$ be a point 
	satisfying \eqref{e:knd}.
	Consider the probes $P$ and $Q$ where,
	for some small $\eps>0$,
	$$
		b_{P} = (0, w_{2})\,,\,\, v_{P} = (1,0) \quad\mbox{and}\quad b_{Q} = (\tfrac{\ka}{n}-\eps, 0)\,,\,\, 
		v_{Q} = (d,1),
	$$
	and $P$ ends at the point $x_{PQ}$ on $Q$.  
	Observe that $Q$ 
	is parallel to the line defined by an equality sign in \eqref{e:knd}, 
	so since $w$ satisfies \eqref{e:knd} it follows that $w$ lies in the interior of $P$ for sufficiently 
	small $\eps$.

	Note that $v_{P}$ is parallel to the base facet $F_{Q} = \{x_{2} = 0\}$ of $Q$.
	
	\begin{figure}[h]
	
	\begin{center} 
	\leavevmode 
	\includegraphics[width=4.5in]{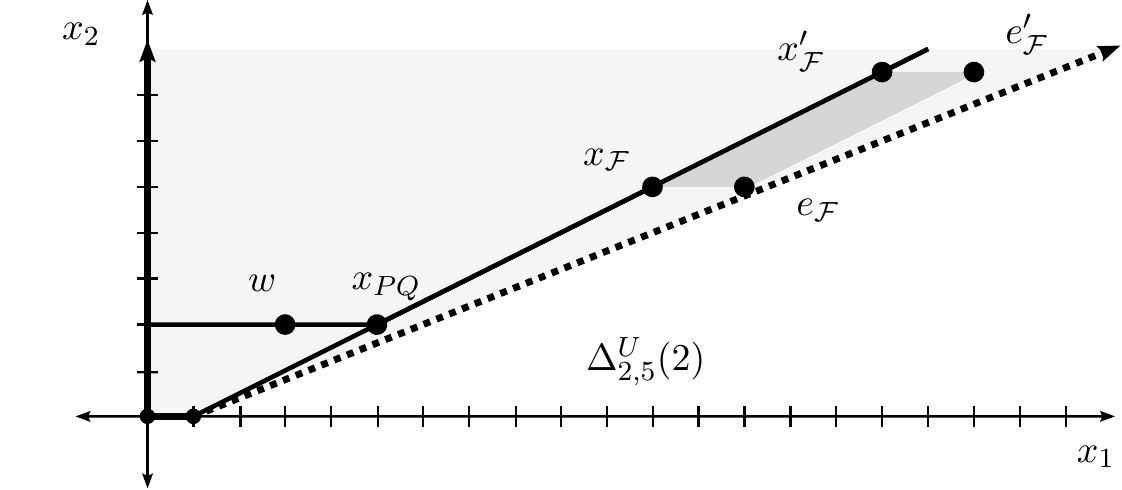}
	\end{center} 

	\caption{An extended probe with flag used in the proof of Lemma~\ref{l:openEP}.}
	\label{f:EPinC2}
	\end{figure}
	 
	For the three parameters $\a, \a', \ell_{\cF} > 0$, consider
	the parallel extended probe with flag $\cDP = P \cup Q \cup \cF$ where
	\begin{align*}
		x_{PQ} &= b_{Q} + w_{2}(d,1) && x_{\cF} = b_{Q} + \a(d,1) && x'_{\cF} = b_{Q} + \a'(d,1) \\
		\ell(P) &= \tfrac{\ka}{n} + w_{2}d && \ell(\cF) =  \ell_{\cF} &&
		\ell(\cDP) =  \tfrac{\ka}{n} + w_{2}d + \ell_{\cF}
	\end{align*}
	See Figure~\ref{f:EPinC2} for an example.
	First take $\ell_{\cF}$ sufficiently
	large so that \eqref{e:p1half} is satisfied as
	$$
		d_{\aff}(w, b_{P}) = w_{1}  < \tfrac{1}{2}(\tfrac{\ka}{n} + w_{2}d + \ell_{\cF}) = \tfrac{1}{2}\, \ell(\cDP).
	$$ 
	Then take $\a, \a' > \frac{n}{m-nd}\, \ell_{\cF}$ to ensure that the endpoints of the flag
	$$
		e_{\cF} = b_{Q} + \a(d,1) + \ell_{\cF}(1,0)  \quad \mbox{and}\quad
		e_{\cF}' = b_{Q} + \a'(d,1) + \ell_{\cF}(1,0) = (\tfrac{\ka}{n} + \a' d + \ell_{\cF}, \a')
	$$
	stay in $\De_{n,m}^{U}$.  Finally taking $\a' > \a + w_{2}$ ensures that condition 
	\eqref{e:pflag} in Theorem~\ref{t:pdeflected} is satisfied.  Therefore for these parameter
	values the extended probe $\cDP$ displaces the Lagrangian fiber $L_{w} \subset U_{n,m}$.
\end{proof}

\begin{remark}\label{r:EPvsP} 
	If $\lceil m/n \rceil > 2$, then there are $\lceil m/n \rceil - \lceil m/2n \rceil$ regions of 
	infinite measure in $\De_{n,m}^{U}$ consisting of points 
	that can be displaced by extended probes but not by standard probes.
	If $\lceil m/n \rceil = 2$, then the only points in $\De_{n,m}^{U}$ where extended probes
	are needed are the points on a ray with direction $(1,1)$.  See Figure~\ref{f:DXnm} for examples.
\end{remark}
%%%%%%%%%%%%%%%%%%%%%%%%%%%%%
%%%%%%%%%%%%%%%%%%%%%%%%%%%%%

%%%%%%%%%%%%%%%%%%%%%%%%%%%%%%%%%%%%
%%%%%%%%%%%%%%%%%%%%%%%%%%%%%%%%%%%%
%%%%%%%%%%%%%%%%%%%%%%%%%%%%%%%%%%%%

%%%%%%%%%%%%%%%%%%%%%%%%%%%%%%%%%%%%
%%%%%%%%%%%%%%%%%%%%%%%%%%%%%%%%%%%%
%%%%%%%%%%%%%%%%%%%%%%%%%%%%%%%%%%%%
\section{Displaceability in toric orbifolds and their resolutions}\label{s:DinTOR}

Wilson and Woodward \cite{WW11} recently observed that the quasi-map Floer homology developed
in \cite{Wd11}, can be used in orbifold and noncompact settings to obtain large families of nondisplaceable Lagrangian fibers.  As we will see, after (partially) resolving an orbifold singularity, many of these nondisplaceable fibers become displaceable by extended probes.  In this section we investigate this phenomenon.

%%%%%%%%%%%%%%%%%%%%%%%%%%%%%
%%%%%%%%%%%%%%%%%%%%%%%%%%%%%
\subsection{Proving nondisplaceability results with potential functions}\label{ss:pot}

Given a presentation of a rational 
simple 
polytope
\begin{equation}\label{e:mp}
	\De = \{x \in \R^{n} \mid \ell_{i}(x) \geq 0\,,\,i=1, \ldots, N\} \mbox{ with }
	\ell_{i}(x) = \ip{\eta_{i}, x} + \ka_{i}
\end{equation}
where $\ka_{i} \in \R$ and $\eta_{i} \in \Z^{n}$ are 
integer vectors, one can build a symplectic toric   
orbifold $(M^{2n}_{\De}, \w, \T^{n}, \Phi)$ such that $\Phi\co M_{\De} \to \R^{n}$ is proper and 
$\De = \Phi(M_{\De})$, which is
unique up to equivariant symplectomorphism by \cite{KL09}.

By \cite[Proposition~6.8]{Wd11},  
$\De$ has at least one vertex exactly if 
 $M_{\De}$ can be represented as a symplectic 
reduction  $\C^{N}\!/\!\!/G$ 
of the standard $(\C^{N}, \w_{0}, (S^{1})^{N}, \Phi_{0})$, where 
$G \subset (S^{1})^{N}$ 
is a suitable 
$(N-n)$ dimensional subtorus and 
$\Phi\co M_{\De} \to \R^{n} \equiv \ft^{*}$ is the moment map for the action of $\T=(S^{1})^{N}/G$ on $M_{\De}$.
If $\De$ does not have a vertex, then, as was noted in 
\cite[Corollary 6.9]{Wd11},
$\De \cong \De' \times V$ where $\De'$ is a rational simple
polytope with a vertex, $V = \{x \in \R^{n} \mid \ip{\eta_{i}, x} = 0 \mbox{ for } i = 1, \dots, N\}$, and one can take $M_{\De} = M_{\De'} \times (T^{*}S^{1})^{r}$ where $r = \dim(V)$.
By \cite[Proposition~6.10]{Wd11}, the results of \cite{Wd11,WW11} hold in both cases.

Consider  the field of generalized Laurent series in the variable $q$
$$
\La = \left\{\sum_{d\in\R} a_{d}q^{d} \mid a_{d} \in \C \mbox{ and $\{d \mid a_{d} \not=0\}
\subset \R$ is discrete and bounded below}\right\}.
$$
The field $\La$ is complete with respect to the norm $\norm{\cdot} = e^{-\nu(\cdot)}$
induced by the non-Archimedian valuation 
$$
\nu\co\La \to \R \cup \{\infty\} \quad\mbox{where}\quad
\nu\left(\sum_{d} a_{d}q^{d}\right) = \min(d \mid a_{d} \not=0),
$$
with the convention that $\nu(0) = \infty$, and $\nu$ satisfies
$$
	\nu(xy) = \nu(x)+\nu(y)\quad\mbox{and}\quad \nu(x + y) \geq \min(\nu(x), \nu(y))
$$
where the inequality is an equality if $\nu(x) \not= \nu(y)$.
The subring of elements with only non-negative powers of $q$,
$\La_{0} = \{\nu \geq 0\}$ is a local ring with maximum ideal
$\La_{+} = \{\nu > 0\}$.  Completeness gives that the exponential function $\exp\co\La_{0} \to \La_{0}$,
defined via the standard power series, is surjective onto the units
$\La_{0}^{\times} = \{\nu = 0\}$.

Associated to a rational simple polytope $\De$
there is the 
\emph{bulk deformed potential}
that for each
$x \in \Int \De$ and $\a = (\a_{1}, \ldots, \a_{N}) \in \La_{0}^{N}$ is a function
\begin{equation}\label{e:P}
	W_{x,\a}\co\La_{0}^{n} \to \La_{0} \quad\mbox{defined by}\quad 
	W_{x,\a}(\b) = \sum_{i=1}^{N} \exp(\ip{\eta_{i}, \b} + \a_{i})\,\, q^{\ell_{i}(x)}
\end{equation}
where $\eta_{i}$ and $\ell_{i}$ are from \eqref{e:mp}.

The following theorem allows one to prove nondisplaceability results merely by finding critical points of the potential function $W_{x,\a}$.  It was proved by Fukaya--Oh--Ohta--Ono \cite[Theorem 9.6]{FOOO10} for smooth closed toric manifolds (and for geometric $W$),
by Woodward \cite[Proposition 6.10 and Theorem 7.1]{Wd11} for rational simple polytopes,
and by Wilson--Woodward \cite[Theorem 4.7]{WW11} for rational simple polytopes for open symplectic toric orbifolds.

\begin{thm}\label{t:potential}
	For a toric orbifold
	$(M_{\De}, \w, \Phi)$
	as above and $x\in\INT(\De)$, if there exists $\a\in\La_{0}^{N}$ 
	such that
	$W_{x,\a}\co\La_{0}^{n} \to \La_{0}$ has a critical point, 
	then the Lagrangian torus fiber
	$\Phi^{-1}(x) = L_{x} \subset (M_{\De}, \w)$ is nondisplaceable. 
\end{thm}

The basic idea is due to Cho--Oh \cite{ChO06}, where the holomorphic disks used to define
the $A_{\infty}$-structure associated to the Lagrangian Floer homology and quasi-map Floer homology for $L_{x}$ are explicitly classified.  It turns out that if $\b$ is a critical 
point for $W_{x,\a}$
then, in the smooth case, Lagrangian Floer homology 
with differential $d^\b$ depending on $\b$
is defined and nonzero for $L_x$, so that
$L_{x}$ is nondisplaceable.  
For general $W_{x,\a}$ a similar statement holds for  the  quasi-map 
Floer homology of $L_{x}$.  Note that the parameters $\a$ and $\b$ correspond to bulk
deformations and weak bounding cochains, in the language of Fukaya--Oh--Ohta--Ono.
If $x\in\INT(\De)$ is a critical point of $W_{x,\a}$ we will say that it has
{\bf nontrivial (or nonzero) qW invariants}.

\begin{remark}\label{r:WW} (i) Observe that the potential 
function $W_{x,\a}$ 
depends on the presentation of $\De$ in \eqref{e:mp} as a polytope and not just on $\De$ as a subset of
$\R^{n}$.  Equivalently $W_{x,\a}$ depends on the presentation of $M_{\De}$ as a reduction of $\C^{N}$ and not just on $M_{\De}$ as a symplectic toric orbifold. 
In papers such as \cite{ChO06,FOOO10} that work in the context of Lagrangian Floer homology on smooth manifolds  it is assumed that the polytope $\De$ has precisely $N$ facets, and one builds the invariant by counting holomorphic discs in 
$M_{\De}$
that intersect these facets.   
In this case, we call $W_{x,\a}$ the ``geometric" potential function.
However,  in the quasi-map approach of Woodward \cite{Wd11}, the invariant is built from 
holomorphic discs in $\C^N$, that intersect the $N$ facets of $\R^N_+$.  Since the geometry takes place in $\C^N$ there is no need for each of these $N$ facets to descend to a geometric facet of $\De$; some of them may be \lq\lq ghosts" with constants $\ka_i$ chosen so large that
$\ell_i(x)> 0$ for all $x\in  \De\setminus f$ where $f$ is a (possibly empty) face of dimension 
less than $n-1$.
\MS

\NI (ii)
It is clear from equation \eqref{e:P} above, that if a ghost facet is parallel to a geometric facet of $\De$ then we can amalgamate the two corresponding terms in $W_{x,\a}$: if the geometric facet
has $\ell_1(x) = \langle\eta_1,x\rangle +\ka_1$ then the ghost facet is $\langle\eta_1,x\rangle + \ka_1'=0$ where 
$\ka_1'>\ka_1$,
so that if $W$ is the original potential and $W'$ is the potential with the ghost facet,
we have that $W'_{x,\a'} = W_{x,\a}$ where
$$
\a' = (\a'_1, \a_1, \a_2, \dots, \a_N)\ \mbox{ and }\ 
\a = (\a_1+\log(1+ \exp(\a'_1 -\a_1) q^{\ka_1'-\ka_1}),\a_2,\dots, \a_N).
$$ 
Thus, the parallel ghost facet affects the terms in $\a$ with positive $q$ weight; in the language of \cite{FOOO11} it is a bulk deformation.
\end{remark}

\begin{remark}
If $\De$ is a smooth compact moment polytope with rational support constants, it follows from 
\cite[Proposition 4.7]{FOOO10a} and \cite[Theorem 4.5]{FOOO11} that there is always a $u \in \Int(\De)$
such that the geometric potential $W_{u,\a}$ has a critical point for some $\a \in \La_{0}^{N}$.
\end{remark}

 In \cite{WW11}, Wilson--Woodward observed that 
 ghost facets can give new information if $\De$ has singularities or 
 corresponds to an open symplectic toric orbifold.
 Lemma~\ref{l:DNDm} and Remark~\ref{r:DMDm}  below show precisely how 
 ghost facets may create lines with nontrivial invariants and Theorem~\ref{t:DND}(i) is
 an example where ghost facets create open sets with nontrivial invariants.
 In contrast to this we will now prove that ghost facets give no new information if the polytope
 is smooth and closed,
 which explains \cite[Remark 6.11]{Wd11}.
 Note that part (i) of the next proposition has analogs in all dimensions, but we
 restrict to dimension $2$ for simplicity.

\begin{prop}\label{p:noghost}  Let $\De$ be a smooth 
closed polytope in $\R^{n}$.
\begin{itemize}\item[(i)]  If $\De$ is compact and $2$-dimensional, the set of points in $\De$ with nontrivial qW invariants is the union of a finite number of points with at most one line segment.  
\item[(ii)]  In any dimension, adding ghost facets to the potential does not change the set of points
$u \in \Int \De$ such that $W_{u,\a}$ has a critical point for some $\a \in \La_{0}^{N}$.
\end{itemize}
\end{prop}
\begin{proof}    We use the notation of \eqref{e:mp}. We first prove (i) in the case of the geometric qW potential to explain the idea in a simple case.  We then prove (ii), which implies (i) in the general case.

For each point $u\in  \De$
define $s(u): = \min_{i\le N} \ell_i(u)$ and denote the set of edges that are closest to $u$
by $E_1(u) = \{i\in \{1,\dots,N\}: \ell_i(u) = s(u)\}$. 
If $\#E_1(u) = 2$, and the edges in $E_1(u)$ are not parallel
then we can choose coordinates so that one edge in $E_1(u)$ has equation 
$x_1 =0$, while the other has the form $ax_1+bx_2 =0$ 
where $b\ne 0$.  Then, because  
for all $\a\in \La_0$ we have $e^\a = z+ \mbox{positive powers of $q$}$, 
where $z \in \C^{*}$, we find that
$$
\partial_{\b_2}  W_{u,\a} = z\, b\, e^{a\b_1+b\b_2}q^{s(u)} + O(q^c),\quad c>s(u),
$$
which means that $u$ is not a critical point of $W_{u,\a}$.  A similar argument shows that $u$ is not critical when 
$\#E_{1}(u) = 1$. On the other hand if the two edges in $E_1(u)$ are parallel and we choose 
coordinates so that these have equations 
$\pm x_{1} + \ka = 0$,
then $\partial_{\b_1}  W_{u,\a} =0$ can be solved to lowest order in $q$.
Further the equation $ \partial_{\b_2}  W_{u,\a} =0$ starts with terms involving $q^c$ where $c>s(u)$, and its lowest order terms also 
have a solution  if at least two of these involve the same power of $q$.
Equivalently, we need $\#E_2(u)>1$, where $E_2(u)$ consists of those facets not in $E_1(u)$ that are closest to $u$.
We may now appeal to
\cite[Theorem 4.5]{FOOO11} which says that if the system of 
lowest order equations has a solution, then one can choose the higher order terms in $\a,\b$ to obtain a solution of the full system of
 equations $ \partial_{\b_1}  W_{u,\a}(\b)=0, \partial_{\b_2}  W_{u,\a}(\b)=0$.

All the other critical points  have 
  $\#E_1(u) \ge 3$.  Since there are only  finitely many such points,
it remains to check that there is at most one line segment  consisting of points with $\#E_1(u) = 2$.

To see this, note first that for any two parallel edges, the set of points equidistant from them is convex.  Hence if there are two such line segments, $\De$ must have two sets of parallel sides, and hence be the blow up of a rectangle.  But in a rectangle only one set of parallel lines can appear as $E_1(u)$, and if it is a square there are no points with $\#E_1(u)= 2$.  This proves (i).

Now consider (ii). Remark~\ref{r:WW}  deals with the case when
the ghost facet is parallel to some facet of $\De$.  Therefore suppose it is not.
Without loss of generality, we may
suppose that $\ell_g(x) + \eps \geq 0$ defines a ghost facet for $\eps \geq 0$ 
that intersects $\De$  when $\eps=0$ in a
codimension $d$ face  
$\{\ell_g = \ell_1 = \dots = \ell_d = 0\}$.
Then we may choose coordinates so that $\ell_i(x) = x_i$ and 
$\ell_g(x) = a_1 x_1 + \dots + a_d x_d$ where $a_i \geq 1$.  So in particular we see that for
any interior point $x$ and any $\eps \geq 0$:
\begin{equation}\label{e:eg}
	\mbox{$0 < \ell_i(x) < \ell_g(x) + \eps$ for $i = 1,\dots,d$}.
\end{equation}
The potential with the ghost facet added is given by 
$W(g)_{u,\a} = W_{u,\a} + e^{\ell_g(\b) + \a_g} q^{\ell_g(u)}$.  Observe for $j > d$ that
$\partial_{\b_j}W$ is unaffected by the ghost term.  For $j \leq d$, it follows from \eqref{e:eg} that
the leading order terms in $\partial_{\b_j}W$ are unaffected by the ghost term.  Therefore the leading order term critical point equation for $W(g)_{u,\a}$ and $W_{u,\a}$ are the same, so again by 
\cite[Theorem 4.5]{FOOO11} both potentials have the same set of points $u \in \Int(\De)$ that give rise to critical points.  This proves (ii).
\end{proof}

%%%%%%%%%%%%%%%%%%%%%%%%%%%%%
%%%%%%%%%%%%%%%%%%%%%%%%%%%%%

%%%%%%%%%%%%%%%%%%%%%%%%%%%%%
%%%%%%%%%%%%%%%%%%%%%%%%%%%%%
\subsection{A simple example and its resolution to $\cO(-m)$, for $m \geq 2$}\label{ss:Om}
For an integer $m \geq 2$, consider the orbifold $(M_{m}, \w, \Phi_{m})$ whose
moment polytope is the sector
$$
	\De_{1,m} = \Big\{x = (x_{1}, x_{2}) \in \R^{2} \mid \ell_{0}(x):= x_{1} \geq 0\,,\,\,
	\ell_{2}(x):= -x_{1} + mx_{2} \geq 0\Big\}.
$$
Here $M_{m} = \C^{2}/\Z_{m}$ where the generator in the cyclic group 
acts by diagonal multiplication by $e^{2\pi i/m}$.
If the torus $\T^2$ acts via
$$
(\la_1,\la_2)\cdot(z_1,z_2) = (\la_1\la_2 z_1,\la_2z_2),
$$
the moment map is given by
$$
	\Phi_{1,m}(z_{1}, z_{2}) = \big( \abs{z_{1}}^{2},\, 
	\tfrac{1}{m}\abs{z_{1}}^{2} + \tfrac{1}{m} \abs{z_{2}}^{2} \big).
$$
The orbifold singularity at the origin can be resolved with the facet 
$$ 	
\{\ell_{1}^{\ka}(x) := x_{2} + \ka \geq 0\},\quad \ka < 0.
$$ 
In fact,  if $\ka < 0$, then the polytope
\begin{equation}\label{e:R1m}
	\ov{\De}_{1,m}(\ka) = \{x \in \R^{2} \mid \ell_{0}(x) \geq 0\,,\,\ell_{1}^{\ka}(x) \geq 0\,,\,
	\ell_{2}(x) \geq 0 \}
\end{equation}
is smooth and is the moment polytope for the standard toric structure on the line bundle 
$\cO(-m) \to \CP^{1}$.  On the other hand if $\ka \geq 0$ then $\ov{\De}_{1,m}(\ka) = \De_{1,m}$ 
as subsets of $\R^{2}$, so $\{\ell_{1}^{\ka}(x) \geq 0\}$ defines a ghost facet in the presentation 
\eqref{e:R1m}.  The effect of resolving the orbifold singularity in this case is easy to explain, while the answers become more complicated for the later examples.  

%%%%%%%%%%%%%%%%%%%%%%%
\subsubsection{The displaceable and nondisplaceable fibers before resolving}\label{ss:DNDm}

\begin{figure}[h]

\begin{center} 
\leavevmode 
\includegraphics[width=3in]{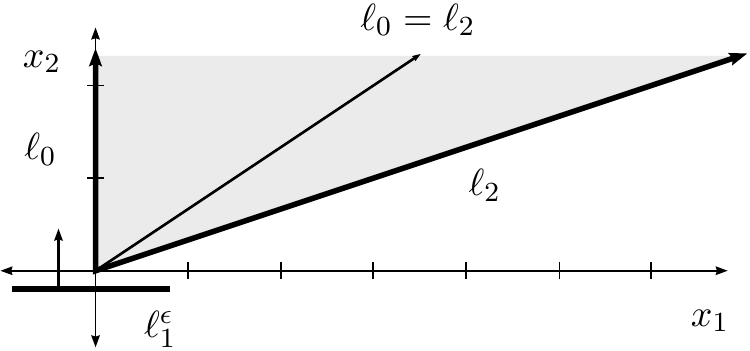}
\end{center} 

	\caption{The moment polytope $\De_{1,m}$ when $m=3$.
	The displaceable points are in light gray
	and the nondisplaceable points are on the line $\ell_{0} = \ell_{2}$.
	The nondisplaceability proof uses ghost facets $\{\ell^{\e}_{1}(x) \geq 0\}$ for varying
	$\e \geq 0$.}
	\label{f:1m}
\end{figure}

\begin{lem}\label{l:DNDm}
	If $x \in \De_{1,m}$ is on the line $\ell_{0}(x) = \ell_{2}(x)$,
	then the Lagrangian fiber $L_{x}$ is nondisplaceable.
	If $x$ is not on the line $\ell_{0}(x) = \ell_{2}(x)$, then $L_{x}$ is displaceable
	by a horizontal probe.
\end{lem}
\begin{proof}
	The displaceability statement is straightforward.
	
	For each $x \in \De_{1,m}$ on the line $\ell_{0}(x) = \ell_{2}(x)$, there is some $\e \geq 0$
	such that
	\begin{equation}\label{e:m1F}
		\ell_{0}(x) = \ell_{2}(x) = \ell_{1}^{\e}(x)
	\end{equation}
	For a given $x$ and the associated $\e \geq 0$, 
	consider the presentation $\ov{\De}_{1,m}(\e)$ from \eqref{e:R1m}, which has the ghost facet
	$\{\ell_{1}^{\e}(x) \geq 0\}$.  Its potential function is
	$$ 
		W_{x,\a}(\b_{1}, \b_{2}) = 
		e^{\b_{1}} q^{\ell_{0}(x)}+
		e^{-\b_{1} +m\b_{2}} q^{\ell_{2}(x)} + e^{\b_{2}+\a} q^{\ell_{1}^{\e}(x)}.
	$$ 
	Setting $\partial_{\b_{1}}W_{x,\a} = \partial_{\b_{2}}W_{x,\a} = 0$ 
	and using \eqref{e:m1F} to cancel out the powers of $q$ gives
	\begin{equation}\label{e:pm1gc}
		e^{\b_{1}} -e^{-\b_{1} +m\b_{2}}= 0 \quad \mbox{and}\quad
		m\, e^{-\b_{1}+m\b_{2}}  + e^{\b_{2}+\a} = 0.
	\end{equation}
	These equations are solved by $(\b_{1}, \b_{2}, \a) = (0,0,\log(-m))$ over $\C$.
	Therefore by Theorem~\ref{t:potential}, the Lagrangian
	fiber $L_{x}$ is nondisplaceable.
\end{proof}
	
	\begin{remark} \label{r:DMDm} (i)
	The ghost facet was needed, for without it the potential function is
	$$
		W_{x,\a}(\b_{1}, \b_{2}) = 
		e^{\b_{1}+\a_{2}} q^{x_{1}} + e^{-\b_{1} + m\b_{2}+\a_{1}}q^{-x_{1} + mx_{2}},
	$$
	which has no critical points since $\partial_{\b_{2}}W_{x,\a} $ has just one non-zero term
	for all $x$ and $\a$.

(ii)   The bulk deformation $e^{\a}$ is also needed when $m=2$.  To see this, note
that	under the substitution $y_{1} = e^{\b_{1}}$ and $y_{2} = e^{\b_{2}}$
	\eqref{e:pm1gc} becomes
	$$
		y_{1}^{2} = y_{2}^{m} \quad\mbox{and}\quad y_{1} = -m y_{2}^{m-1}.
	$$ 
	For $m=2$ this says
	$$
		y_{1}^{2} = y_{2}^{2} \quad\mbox{and}\quad y_{1} = -2 y_{2},
	$$
	which has no solution except $y_{1} = y_{2} = 0$.
	\end{remark}
%%%%%%%%%%%%%%%%%%%%%%%

%%%%%%%%%%%%%%%%%%%%%%%
\subsubsection{Resolving $\De_{1,m}$ to $\cO(-m)$}\label{s:1m}

If $\ka < 0$, then $\ov{\De}_{1,m}(\ka)$ is a resolution
of $\De_{1,m}$. 
The next result shows that all the previously nondisplaceable fibers can now be
displaced, either by standard probes based on the new facet $\{\ell_{1}^{\ka}(x) = 0\}$ 
or by parallel extended probes deflected by a 
probe $Q$ that is based on the new facet.

\begin{lem}\label{l:O-m} 
	If $\ka < 0$, then the Lagrangian
	fiber $L_{u} \subset \cO(-m)$ is displaceable by extended probes for
	all   $u \in \ov{\De}_{1,m}(\ka)$.
\end{lem}
\begin{proof}
Just as in the Hirzebruch surface case, for even
$m \geq 2$ standard probes displace everything.
When $m > 2$ is odd, horizontal probes displace everything except the points
$x \in \ov{\De}_{1,m}(\ka)$ on the line $\ell_{0}(x) = \ell_{2}(x)$, which up to translation
is identified with   
\begin{equation}\label{e:DX1m}
	\{x_{1} = \tfrac{m}{2}x_{2} + \tfrac{-\ka}{2}\} \subset 
	\{x \in \R^{2}_{+} \mid -x_{1} + mx_{2} -\ka > 0\} =: \De^{U}_{1,m}(-\ka).
\end{equation}
So it suffices to prove that points on the line in \eqref{e:DX1m} are displaceable
and this follows from Lemma~\ref{l:openEP} since $m/2 < m-1$.
\end{proof}
%%%%%%%%%%%%%%%%%%%%%%%

%%%%%%%%%%%%%%%%%%%%%%%%%%%%%
%%%%%%%%%%%%%%%%%%%%%%%%%%%%%

%%%%%%%%%%%%%%%%%%%%%%%%%%%%%
%%%%%%%%%%%%%%%%%%%%%%%%%%%%%
\subsection{Cyclic surface singularities and their Hirzebruch--Jung resolutions}\label{s:HJR}

%%%%%%%%%%%%%%%%%%%%%%%
\subsubsection{The orbifolds}\label{s:Mnm}

For relatively prime positive integers $m > n \geq 2$ consider the complex orbifold
$M_{n,m} = \C^{2}/\G$ where $\G = \G_{n,m}$ is the cyclic subgroup of $U(2)$ generated by
the matrix
$$
	\begin{pmatrix}
	\exp(2\pi i n/m) & 0 \\
	0 & \exp(2\pi i/m)
	\end{pmatrix}.
$$
The standard symplectic toric structure on $(\C^{2}, \w_{0}, (S^{1})^{2}, \Phi_{0})$ induces a symplectic
toric orbifold structure on $(M_{n,m},\, \w,\, \T,\, \Phi_{n,m})$ with an orbifold singularity at the origin, where 
$$
\Phi_{n,m}([z_{1}, z_{2}]) = (\abs{z_{1}}^{2},\, \tfrac{n}{m}\abs{z_{1}}^{2} + \tfrac{1}{m}\abs{z_{2}}^{2})
$$
and the moment polytope is $\Phi_{n,m}(M_{n,m})$
\begin{equation}\label{e:nm}
%	\Phi_{n,m}(M_{n,m}) = 
	\De_{n,m} = \{x \in \R^{2} \mid \ell^{v}(x) := x_{1} \geq 0\,,\,\, 
	\ell^{s}(x):=-nx_{1} + mx_{2} \geq 0\}.
\end{equation}
Note that interior conormals for $\De_{n,m}$ are
$$
	\ell^{v}(x) = \ip{\eta^{v}, x} \mbox{ where } \eta^{v} = (1,0) \quad\mbox{and}\quad
	\ell^{s}(x) = \ip{\eta^{s}, x} \mbox{ where } \eta^{s} = (-n,m)
$$
and we will call $\{\ell^{v} = 0\}$ the \emph{vertical edge} and $\{\ell^{s} = 0\}$ the \emph{slant edge}.
Note $\{\ell^{v}(x) = \ell^{s}(x)\}$ defines the line $\{x_{1} = \tfrac{m}{n+1}\,x_{2}\}$, which we will call the \emph{midline}.

\begin{figure}[h]

	\begin{center} 
	\leavevmode 
	\includegraphics[width=3in]{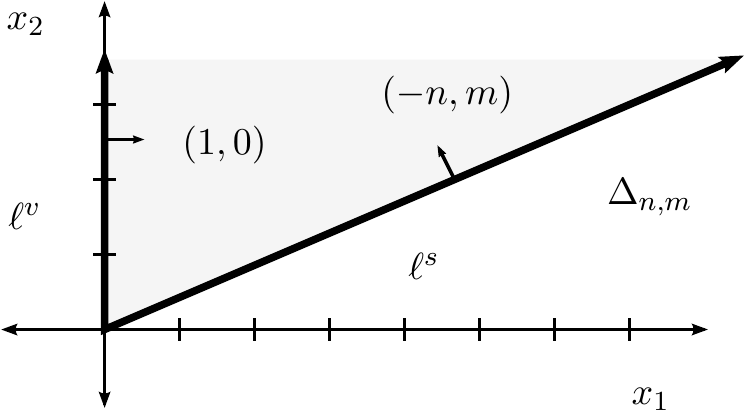}
	\end{center} 

	\caption{The moment polytope $\De_{n,m}$ for $(n,m) = (3,7)$.}
	\label{f:nm}
\end{figure}
In terms of the polytope, the assumption that $m > n$ is harmless since if $n > m \geq 2$,
then by applying a shear $(x_{1}, x_{2}) \mapsto (x_{1}, -x_{1} + x_{2})$, which is in $GL_{2}(\Z)$,
we see that $(M_{n,m},\,\De_{n,m})$ is equivalent to $(M_{(n-m,m)},\, \De_{(n-m,m)})$.
%%%%%%%%%%%%%%%%%%%%%%%

%%%%%%%%%%%%%%%%%%%%%%%
\subsubsection{Hirzebruch--Jung resolutions}

Associated to $(M_{n,m},\,\De_{n,m})$ is a minimal resolution of the symplectic toric orbifold singularity at the origin which is known in algebraic geometry as a Hirzebruch--Jung resolution.  The version in the symplectic toric setting is due to Orlik--Raymond \cite{OR70}, see also \cite{CS04}.  To find the resolution, one writes $n/m$ as a continued fraction using positive integers $E_{j} \geq 2$:
\begin{equation}\label{e:cf}
	\frac{n}{m} = \frac{1}{E_{1} - \frac{1}{E_{2} - \cdots \frac{1}{E_{k}}}} =: (E_1,\dots,E_k).
\end{equation}
The positive integers $E_{j}$ are given by the Euclidean 
algorithm where $0 \leq r_{j+1} < r_{j}$,
\begin{equation}\label{e:ea}
	r_{-1} = m, \quad r_{0} = n,\quad E_{j+1} = \left\lceil \tfrac{r_{j-1}}{r_{j}}\right\rceil,
	\quad r_{j+1} = E_{j+1}r_{j} - r_{j-1},
\end{equation}
and $k$ is the smallest number such that $r_{k} = 0$.  The sequence of integers $E_{1}, \dots, E_{k}$ determine a sequence of $k+2$ interior conormals in $\ft_{\Z}$, starting with the conormal for the vertical edge
$$
	\eta_{0} = \eta^{v} = (1,0) \quad\mbox{then}\quad \eta_{1} = (0,1)
$$
and then defined recursively for $1 \leq j \leq k$
$$ 
	\eta_{j+1} = E_{j} \eta_{j} - \eta_{j-1}
$$
where the last one is the conormal for the slant edge
$$
	 \eta_{k+1} = \eta^{s} = (-n,m).
$$
These conormals are such that if $\eta_{j+1} = (-n_{j}, m_{j})$, then
$$
	\frac{n_{j+1}}{m_{j+1}} = \frac{1}{E_{1} - \frac{1}{E_{2} - \cdots \frac{1}{E_{j}}}}
	= (E_1,\dots,E_j).
$$
For appropriate support constants $\ka = (\ka_{1}, \dots, \ka_{k}) \in \R_{< 0}^{k}$ the polytope
\begin{equation}\label{e:mrp}
%	\ov{\De}_{n,m}(\ka) = \{x \in \R^{2} \mid \ell_{0}(x) \geq 0\,,\,\, \ell_{1}^{\ka_{1}}(x) \geq 0 \,,\,\,
%	\dots\,,\,\, \ell_{k}^{\ka_{k}}(x) \geq 0\,,\,\, \ell_{k+1}(x) \geq 0\}
	\ov{\De}_{n,m}(\ka) = \{x \in \R^{2} \mid \ell_{0}(x) \geq 0,\, \ell_{1}^{\ka_{1}}(x) \geq 0,
	\dots ,\, \ell_{k}^{\ka_{k}}(x) \geq 0,\, \ell_{k+1}(x) \geq 0\}
\end{equation}
where
$$ 
	\ell_{0}(x) = \ell^{v}(x) = x_{1}\,, \quad \ell_{j}^{\ka_{j}}(x):= \ip{\eta_{j}, x} + \ka_{j}\,, \quad
	\ell_{k+1}(x) = \ell^{s}(x) = -nx_{1} + mx_{2}\,,
$$
has  $k+2$ edges, is smooth, and corresponds to a symplectic toric manifold
$$
	(\ov{M}_{n,m},\, \ov{\w}_{\ka},\, \ov{\Phi}_{n,m})
$$
that is called a {\bf minimal resolution} of $(M_{n.m},\,  \De_{n,m})$.
%%%%%%%%%%%%%%%%%%%%%%%

%%%%%%%%%%%%%%%%%%%%%%%
\subsubsection{Symmetries}

The class of examples $(M_{n,m},\,\De_{n,m})$ where $m > n \geq 2$ has the following symmetry, which we will exploit to shorten the proofs below. Let $(\tilde{n}, q)$ be the integers that solve
\begin{equation}\label{e:pq}
	mq-n\tilde{n}=-1 \quad\mbox{for minimum positive $\tilde{n}$},
\end{equation}
so that $1 < \tilde{n} < m$ and $0 < q < \tilde{n}$.  Then 
the matrix
\begin{equation}\label{e:M}
		S = \begin{pmatrix} -\tilde{n} & m\\ -q& n \end{pmatrix}
		\mbox{,  with } S\begin{pmatrix} 0\\ 1\end{pmatrix} =
		\begin{pmatrix} m\\ n\end{pmatrix} \mbox{ and }
		S\begin{pmatrix} m\\ \tilde{n}\end{pmatrix} =
		\begin{pmatrix} 0\\ 1\end{pmatrix},
\end{equation}
has $\det S = -1$ and $S(\De_{\tilde{n},m}) = \De_{n,m}$.
Furthermore, $S$ interchanges the roles of the vertical edge and the slant edge, while
mapping the midline to the midline.  
\begin{figure}[h]

	\begin{center} 
	\leavevmode 
	\includegraphics[width=5in]{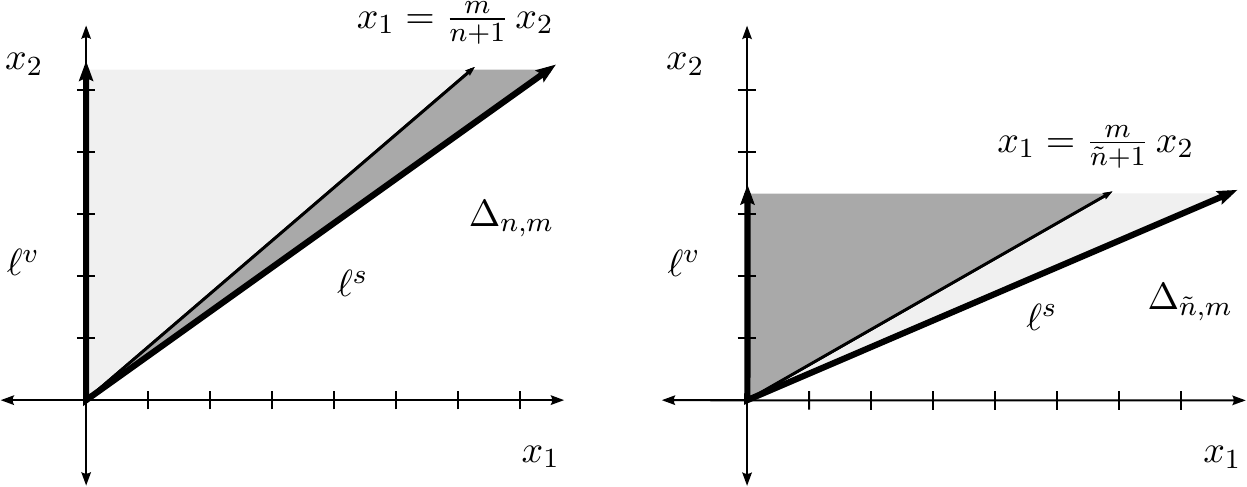}
	\end{center} 

	\caption{	
	Left: $\De_{n,m}$.
	Right: $\De_{\tilde{n},m}$.
	The matrix $S$ in \eqref{e:M} maps $\De_{\tilde{n},m}$ to $\De_{n,m}$,
	mapping the light gray region to the light grey region,
	and likewise for the dark grey regions.	
	Here $(n,m) = (5,7)$ and $(\tilde{n},q) = (3,2)$.}
	\label{f:nmS}
\end{figure}
Therefore we will often only need to prove a result for points to the left of the midline:  
the properties of the points to the right of the midline will be deduced by applying the matrix $S$.

This symmetry provided by $S \in GL(\ft^{*}_{\Z})$ is compatible with the resolution given by the continued fraction expansion.  Namely if $E_{1}, \dots, E_{k}$ are associated with the pair $(n, m)$, then 
$$
	\frac{\tilde{n}}{m} = (E_k,\dots, E_1),
$$
so that $(\tilde{n}, m)$ is given by reversing
the order of the $E_{j}$'s.  If $\eta_{0}, \dots, \eta_{k+1}$
are the conormals associated with $\ov{\De}_{n,m}$, 
and $\tilde{\eta}_{0}, \dots, \tilde{\eta}_{k+1}$
are the conormals associated with $\ov{\De}_{\tilde{n}, m}$, then one can check that 
$$
	 S^{*}\eta_{j} = \tilde{\eta}_{k+1-j},
$$
where $S^{*} \in GL(\ft_{\Z})$ is the transpose.
Therefore $S$ maps one minimal resolution from \eqref{e:mrp} to the other
$$
S(\ov{\De}_{\tilde{n},m}(\tilde{\ka})) = \ov{\De}_{n,m}(\ka) \quad\mbox{where}\quad 
\tilde{\ka}_{k+1-j} = \ka_{j}.
$$
In particular we have that the first three and the last three conormals for $\ov{\De}_{n,m}$ are
\begin{align}
	\eta_{0} &= (1,0) && \eta_{1} = (0,1) && \eta_{2} = (-1, E_{1}) \label{e:nfr}\\
	\eta_{k+1} &= (-n,m) && \eta_{k} = (-q, \tilde{n}) && \eta_{k-1} 
	= (-(E_{k}q-n), E_{k}\tilde{n} - m). \label{e:nfr1}
\end{align}
%%%%%%%%%%%%%%%%%%%%%%%

%%%%%%%%%%%%%%%%%%%%%%%%%%%%%
%%%%%%%%%%%%%%%%%%%%%%%%%%%%%

%%%%%%%%%%%%%%%%%%%%%%%%%%%%%
%%%%%%%%%%%%%%%%%%%%%%%%%%%%%
\subsection{Displaceability in sectors and their blowups}\label{ss:DS}

For relatively prime positive integers $m > n \geq 2$, consider the symplectic
toric orbifold $(M_{n,m},\, \De_{n,m})$ from \eqref{e:nm}. 
Let $\tilde{n}, q$ be given by \eqref{e:pq} and let $E_{1}, \dots, E_{k}$ be the sequence of integers from \eqref{e:ea} associated to $(n,m)$.
In this section we will use the notation
$$
	E = E_{1} := \left\lceil\tfrac{m}{n}\right\rceil \quad\mbox{and} \quad
	\Tilde{E} = E_{k} := \left\lceil\tfrac{m}{\tilde{n}}\right\rceil.
$$

%%%%%%%%%%%%%%%%%%%%%%%
\subsubsection{The displaceable and nondisplaceable fibers in $\De_{n,m}$}

\begin{thm}\label{t:DND}
	Let $x = (x_{1}, x_{2}) \in \De_{n,m}$.
	\begin{enumerate}
	\item[(i)] The Lagrangian fiber $L_{x}$ is nondisplaceable if
	\begin{equation}\label{e:snd}
		\frac{E}{2}\, x_{2} \leq x_{1} \leq \frac{2m-\Tilde{E}\,\tilde{n}}{2n-\Tilde{E}\,q}\, x_{2}.
	\end{equation}
	\item[(ii)] If $E = \Tilde{E} = 2$, then all other fibers $L_{x}$ are displaceable.
	\item[(iii)] If $E > 2$, then all fibers with $x_{1} < \tfrac{E}{2} x_{2}$ are displaceable
	except possibly for those with $\tfrac{m}{2n}\,x_{2} \leq x_{1} < \tfrac{E}{2}\,x_{2}$.
	\item[(iv)] If $\Tilde{E} > 2$, then all fibers with 
	$\tfrac{2m-\Tilde{E}\tilde{n}}{2n-\Tilde{E}q}\,x_{2} < x_{1}$
	are displaceable except possibly for those with 
	$\tfrac{2m-\Tilde{E}\tilde{n}}{2n-\Tilde{E}q}\,x_{2} < x_{1} \leq \tfrac{m\tilde{n}}{n\tilde{n}+1}\,x_{2}$.
	\end{enumerate}
\end{thm}

\begin{figure}[h]

	\begin{center} 
	\leavevmode 
	\includegraphics[width=6in]{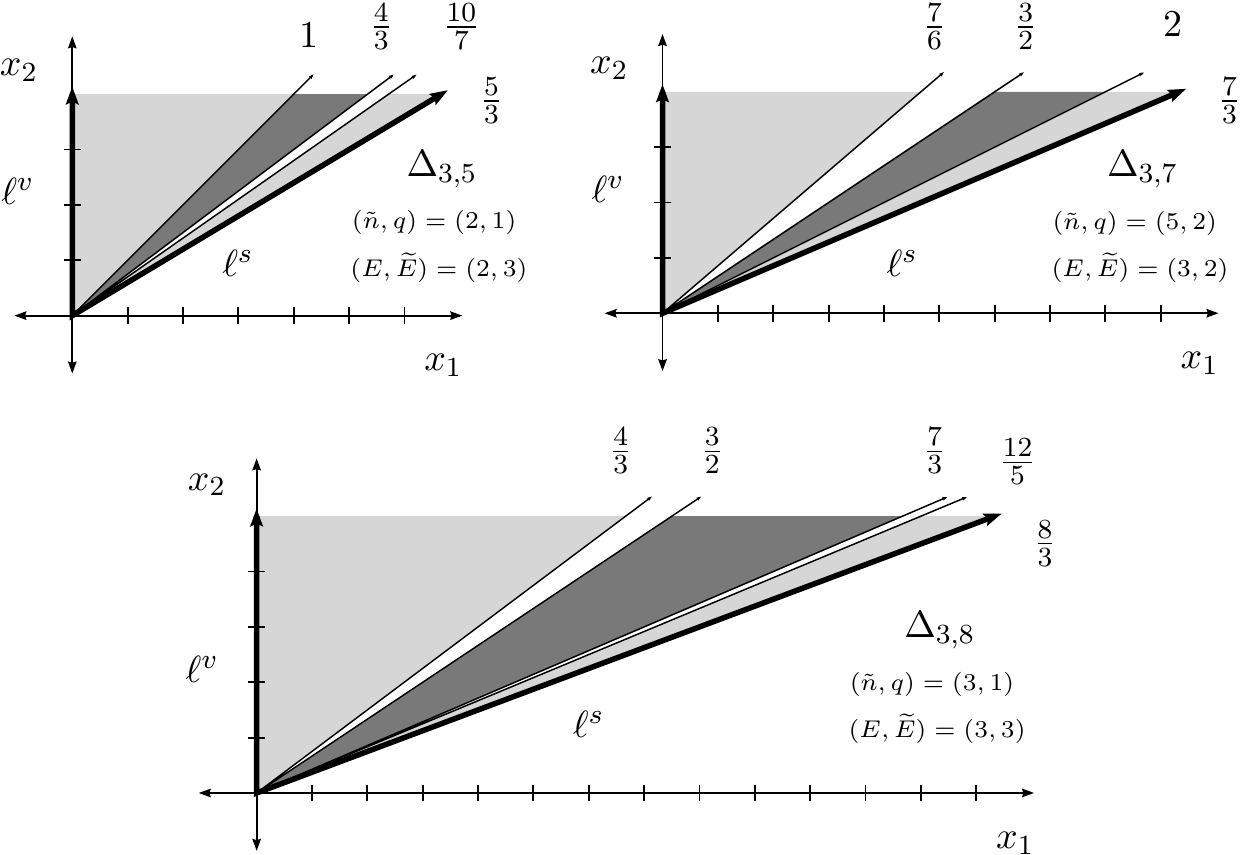}
	\end{center} 

	\caption{	
	Examples of Theorem~\ref{t:DND}. 
	The  dark grey regions
	are closed and  nondisplaceable, the  light grey regions are open and 
	displaceable, and the white regions
	are unknown.}
	\label{f:DaNDsectors}
\end{figure}

\begin{proof}[Proof of Theorem~\ref{t:DND}(i)]
	Recall that $\De_{n,m} = \{x \in \R^{2} \mid \ell^{v}(x) \geq 0\,,\,\, \ell^{s}(x) \geq 0\}$ where
	$$
		\ell^{v}(x) := x_{1} \quad\mbox{and}\quad \ell^{s}(x):= -nx_{1} + mx_{2}.
	$$
	For ghost facets, we will use the first two conormals $\eta_{1} = (0,1)$ and $\eta_{2}= (-1, E)$
	associated with the Hirzebruch--Jung resolution. Thus we take
	$$
		\ell_{1}^{\ka_{1}}(x):= x_{2} + \ka_{1} \quad\mbox{and}\quad 
		\ell_{2}^{\ka_{2}}(x):= -x_{1} + Ex_{2} + \ka_{2},
	$$
	which define ghost facets for $\De_{n,m}$ when $\ka_{1}, \ka_{2} \geq 0$ are non-negative.
		
	Observe that if the point  $x \in \De_{n,m}$ satisfies
	$$ 
		\frac{E}{2}\, x_{2} \leq x_{1} \leq \frac{m}{n+1}\,x_{2}
	$$
	then
	\begin{equation}\label{e:ell1}
		\ell^{v}(x) = \ell_{1}^{\e_{1}}(x) = \ell_{2}^{\e_{2}}(x) \leq \ell^{s}(x)
	\end{equation}
	for suitable 
	$\e_{1}, \e_{2} \geq 0$.
	The potential function of  $\De_{n,m}$ with the added
	 ghost facets $\{\ell_{1}^{\e_{1}} \geq 0\}$ and $\{\ell_{2}^{\e_{2}} \geq 0\}$ is
	$$
		W_{x,\a}(\b) = e^{\b_{1} + \a_{1}} q^{\ell^{v}(x)} + e^{-n\b_{1} + m\b_{2}}q^{\ell^{s}(x)}
		+ e^{\b_{2}+\a_{2}}q^{\ell_{1}^{\e_{1}}(x)} + e^{-\b_{1}+E\b_{2}}q^{\ell_{2}^{\e_{2}}(x)}.
	$$
	Changing variable so that $y_{1} = e^{\b_{1}}$ and $y_{2} = e^{\b_{2}}$
	and using \eqref{e:ell1}, the critical point equation becomes
	\begin{align}\label{e:Wy12}
		y_{1}\,\partial_{y_{1}} W_{x,\a} &= 
		(e^{\a_{1}}y_{1} - y_{1}^{-1}y_{2}^{E} - ny_{1}^{-n}y_{2}^{m} q^{\ell^{s}(x)-\ell^{v}(x)})\, 
		q^{\ell^{v}(x)}
		= 0\\ \notag
		y_{2}\,\partial_{y_{2}} W_{x,\a} &= 
		(e^{\a_{2}}y_{2} + E y_{1}^{-1}y_{2}^{E} + my_{1}^{-n}y_{2}^{m} q^{\ell^{s}(x)-\ell^{v}(x)})\, 
		q^{\ell^{v}(x)} = 0		
	\end{align}
	Then $(y_{1}, y_{2}) = (1,1)$ is a critical point of $W_{x,\a}$ when
	\begin{equation}\label{e:alphas}
		e^{\a_{1}} = 1 + n q^{\ell^{s}(x) -\ell^{v}(x)} \quad\mbox{and}\quad
		e^{\a_{2}} = -E - m q^{\ell^{s}(x) -\ell^{v}(x)}.
	\end{equation}
	These equations are solvable by $\a\in \La_0^2$ since $\ell^{s}(x) \geq \ell^{v}(x)$ and neither term in
	\eqref{e:alphas} is zero.
	Hence $L_{x}$ is nondisplaceable by Theorem~\ref{t:potential}.
	
	So far we have proved that for $x \in \De_{n,m}$ such that
	\begin{equation}\label{e:left}
		\frac{E}{2}\,x_{2} \leq x_{1} \leq \frac{m}{n+1}\,x_{2}
	\end{equation}
	the fiber $L_{x}$ is nondisplaceable.  Likewise, for $y \in \De_{\tilde{n},m}$,
	$L_{y}$ is nondisplaceable if
	$$
		\frac{\Tilde{E}}{2}\,y_{2} \leq y_{1} \leq \frac{m}{\tilde{n}+1}\,y_{2}.
	$$
	The image of this region under the symmetry $S$ from \eqref{e:M} is the subset of $\INT
	\De_{n,m}$
	where
	\begin{equation}\label{e:right}
		\frac{m}{n+1}\,x_{2} \leq x_{1} \leq 
		\frac{2m-\Tilde{E}\,\tilde{n}}{2n-\Tilde{E}\,q}\, x_{2}.
	\end{equation}
	Piecing \eqref{e:left} and \eqref{e:right} together, we have proved that if $x \in \De_{n,m}$ 
	satisfies 	\eqref{e:snd}, 
	then $L_{x}$ is nondisplaceable.
\end{proof}

The proof of Theorem~\ref{t:DND} is completed by the following lemma. 

\begin{lem}\label{l:Dsec}
	Probes based on the vertical edge in $\De_{n,m}$ displace the following points:
	\begin{itemize}
	\item  
	Points in $\{x_{1} < x_{2}\}$ by probes with direction $(1,1)$.			
	\item 
	Points in $\{x_{1} < \tfrac{m}{2n}\,x_{2}\}$ by probes with 
	direction $(1,0)$.
	\end{itemize}
	Probes based on the slant edge in $\De_{n,m}$ displace the following points:
	\begin{itemize}
	\item 
	Points in $\{x_{1} > \tfrac{m-\tilde{n}}{n-q}\, x_{2}\}$ by probes
	with direction $(m-\tilde{n},n-q)$.		
	\item 
	Points in $\{x_{1} > \tfrac{m\tilde{n}}{n\tilde{n}+1}\, x_{2}\}$ by probes
	with direction $(-\tilde{n},-q)$.
	\end{itemize}
\end{lem}
\begin{proof}[Proof of Lemma~\ref{l:Dsec}]
	The first two claims are similar to Lemma~\ref{l:DOsecP}, and are straightforward to check.
	The last two claims are the transform under the symmetry $S$ 
	of the first two claims for the sector $\De_{\tilde{n},m}$.
\end{proof}

We next show that the lower bound $\tfrac{E}{2}$ in Theorem~\ref{t:DND} (i)
is optimal with our current methods.

\begin{lem}\label{l:DND2} The qW invariants vanish for points in $\De_{n,m}$ with
$x_1<\frac E2 x_2$.
\end{lem}
\begin{proof}
	If the potential function $W_{x, \a}$ has a critical point at  a point
	$x \in \De_{n,m}$ left of the midline, there must be at least
	one  ghost facet $\ell^{\e}(x) = \ip{\eta, x} + \e$ such that 
	\eqref{e:ell1} holds for support
	constants $\e \geq 0$. 
	If $\eta = (-b,a)$ for some non-negative integers, 
	then $\tfrac{b}{a} \leq \tfrac{n}{m}$ must hold 
	in order for $\ell^{\e}$ to define a ghost facet.  Since  $\ell^{\e}(x) = \ell^{v}(x)$  
	implies
	$\tfrac{a}{b+1}\,x_{2} \leq x_{1}$, this  
	gives a potentially new   lower bound.
	The  choice  $\eta_{1} = (0,1)$ is optimal since it gives the bound $x_{2} \leq x_{1}$.
	However, adding just this ghost facet by itself is not enough
	since the second equation in \eqref{e:Wy12} would then have just one term and so have no solution.  Any other
	of $\eta$ must have $a,b > 0$ positive, and we claim that in this case 
	 $\tfrac{E}{2} \leq \tfrac{a}{b+1}$, so that the lower bound is no better than before.
	
	To see this, recall that $E = \lceil m/n \rceil$.
	If $a = b+1$, then $2 \geq \tfrac{a}{b} \geq \tfrac{m}{n}$, so $E=2$
	and hence $\tfrac{E}{2} = \tfrac{a}{b+1}$.  Suppose $a \geq b+2$, then
	since $E-1 < \tfrac{m}{n} \leq \tfrac{a}{b}$, we have that
	$\tfrac{E}{2} \leq \tfrac{1}{2} (\tfrac{a}{b} +1)$ and hence it suffices to
	prove $\tfrac{1}{2} (\tfrac{a}{b} +1) \leq \tfrac{a}{b+1}$.  This is equivalent to
	$a(b-1) \geq b^{2} +b$, which holds since $a \geq b+2$.
\end{proof}
%%%%%%%%%%%%%%%%%%%%%%%

%%%%%%%%%%%%%%%%%%%%%%%
\subsubsection{Displaceable fibers after a blow up}

Observe that in the proof of Theorem~\ref{t:DND}(i)
we used the ghost
facets with conormal $(0,1)$ 
to prove the nondisplaceability of the points in \eqref{e:left}, which are to the left of the midline, and implicitly we used their transforms under 
the symmetry in \eqref{e:M} with conormal $S^*(0,1) = (-q,\tilde n)$
to deal with the points to the right of the midline.
Hence, if we partially resolve the orbifold singularity with these two edges, many fibers with previously nonzero qW invariants now have vanishing invariants.
At the same time,
since probes with direction $(1,0)$ based on the vertical edge, are parallel to the new edge with conormal $(0,1)$ (and likewise on the right),
this partial resolution also  
causes many fibers to become displaceable using parallel extended probes with flags.  

\begin{prop}\label{p:asympresolve}
For the polytope $\De_{n,m}$ as in \eqref{e:nm}, 
consider a minimal resolution $(\ov{M}_{n,m},\, \ov{\De}_{n,m}(\ka)$)
given by \eqref{e:mrp} where $\ka = (\ka_{1}, \dots, \ka_{k}) \in \R_{<0}^{k}$ are the support constants.
If $x \in \Int \ov{\De}_{n,m}(\ka)$ is not in the region
$$
	(E_{1}-1)\,x_{2} \leq x_{1} \leq \frac{m-(E_{k}-1)\tilde{n}}{n-(E_{k}-1)q}\,x_{2}
$$ 
then the Lagrangian fiber $L_{x} \subset (\ov{M}_{n,m}, \ov{\w}_{\ka})$ can be displaced by extended probes in $\ov{\De}_{n,m}(\ka)$, provided that the terms in $\ka$ are sufficiently close to zero.
\end{prop}
\begin{proof}
	For $x_{1} < (E_{1}-1)\,x_{2}$, it follows from Lemma~\ref{l:openEP} that $L_{x}$ can be displaced
	by a parallel extended probe with flag $\cDP = P \cup Q \cup \cF$ in
	$\ov{\De}_{n,m}(\ka)$.  Here $P$ is based on the
	vertical edge with direction $v_{P} = (1,0)$ and $Q$ is based on the new edge
	$\{\ell_{1}^{\ka_{1}} = 0\}$ with direction $v_{Q} = (E-1, 1)$.
	
	Observe that $x_{1} > \tfrac{m-(E_{k}-1)\tilde{n}}{n-(E_{k}-1)q}\,x_{2}$
	is the transform of $x_{1} < (\Tilde{E}_{1}-1)\,x_{2}$ for $\De_{\tilde{n},m}$ under
	the transformation $S$ from \eqref{e:M}.  In $\ov{\De}_{n,m}(\ka)$ one builds
	an extended probe where $P$ is based on the slant edge, with direction $v_{P} = (-q,-\tilde{n})$,
	and the deflecting probe $Q$  has direction $v_{Q} = (m-(\Tilde{E}-1)\tilde{n}, n-(\Tilde{E}-1)q)$
	and is based on the new edge $\{\ell_{k}^{\ka_{k}} = 0\}$.
\end{proof}

\begin{remark}\label{r:resolveP} (i)
	Unwrapping how Proposition~\ref{p:asympresolve} uses Lemma~\ref{l:openEP}
	gives the following more precise version.
	Suppose that a 	(partial) resolution $\Tilde{\De}_{n,m}$ of $\De_{n,m}$
	contains 
	an  edge $\{\ell_{1}^{\ka_{1}} = 0\}$ with one endpoint on 
	the vertical edge $\{\ell^{v} = 0\}$ and the other at $(y_{1}, y_{2})$, where $y_1>0$.
	Then the points $x \in \Tilde{\De}_{n,m}$
	such that
	$$ 
		x_{1} - y_{1} < (E_{1}-1)\,(x_{2}-y_{2})
	$$ 
	are displaceable by extended probes using Lemma~\ref{l:openEP}.
	The analogous statement holds when
	$\{\ell_{k}^{\ka_{k}} = 0\}$ appears next to the slant edge $\{\ell^{s} = 0\}$.\medskip

\NI(ii)
	Comparing the results of Proposition~\ref{p:asympresolve}
	to Theorem~\ref{t:DND}, if $E = E_{1} > 2$ then the points $x \in \De_{n,m}$ in the region
	$$
		\tfrac{E}{2}x_{2} \leq x_{1} < (E-1)x_{2}
	$$
	can be displaced after we partially resolve $\De_{n,m}$ with $\{\ell_{1}^{\ka_{1}} \geq 0\}$
	and are nondisplaceable before partially resolving.  
	If $E_{1} = 2$, then this region is empty.
\end{remark}
%%%%%%%%%%%%%%%%%%%%%%%%%%%%%
%%%%%%%%%%%%%%%%%%%%%%%%%%%%%

%%%%%%%%%%%%%%%%%%%%%%%%%%%%%
%%%%%%%%%%%%%%%%%%%%%%%%%%%%%
\subsection{Examples of  minimal resolutions}\label{s:EMR}

In this section 
we discuss a few of the minimal resolutions
$\ov{\De}_{n,m}(\ka)$ in \eqref{e:mrp}.  
We will use standard probes as well as 
 the extended probes described in Proposition~\ref{p:asympresolve} and 
 Remark~\ref{r:resolveP}.
 \footnote{We leave it to the reader to check that the probes with trapezoidal flags defined in \S\ref{s:EPwFNP}
 displace no new points.}
 
In general for minimal resolutions
$(\ov{M}_{n,m},\, \ov{\De}_{n,m}(\ka))$
there is no known nondisplaceable fiber. 
In Section~\ref{s:1m} where $k=1$ and hence
$(n,m) = (1,m)$, we showed in Lemma~\ref{l:O-m} that every fiber is displaceable in a minimal resolution 
$\cO(-m) = \ov{M}_{1,m}$.  In every other case there will be fibers that we cannot displace with extended probes. 

%%%%%%%%%%%%%%%%%%%%%%%
\subsubsection{The case $k = 2$}

Suppose the continued fraction expansion for $(n,m)$ has length $k=2$ given by $E_{1}$ and $E_{2}$.  Then
$$
	(n,m) = (E_{2},\, E_{1}E_{2} - 1) \quad\mbox{and}\quad (\tilde{n}, q) = (E_{1}, 1)
$$
and the facets for a minimal resolution $\ov{\De}_{(E_{2},\, E_{1}E_{2} - 1)}$ have interior conormals
\begin{align*}
	\eta_{0} = (1,0) && \eta_{1} = (0,1) && \eta_{2} = (-1, E_{1}) && \eta_{3} = (-E_{2}, E_{1}E_{2} - 1).
\end{align*}
Note that in this case the upper and lower bounds of Proposition~\ref{p:asympresolve} coincide,
so we have the following corollary taking into account Remark~\ref{r:resolveP}. 

\begin{figure}[h]	

	\begin{center} 
	\leavevmode 
	\includegraphics[width=4in]{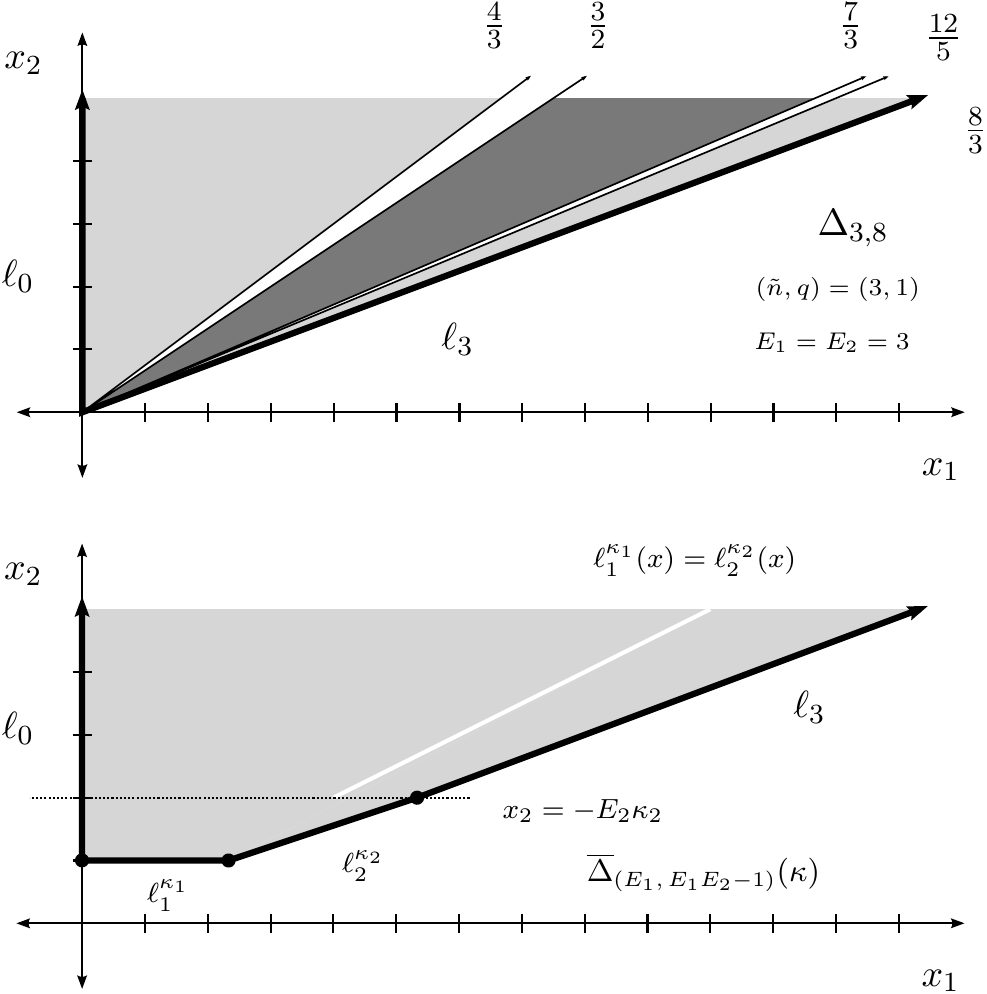}
	\end{center} 

	\caption{
	Illustration of Corollary~\ref{c:k1}.
	Above: $\De_{3,8}$ before resolving.  Below: The resolution $\ov{\De}_{(E_{1},\,E_{1}E_{2}-1)}$
	drawn in the case $(n,m) = (3,8)$.  The light gray regions are displaceable
	and the dark grey regions are nondisplaceable.
	It is unknown if the points on the ray with slope $E_{1}-1$ are displaceable.}
	\label{f:k1}
\end{figure}

\begin{cor}\label{c:k1}
	Let $\ka_{1}, \ka_{2} < 0$ be such that $\ov{\De}_{(E_{2},\, E_{1}E_{2} - 1)}(\ka)$
	is a minimal resolution from \eqref{e:mrp}.
	A Lagrangian fiber $L_{x}$ in $\ov{\De}_{(E_{2},\, E_{1}E_{2} - 1)}(\ka)$ 
	is displaceable by probes provided it is not
	on the ray given by
	$$
		\ell_{1}^{\ka_{1}}(x) = \ell^{\ka_{2}}_{2}(x)
		\quad\mbox{and}\quad
		\ell_{1}^{\ka_{1}}(x) \geq \max\Big\{\ell_{1}^{\ka_{1}}(y) \mid 
		y \in \ov{\De}_{(E_{2},\,E_{1}E_{2}-1)}(\ka)\,,\,\, \ell_{2}^{\ka_{2}}(y) = 0\Big\}
	$$
	that is
	\begin{equation}\label{e:?linek1}
	x_{1} = (E_{1}-1)\,x_{2} + \ka_{2}-\ka_{1} \quad\mbox{and}\quad
	x_{2} \geq -E_{2}\ka_{2}.
	\end{equation}
\end{cor}
\begin{proof}
	We only need to prove that if $x$ is on the line \eqref{e:?linek1}
	and $x_{2} < -E_{2}\ka_{2}$, then $L_{x}$ is displaceable.  These points
	can be displaced by probes with direction $(-1,0)$ based on the
	edge $\{\ell_{2}^{\ka_{2}} = 0\}$.  See Figure~\ref{f:k1}. 
\end{proof} 

\begin{remark}
	 If $x$ lies on the line in \eqref{e:?linek1}, then
	 $$
	 	\ell_{0}(x) > \ell_{1}^{\ka_{1}}(x)\,, \quad \ell_{1}^{\ka_{1}}(x) = \ell_{2}^{\ka_{2}}(x)\,,
		\quad \ell_{2}^{\ka_{2}}(x) < \ell_{3}(x)
	 $$
	 and hence the potential function $W_{x,\a}(\b_{1}, \b_{2})$ from \eqref{e:P} will not
	 have a critical point.  In particular one can check that 
	 $\partial_{\b_{1}} W_{x,\a}(\b) \not= 0$ for all $\b = (\b_{1}, \b_{2}) \in (\La_{0})^{2}$.
	 So Theorem~\ref{t:potential} cannot be used to prove $L_{x}$ is nondisplaceable.
\end{remark}
%%%%%%%%%%%%%%%%%%%%%%%

%%%%%%%%%%%%%%%%%%%%%%%
\subsubsection{$A_{n}$-singularities}

The $A_{n}$-singularity is $\C^{2}/\G$ where $\G$ is the subgroup of $SU(2)$ generated by
$\begin{pmatrix}
\z^{-1} & 0\\
0 & \z
\end{pmatrix}$
where $\z = e^{2\pi i/(n+1)}$.  Comparing with Section~\ref{s:Mnm}, we see that the $A_{n}$-singularity
is given by $(M_{(n,n+1)},\, \w,\, \De_{(n,n+1)})$ and note the associated $(\tilde{n}, q) = (n, n-1)$.
The continued fraction expansion for $(n,n+1)$ has length $k=n$ and is given by 
$$
E_{1} = \dots = E_{n} = 2
$$ 
and the facets for a minimal resolution $\ov{\De}_{(n, n+1)}(\ka)$ have interior conormals
\begin{align*}
	&\eta_{0} = (1,0) &&\eta_{1} = (0,1) &&\eta_{2} = (-1,2) &&\eta_{3}=(-2,3) &&
	\dots &&\eta_{n+1} = (-n,n+1).
\end{align*}
Note that in this case the upper and lower bounds of Proposition~\ref{p:asympresolve} coincide,
so we have the following corollary taking into account Remark~\ref{r:resolveP}. 

\begin{figure}[h]

\begin{center} 
\leavevmode 
\includegraphics[width=3in]{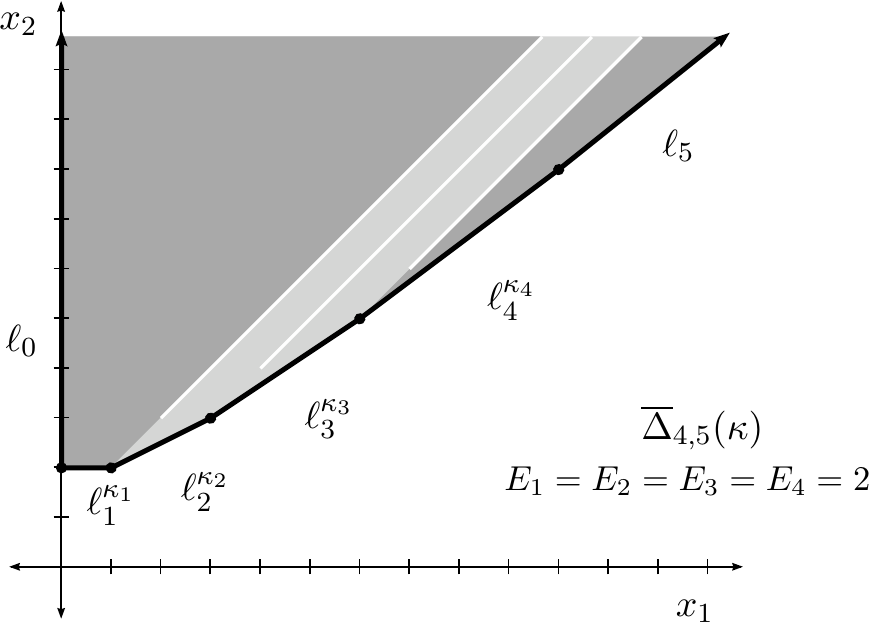}
\end{center} 

	\caption{
	Minimal resolution of $A_{n}$-singularity.
	The light grey regions are displaceable by standard probes
	with direction (1,1) or parallel to one of the added facets.
	The medium grey regions are displaceable by extended probes
	from Proposition~\ref{p:asympresolve}.
	Displaceability for the white rays is unknown.}
	\label{f:AnSing}
\end{figure}

\begin{cor}\label{c:An}
	Let $\ka \in \R_{<0}^{n}$ be such that $\ov{\De}_{(n,n+1)}(\ka)$
	is a minimal resolution from \eqref{e:mrp}.  A point $x \in \ov{\De}_{(n,n+1)}(\ka)$
	is displaceable by probes if it does not belong to one of the $(n-1)$ rays given by
	$$
		\ell_{j}^{\ka_{j}}(x) = \ell^{\ka_{j+1}}_{j+1}(x)
		\quad\mbox{and}\quad
		\ell_{j}^{\ka_{j}}(x) \geq \max\Big\{\ell_{j}^{\ka_{j}}(y) : 
		y \in \ov{\De}_{(n,n+1)}(\ka)\,,\,\, \ell_{j+1}^{\ka_{j+1}}(y) = 0\Big\}
	$$
	for $j=1, \dots, n-1$.
\end{cor}
%%%%%%%%%%%%%%%%%%%%%%%

%%%%%%%%%%%%%%%%%%%%%%%
\subsubsection{Open regions of unknown points}

If $E_{1}, \dots, E_{k}$ is the continued fraction decomposition of $(n,m)$, suppose that
its length is at least $k \geq 3$ and some $E_{j} \geq 3$ for $j\not=1,k$.  Then in any
minimal resolution $\ov{\De}_{n,m}(\ka)$ there will be open regions of points that extended probes do not displace.

\begin{figure}[h]

	\begin{center} 
	\leavevmode 
	\includegraphics[width=4in]{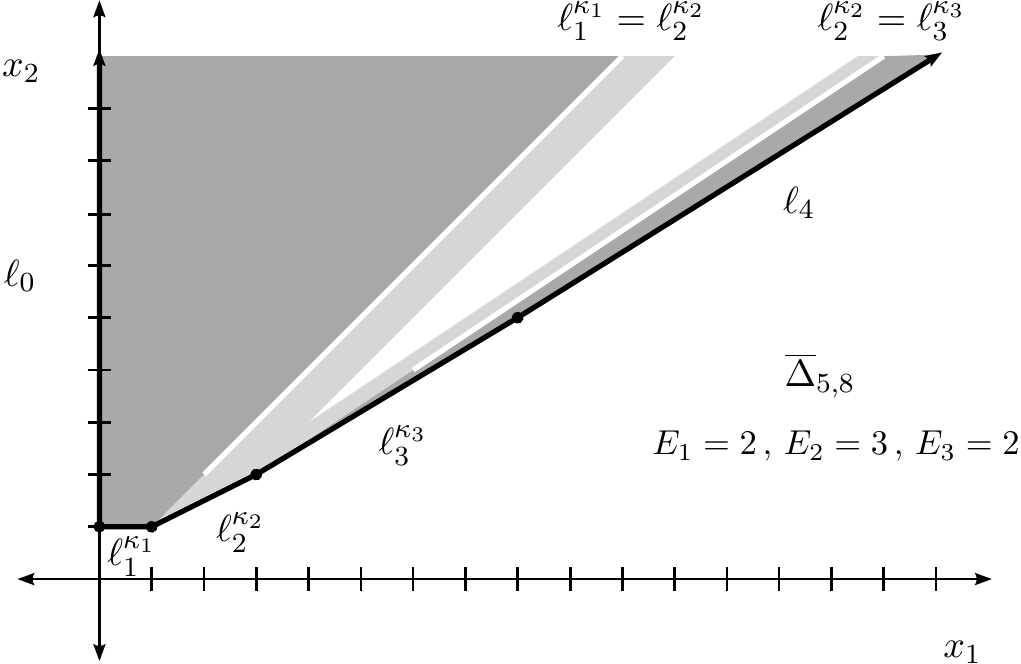}
	\end{center} 

	\caption{
	Displaceable fibers in $\ov{\De}_{5,8}(\ka)$.
	The light grey regions are displaceable by standard probes
	with direction $(1,1)$, $(3,2)$, or parallel to one of the added facets.
	The medium gray regions are displaceable by extended probes
	from Proposition~\ref{p:asympresolve}.
	Displaceability is unknown for points in the white rays and regions.}
	\label{f:5-8TR}
\end{figure}

\begin{example}\label{ex:FR5-8}
	Consider a minimal resolution of $\ov{\De}_{5,8}(\ka)$.  
	Since $(n,m) = (5,8)$ we have $(\tilde{n}, q) = (5, 3)$ and the upper
	and lower bounds in Proposition~\ref{p:asympresolve} do not coincide.
	Since the continued fraction expansion of $(5,8)$ is given by 
	$E_{1} = 2$, $E_{2} = 3$, and $E_{3} = 2$,
	the conormals for $\ov{\De}_{5,8}(\ka)$ are
	\begin{align*}
		&\eta_{0} = (1,0) &&\eta_{1} = (0,1) &&\eta_{2} = (-1,2) &&\eta_{3}=(-3,5) &&\eta_{4} = (-5,8).
	\end{align*}
	The displaceable fibers in $\ov{\De}_{5,8}(\ka)$ are displayed in Figure~\ref{f:5-8TR};
	as we can see there is an open region of unknown fibers.
\end{example}
%%%%%%%%%%%%%%%%%%%%%%%

%%%%%%%%%%%%%%%%%%%%%%%%%%%%%
%%%%%%%%%%%%%%%%%%%%%%%%%%%%%
\subsection{The weighted projective planes $\bP(1,p,q)$}\label{s:135}

Consider the weighted projective plane  $\bP(1,3,5)$, with moment polytope
\begin{equation}\label{e:P135}
	%\De = 
	\Big\{x \in \R^{2} \mid \ell_{1}(x):= x_{1} \geq 0,\,
	\ell_{2}(x):= x_{2} \geq 0,\, \ell_{3}(x):= -5x_{1} -3x_{2} + 15 \geq 0\Big\}.
\end{equation}
McDuff showed in \cite[Lemma 4.4]{Mc11} that $\De$ has an open subset of points that 
cannot be displaceable by probes; moreover, this open subset persists even
 after resolving the orbifold singularities.
Wilson--Woodward on the other hand showed in \cite[Example 4.11]{WW11} that many, but not all, 
of the fibers that cannot be displaced by probes are actually nondisplaceable in $\De$.
Figure~\ref{f:135} summarizes their results.  
Using Remarks \ref{r:nohelp} and \ref{r:extprob}(ii), one can see that one cannot do better 
by using extended probes.  In this section we  work out 
which points can be displaced by extended probes  when we resolve the singularities.

\begin{figure}[h]
	\begin{center} 
	\leavevmode 
	\includegraphics[width=2in]{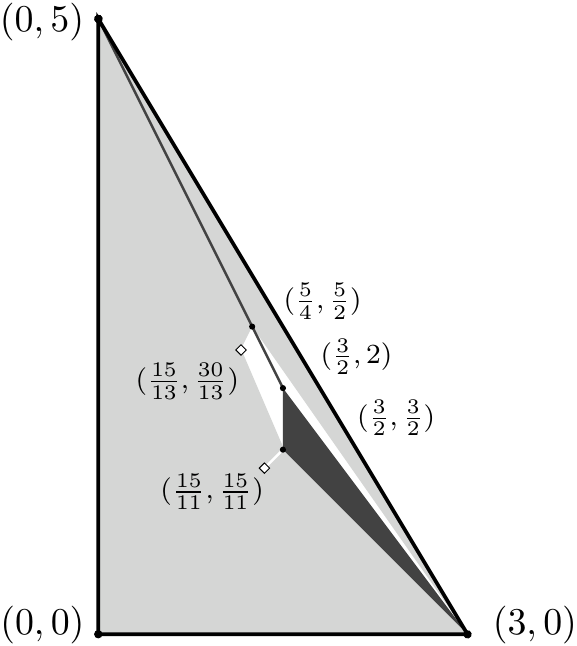}
	\end{center} 

	\caption{The moment polytope for $\bP(1,3,5)$. 
	The probe displaceable points in light gray, the nondisplaceable points in dark grey,
	and the unknown points in white.}
	\label{f:135}
\end{figure}
Observe that $\De$ near the vertex $(3,0)$ is locally equivariantly symplectomorphic
to a neighborhood of the origin in $\De_{3,5}$, and hence locally the results in Figure~\ref{f:135}
match those in Figure~\ref{f:DaNDsectors}.  
The region near $(0,5)$ in $\De$ is literally of the form $\De_{5,3}$,
which by shearing is equivalent to $\De_{2,3}$ and Theorem~\ref{t:DND} says that in $\De_{2,3}$ there is one line of nondisplaceable fibers and everything else is displaceable.  This is what we see in the region near $(0,5)$ in Figure~\ref{f:135}.

The normals for $\bP(1,3,5)$ are given by
$$ 
	\eta_{1} = (1,0)\,,\quad \eta_{2} = (0,1)\,,\quad \eta_{3} = (-5, -3)
$$
and the interior conormals for a minimal resolution of $\bP(1,3,5)$ are given by
\begin{equation}\label{e:NR135}
\eta_{4} = (-1, -1)\,,\quad \eta_{5} = (-3,-2)\,,\quad \eta_{6} = (-2, -1)\,,\quad \eta_{7} = (-1, 0).
\end{equation}
In what follows we include facets 
$$ 
F_j = \{\ell_{j}^{\ka_{j}}(x) := \ip{\eta_{j}, x} + \ka_{j} \geq 0\} 
$$
 for
$j = 4,5,6,7$ into the presentation of $\De$ from \eqref{e:P135}.  For reference let us note that these half spaces define
ghost facets in $\De$ when
$$
	\ka_{4} \geq 5\,, \quad \ka_{5} \geq 10\,,\quad \ka_{6} \geq 6\,,\quad \ka_{7} \geq 3.
$$

%%%%%%%%%%%%%%%%%%%%%%%
\subsubsection{Resolution of singularity at $(3,0)$ in $\bP(1,3,5)$}

Consider the resolution of the singularity at $(3,0)$ given by
$\De \cap \{\ell_{6}^{\ka_{6}} \geq 0\,,\,\, \ell_{7}^{\ka_{7}} \geq 0\}$, which we assume
has vertices
\begin{equation}\label{e:a456}
	(0,0),\, (0,5),\, a_{4} = \big(3(\ka_{6} -5), 5(6-\ka_{6})\big),\,
	a_{5} = (\ka_{7}, \ka_{6} - 2\ka_{7}),\, a_{6} = (\ka_{7}, 0).
\end{equation}
This resolution corresponds to a minimal resolution of 
the sector $\De_{3,5}$. 
Since 
the continued fraction expansion of $\frac{5}{3}$ is given by $(E_1,E_2) = (2,3)$,
it follows from 
Corollary~\ref{c:k1} that there is a line of nondisplaceable points near the resolved vertex
lying on the bisector of the edges $F_6,F_7$ and hence in direction $(-1,1)$.
These points cannot be displaced in 
the minimal resolution
$\ov{\De}_{3,5}$
 because, although we can deflect  a vertical probe  $P$
starting on the horizontal base facet $F_2$ by  a  $(-1,0)$ probe $Q$ starting on $F_6$,
the resulting deflected probe $\cF$ is not parallel. Rather it has a trapezoidal flag and tapers to a point as it reaches this line; cf.\ Remark~\ref{r:extprob} (ii). 
However 
because of the vertical edge $F_{1}$,
the probe $Q$ 
is symmetric  in the partial resolution of $\bP(1,3,5)$
so that 
the deflected probe $\cF$ has no flag.
Moreover, in the case $\bP(1,3,5)$ 
(but not in other $\bP(1,n,m)$),
the deflected probe $\cF$ is
symmetric, i.e.\ it exits the polygon transversally so that its direction can be reversed.   Hence this type of probe displaces all but
a codimension $1$ subset: see Figure~\ref{f:135R03DaND}.
Proposition~\ref{p:135RDaND} gives the details.

\begin{figure}[h]	

	\begin{center} 
	\leavevmode 
	\includegraphics[width=6in]{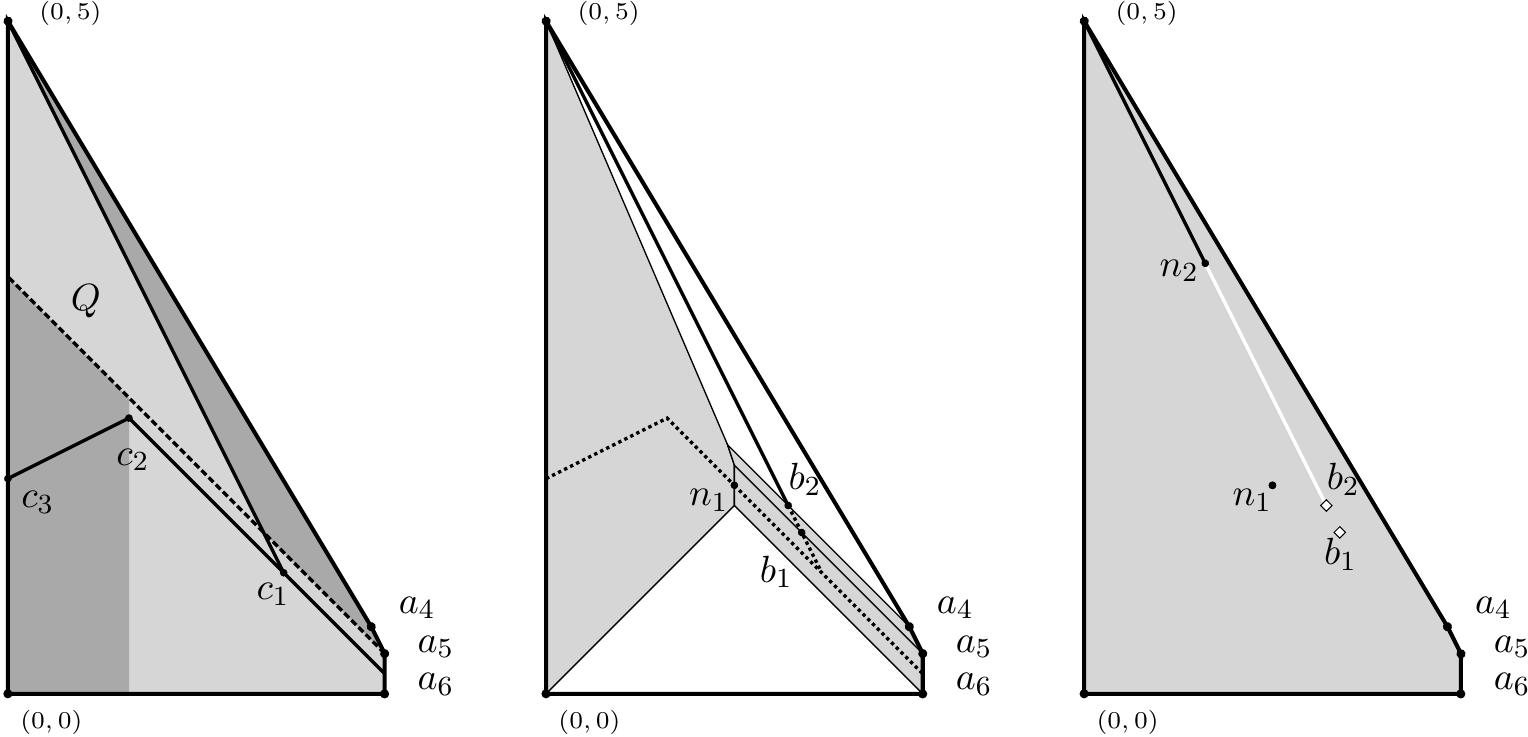}
	\end{center} 

	\caption{Displaceable and nondisplaceable fibers when the singularity
	at $(3,0)$ in $\bP(1,3,5)$ is resolved as in \eqref{e:135R03}.
	\textbf{Left}: In grey, points displaceable by symmetric extended probes deflected 
	by a  probe  $Q$,
	that is based on the edge $F_6$ with conormal $(-2,-1)$. 
	\textbf{Middle}: Displacing some more points with standard probes that have direction
	$\pm (1, -1)$.
	\textbf{Right}: In light grey are the fibers that were displaced in the previous two pictures.
	The point $n_{1}$ and the points on the black line segment connecting $(0,5)$ and $n_{2}$
	are nondisplaceable.
	Displaceability is unknown for points on the white line segment, and the points $b_{1}$, and $b_{2}$.}
	\label{f:135R03DaND}
\end{figure}

\begin{prop}\label{p:135RDaND} 
In the resolution of the singularity at $(3,0)$ in $\De$ from \eqref{e:P135} where
\begin{equation}\label{e:135R03}
\De \cap \{\ell_{6}^{\ka_{6}}(x):= -2x_{1} - x_{2} + \ka_{6} \geq 0\} \cap 
\{\ell_{7}^{\ka_{7}}(x):= -x_{1} + \ka_{7} \geq 0\}, 
\end{equation}
let $\ka_{6}, \ka_{7}$ be such that the vertices of \eqref{e:135R03} are given
by \eqref{e:a456}.
\begin{enumerate}\item [(i)]
	There is an open dense set of points in
	$\De \cap \{\ell_{6}^{\ka_{6}} \geq 0\,,\,\, \ell_{7}^{\ka_{7}} \geq 0\}$,
	whose Lagrangian fibers are displaceable by probes or extended probes.
	\item[(ii)]
	If $x$ is on the line segment connecting
	\begin{equation}\label{e:NDS1}
		(0,5) \mbox{ and } n_{2} = (\ka_{6}-5, 15-2\ka_{6})
	\end{equation}
	or the point $n_{1} = (\ka_{7}/2\,, \ka_{6}/2 - \ka_{7}/2)$
	then the Lagrangian fiber $L_{x}$ is nondisplaceable.
	\end{enumerate}
\end{prop}
\begin{proof}[Proof of Proposition~\ref{p:135RDaND}(ii)]
	If $x$ is on the segment \eqref{e:NDS1}, then for some $\ka_{4} \geq 5$ one can add a ghost facet
	 $\{\ell_{4}^{\ka_{4}}(x):= -x_{1} -x_{2} + \ka_{4} \geq 0\}$ to \eqref{e:135R03} so that
	\begin{equation}\label{e:lNDS1}
		\ell_{1}(x) = \ell_{3}(x) = \ell_{4}^{\ka_{4}}(x) \leq \ell_{6}^{\ka_{6}}(x)\,, \quad
		\ell_{1}(x) < \ell_{2}(x)\,, \quad \ell_{1}(x) < \ell_{7}^{\ka_{7}}(x)
	\end{equation}
	Using \eqref{e:lNDS1} and the change of variables $y_{1} = e^{\b_{1}}$
	and $y_{2} = e^{-\b_{1}-\b_{2}}$, the potential function with the ghost facet added is
	\begin{align*}
		W_{x,\a} = \Big(e^{\a_{1}}y_{1} + e^{\a_{2}}y_{2} + y_{1}^{-2}y_{2}^{3}& +
		y_{1}^{-1}y_{2}\,q^{\ell_{6}^{\ka_{6}}(x)- \ell_{1}(x)}\\
		& + y_{1}^{-1}y_{2}^{-1}\,q^{\ell_{2}(x)-\ell_{1}(x)}
		+ y_{1}^{-1}\,q^{\ell_{7}^{\ka_{7}}(x)-\ell_{1}(x)}\Big)\, q^{\ell_{1}(x)}.
	\end{align*}
	Hence the critical point equations at $(y_{1}, y_{2}) = (1,1)$ are
	\begin{align*}
		\partial_{y_{1}} W_{x,\a} &= 
		\Big(e^{\a_{1}} - 2 -  q^{\ell_{6}^{\ka_{6}}(x)- \ell_{1}(x)} - q^{\ell_{2}(x)-\ell_{1}(x)}
		- q^{\ell_{7}^{\ka_{7}}(x)-\ell_{1}(x)}\Big)\, q^{\ell_{1}(x)} = 0\\
		\partial_{y_{2}} W_{x,\a} &=
		\Big(e^{\a_{2}} + 3 + q^{\ell_{6}^{\ka_{6}}(x)- \ell_{1}(x)} - q^{\ell_{2}(x)-\ell_{1}(x)}\Big)\, q^{\ell_{1}(x)} = 0
	\end{align*}
	and therefore $(y_{1}, y_{2}) = (1,1)$ is a critical point of $W_{x,\a}$ when
	\begin{align*}
		e^{\a_{1}} &= 2 + q^{\ell_{6}^{\ka_{6}}(x)- \ell_{1}(x)} + q^{\ell_{2}(x)-\ell_{1}(x)} + 
			q^{\ell_{7}^{\ka_{7}}(x)-\ell_{1}(x)}
		\quad\mbox{and}\quad\\
		e^{\a_{2}} &= -3 - q^{\ell_{6}^{\ka_{6}}(x)- \ell_{1}(x)} + q^{\ell_{2}(x)-\ell_{1}(x)}.
	\end{align*}
	It follows from \eqref{e:lNDS1} that such $\a_{1}, \a_{2} \in \La_{0}$ exist.
	Hence $L_{x}$ is nondisplaceable by Theorem~\ref{t:potential}.
	Points between $n_{2}$ and $b_{1}$, see Figure~\ref{f:135R03DaND} 
	are closer to the facet $\{\ell_{6}^{\ka} = 0\}$ than any other facet,
	so we cannot prove they are nondisplaceable with a potential; 
	cf.\ Proposition~\ref{p:noghost}.
\MS

	If $x = n_{1} = (\ka_{7}/2\,, \ka_{6}/2 - \ka_{7}/2)$, then 
	\begin{equation}\label{e:lNDS2}
		\ell_{7}^{\ka_{7}}(x) = \ell_{1}(x) < \ell_{2}(x) = \ell_{6}^{\ka_{6}}(x) < \ell_{3}(x).
	\end{equation}
	By \eqref{e:lNDS2} and the change of variables $y_{1} = e^{\b_{1}}$
	and $y_{2} = e^{\b_{2}}$, the potential function is
	$$
		W_{x,\a} = \Big(e^{\a_{1}}y_{1} + y_{1}^{-1}\Big)\,q^{\ell_{1}(x)}
		+ \Big(e^{\a_{2}}y_{2} + y_{1}^{-2}y_{1}^{-1}\Big)\,q^{\ell_{2}(x)}
		+ y_{1}^{-5}y_{2}^{-3}\,q^{\ell_{3}(x)}.	
	$$
	Hence the critical point equations at $(y_{1}, y_{2}) = (1,1)$ are
	\begin{align*}
		\partial_{y_{1}} W_{x,\a} &= 
		\Big(e^{\a_{1}} - 1 -2 q^{\ell_{2}(x)-\ell_{1}(x)} - 5q^{\ell_{3}(x)-\ell_{1}(x)}\Big)\, q^{\ell_{1}(x)} = 0\\
		\partial_{y_{2}} W_{x,\a} &=
		\Big(e^{\a_{2}} - 1 - 3q^{\ell_{3}(x)-\ell_{2}(x)}\Big)\, q^{\ell_{2}(x)} = 0
	\end{align*}
	and therefore $(y_{1}, y_{2}) = (1,1)$ is a critical point of $W_{x,\a}$ when
	$$	
		e^{\a_{1}} = 1+ 2q^{\ell_{2}(x)-\ell_{1}(x)} + 5q^{\ell_{3}(x)-\ell_{1}(x)}
		\quad\mbox{and}\quad
		e^{\a_{2}} = 1 + 3q^{\ell_{3}(x)-\ell_{2}(x)}.
	$$
	By \eqref{e:lNDS2}, such $\a_{1}, \a_{2} \in \La_{0}$ exist.
	Hence $L_{x}$ is nondisplaceable by Theorem~\ref{t:potential}.
\end{proof}	 

\begin{proof}[Proof of Proposition~\ref{p:135RDaND}(i)]
\textbf{Symmetric extended probes:}
	We will displace everything except the points on the black solid lines on the left
	in Figure~\ref{f:135R03DaND}.  The points are
	$$
		c_{1} = (5-\ka_{6}/2\,, \ka_{6} - 5)\,,\quad
		c_{2} = (\ka_{6} - 5\,, 5 - \ka_{6}/2)\,,\quad\mbox{and}\quad c_{3} = (0, 15/2 - \ka_{6}).
	$$
	We will use a symmetric probe
	$Q$ based arbitrarily close to $a_{5}$ on $\{\ell_{6}^{\ka_{6}}(x) = 0\}$,
	with direction $v_{Q} = (-1,1)$ and in particular lies on the line
	$\{x_{1}+x_{2} - \ka_{6} + \ka_{7} - \e = 0\}$ for $0 < \e \ll 1$.  
	The associated reflection is given by
	$$
		A_{Q}(x_{1}, x_{2}) = (-2x_{1} - x_{2} + \ka_{6},\, 3x_{1} +2 x_{2} - \ka_{6}). 
	$$
	
	\textbf{Part 1:}			
	Let $P$ be based at $b_{P}(\l) = (\l, 0)$ for $\ka_{6} - 5 \leq \l < \ka_{7}$,
	with direction $v_{P} = (0,1)$ and form the symmetric extended probe
	$\cSP = P \cup Q \cup P'$ where
	\begin{align*}
		x_{PQ}(\l)&= (\l\,, \ka_{6} - \ka_{7} - \l + \e) && \ell(P) = \ka_{6} - \ka_{7} - \l + \e\\
		x'_{PQ}(\l) &= (\ka_{7}- \l -\e\,, \l + \ka_{6}-2\ka_{7} + 2\e) && \ell(P') = \ka_{7}- \l -\e\\
		e_{P'}(\l) &= (0\,, \ka_{6}-\l) && \ell(\cSP) = \ka_{6} - 2\l
	\end{align*}
	where the direction $v_{P'} = \widehat{A}_{Q}(v_{P}) = (-1, 2)$.  The assumption on
	$\l$ ensures that $e_{P'}$ exists the polytope on $\{\ell_{2} = 0\}$.
	One can check that $\ell(P') < \ell(P)$ so the midpoint of the extended probe always lies on $P$,
	in fact it lies on the line $\{x_{1} + x_{2} - \ka_{6}/2 = 0\}$, which appears as the line
	connecting $c_{1}$ and $c_{2}$ on the left in Figure~\ref{f:135R03DaND}.  
	
	Therefore by Theorem~\ref{t:noflag},
	every point on $P$ before the midpoint is displaceable.  Now observe that since 
	$v_{P'} = (1, -2)$ is integrally transverse to the facet $\{\ell_{1} = 0\}$,
	on which $P'$ exits the polytope, we can swap the roles of $P$ and $P'$.  
	Hence everything on these extended probes past the midpoint are displaceable as well,
	with the exception of when $e_{P'} = (0, 5)$, since then $e_{P'}$ is not on the interior of a facet.
	As $\l$ and $\e$ vary, this sweeps out the points in the regions
	\begin{align*}
		\{ x_{1} \geq \ka_{6}-5\} \cap \{ x_{1} + x_{2} - \ka_{6} + \ka_{7} \leq 0\} \quad\mbox{and}\quad\\
		\{ 2x_{1} + x_{2} - 5 < 0\} \cap \{ x_{1} + x_{2} - \ka_{6} + \ka_{7} \geq 0\},
	\end{align*}
	that are not on the line $\{x_{1} + x_{2} - \ka_{6}/2 = 0\}$.
	These are the light gray regions on the left in Figure~\ref{f:135R03DaND}.

	\textbf{Part 2:}	
	Now let $P$ be based at $b_{P}(\l) = (\l, 0)$,
	with direction $v_{P} = (0,1)$ for $0 < \l \leq \ka_{6} - 5$ and form the associated symmetric 
	extended probe
	$\cSP = P \cup Q \cup P'$ given by
	\begin{align*}
		x_{PQ}(\l)&= (\l,\, \ka_{6}-\ka_{7} - \l + \e) && \ell(P) = \ka_{6}-\ka_{7} - \l + \e\\
		x'_{PQ}(\l) &= (\ka_{7} - \l - \e,\, \l + \ka_{6} - 2\ka_{7} + 2\e) && 
		\ell(P') = 15+2\l-3\ka_{6} + \ka_{7} -\e\\
		e_{P'}(\l) &= \big(3(\ka_{6}-5-\l),\, 5(6-\ka_{6}+\l)\big) && \ell(\cP) = 15-\l -2\ka_{6}
	\end{align*}
	where the direction $v_{P'} = \widehat{A}_{Q}(v_{P}) = (-1, 2)$.  Note that the restriction
	on $\l$ is to ensure that $e_{P'}(\l)$ lies on $\{\ell_{3} = 0\}$.  One can check that
	$\ell(P') < \ell(P)$ so the midpoint of the extended probe always lies on $P$.
	
	Therefore by Theorem~\ref{t:noflag},
	every point on $P$ before the midpoint is displaceable.  Now observe that since 
	$v_{P'} = (1, -2)$ is integrally transverse to the facet $\{\ell_{1} = 0\}$,
	on which $P'$ exits the polytope, we can swap the roles of $P$ and $P'$.
	Hence everything on these extended probes past the midpoint are displaceable as well,
	with the exception of when $e_{P'} = (0, 5)$, since then $e_{P'}$ is not on the interior of a facet.
	As $\l$ and $\e$ vary, this sweeps out the points in the regions
	\begin{align*}
		\{ x_{1} \leq \ka_{6}-5\} \cap \{ x_{1} + x_{2} - \ka_{6} + \ka_{7} \leq 0\} \quad\mbox{and}\quad\\
		\{ 2x_{1} + x_{2} - 5 < 0\} \cap \{ x_{1} + x_{2} - \ka_{6} + \ka_{7} \geq 0\},
	\end{align*}
	that are not on the line segment connecting $c_{2}$ and $c_{3}$.
	These are the dark gray regions on the left in Figure~\ref{f:135R03DaND}.

\textbf{Standard probes:}
	It is straightforward to check that standard probes with directions $\pm(-1,1)$ 
	displace everything not already displaced except the points
	\begin{align}
		n_{1} &= (\ka_{7}/2\,,\,\, \ka_{6}/2 - \ka_{7}/2) &\mbox{where }
		\ell_{1} = \ell_{7}^{\ka_{7}}\,,\,\, \ell_{2} = \ell_{6}^{\ka_{6}}\\
		b_{2} &= (5-\ka_{6}+\ka_{7}\,,\,\, 2\ka_{6} - 2\ka_{7} - 5) &\mbox{where } 
		\ell_{1} = \ell_{3}\,,\,\, \ell_{6}^{\ka_{6}} = \ell_{7}^{\ka_{7}}
	\end{align}
	and the points on the line $\{2x_{1} + x_{2} = 5\}$ such that $(x - a_{4}) \cdot (1,1) \geq 0$,
	whose endpoint is denoted $b_{2}$ in Figure~\ref{f:135R03DaND}.  See the middle
	polytope in Figure~\ref{f:135R03DaND}.
\end{proof}

%%%%%%%%%%%%%%%%%%%%%%%

%%%%%%%%%%%%%%%%%%%%%%%
\subsubsection{Resolution of both singularities of $\bP(1,3,5)$}

One can carry out a similar analysis of the points in the full minimal resolution of 
$\bP(1,3,5)$.  The result is qualitatively the same:  there are $4$ isolated points that are known to be non displaceable because their qW invariants are nonzero, there are a finite number of line segments of unknown properties (more than before because there are more vertices), and otherwise everything is displaceable.
Here are the details.

The polytope is given by
\begin{equation}\label{e:135TR}
	\ov{\De}(\ka) = \{ x \in \R^{2} \mid \ell_{1} \geq 0,\, \ell_{2} \geq 0,\, \ell_{3} \geq 0,\, \ell_{j}^{\ka_{j}} \geq 	0 \mbox{ for $j=4,5,6,7$} \}
\end{equation}
where $\ell_{j}^{\ka_{j}}(x)$ are from \eqref{e:NR135}, and the support constants satisfy
$$
	0 < 5 - \ka_{4} \ll 1, \quad 0 < 10 - \ka_{5} \ll 1,\quad 0 < 6-\ka_{6} \ll 1,\quad 0 < 3-\ka_{7} \ll 1,
$$
and are such that the vertices of $\ov{\De}(\ka)$ are
\begin{equation*}
	(0,0),\, a_{1} = (0, \ka_{4}),\, a_{2} = (\ka_{5}- 2\ka_{4}, 3\ka_{4}-\ka_{5}),\, 
	a_{3} = \big( 3(10-\ka_{5}), 5(\ka_{5}-9)\big),\, a_{4},\, a_{5},\, a_{6}
\end{equation*}
where $a_{4}, a_{5}, a_{6}$ are given in \eqref{e:a456}.

\begin{figure}[h]	

	\begin{center} 
	\leavevmode 
	\includegraphics[width=6in]{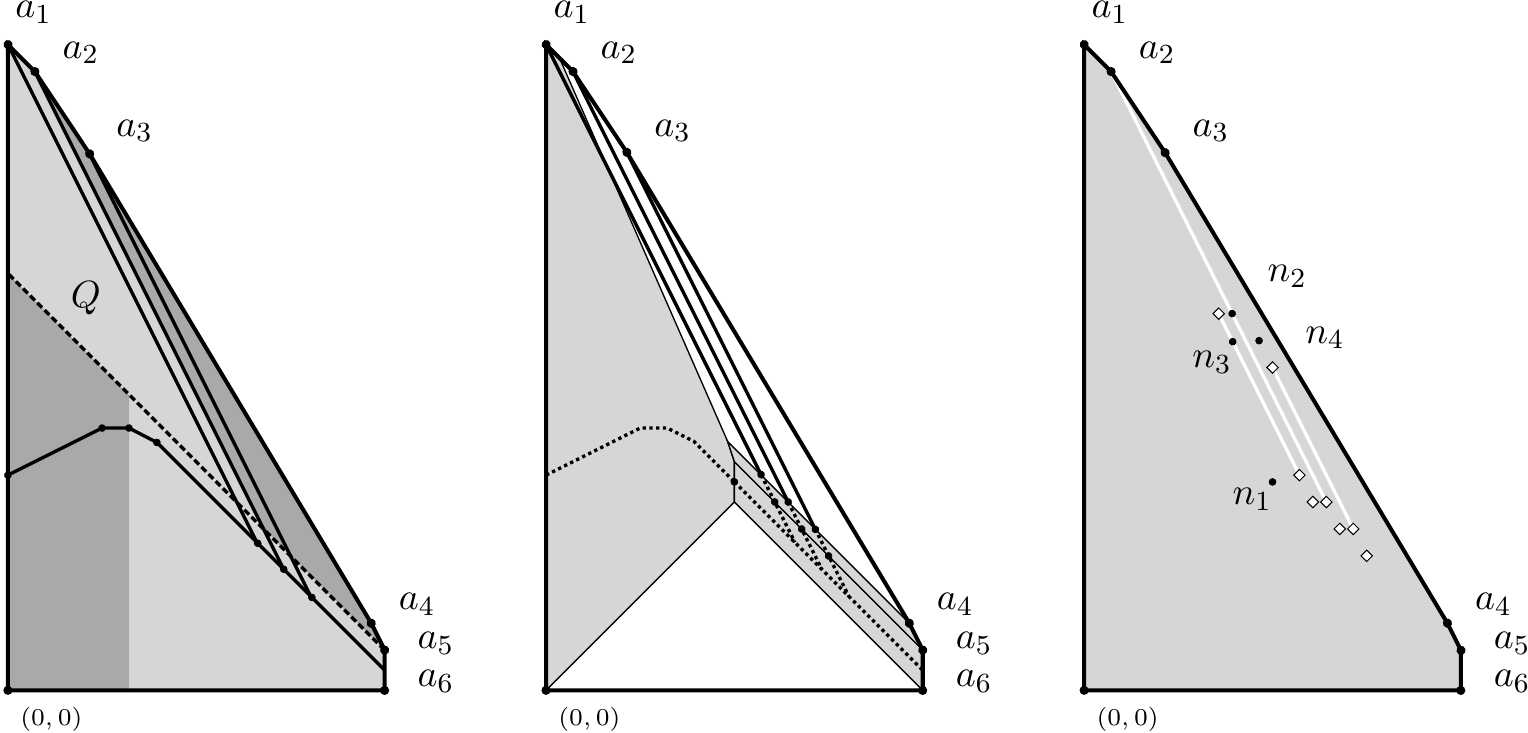}
	\end{center} 

	\caption{Displaceable and nondisplaceable fibers when both singularities
	in $\bP(1,3,5)$ are resolved as in \eqref{e:135TR}.
	\textbf{Left}: In grey, points displaceable by symmetric extended probes that use $Q$,
	which is based on $\{\ell_{6}^{\ka_{6}} = 0\}$, as the deflecting probe. 
	\textbf{Middle}: Displacing some more points with standard probes that have direction
	$\pm (1, -1)$.
	\textbf{Right}: In light grey are the fibers that were displaced in the previous two pictures.
	The white line segments are on the lines from \eqref{e:lineseg135}.
	The point $n_{1}$ is nondisplaceable just as in Proposition~\ref{p:135RDaND}.
	The point $n_{3}$ satisfies $\ell_{1} = \ell_{4}^{\ka_{4}} = \ell_{6}^{\ka_{6}}$,
	the point $n_{2}$ satisfies $\ell_{4}^{\ka_{4}} = \ell_{5}^{\ka_{5}} = \ell_{6}^{\ka_{6}}$,
	the point $n_{4}$ satisfies $\ell_{5}^{\ka_{5}} = \ell_{3} = \ell_{6}^{\ka_{6}}$,
	and all three points are nondisplaceable. 	
	Displaceability is unknown for points on the white line segments and the points
	marked with white diamonds.  See Remark~\ref{r:135TR}.}
	\label{f:135TRDaND}
\end{figure}

\begin{prop}\label{p:135TRDaND}
	In a resolution of $\bP(1,3,5)$ given by $\ov{\De}(\ka)$,
	if $x \in \ov{\De}(\ka)$ does not lie on one of the lines
	\begin{equation}\label{e:lineseg135}
		\{\ell_{1} = \ell_{4}^{\ka_{4}}\}\,,\quad \{\ell_{4}^{\ka_{4}} = \ell_{5}^{\ka_{5}}\}\,,\quad\mbox{or}\quad
		\{\ell_{5}^{\ka_{5}} = \ell_{3}\}
	\end{equation}
	and is not the point $n_{1} = (\ka_{7}/2\,, \ka_{6}/2 - \ka_{7}/2)$,
	then $L_{x}$ is displaceable by probes or extended probes.
\end{prop}
\begin{proof}
	The proof is the same as the proof of Proposition~\ref{p:135RDaND}(i).  See Figure~\ref{f:135TRDaND}.
\end{proof}

\begin{remark}\label{r:135TR}
The polytope on the right in Figure~\ref{f:135TRDaND} depicts 
 stronger displaceability results than
Proposition~\ref{p:135TRDaND}.    
They are obtained by   using the extended probes with trapezoidal flags which 
are introduced in the next section.
One can reach some points on $\ell_{1} = \ell_{4}^{\ka_{4}}$ above $n_{3}$ 
by a probe formed by deflecting  $P$ based on $\ell_{1} = 0$ 
with direction $v_{P} = (1, -1)$ by
$Q$ based on $\ell_{5}^{\ka_{5}} = 0$ with direction $v_{Q} = (1, -2)$.
Also,  points on $\ell_{5}^{\ka_{5}} = \ell_{3}$ above and slightly below $n_{4}$,
are displaceable by  probes formed by deflecting
$P$, based on $\ell_{3} = 0$ or $\ell_{6}^{\ka_{6}} = 0$ with direction $v_{P} = (2,3)$, with a probe
$Q$ based on $\ell_{5}^{\ka_{5}} = 0$ with direction $v_{Q} = (1, -2)$.
\end{remark}

\subsubsection{Resolving  $\bP(1,5,8)$}

\begin{prop}\label{p:openset}  The full minimal resolution of $\bP(1,5,8)$ has an open set of points with trivial qW invariants that cannot be displaced by extended probes.
\end{prop}
\begin{proof}
As illustrated in Figure~\ref{f:5-8TR}, the full resolution $\ov{\De}_{5,8}$ 
has an open set of  unknown points.  In  the resolution $\ov{\bP}(1,5,8)$, these would 
lie near the vertex $(5,0)$
in the region above the ray with direction $(-1,1)$.  Therefore, when we 
resolve at $(5,0)$ they would lie above 
the symmetric probes $Q$ starting on the facet $F$ with conormal $(-2,1)$.
In the case $\bP(1,3,5)$ such points were reached by probes starting on the slant edge
with direction $(1,-2)$  and then deflected by $Q$ to be vertical.  
But the corresponding probes do not exist in $\ov{\bP}(1,5,8)$ because $(1,-2)$
is not complementary to  $(5,-8)$.   Points in this region  do lie on the extensions of vertical  probes from the base that are deflected by $Q$, but they lie more than halfway along such probes.  
Therefore these points cannot be displaced. On the other hand, by Proposition~\ref{p:noghost}, there are at most finitely many points in this region with nonvanishing qW invariants. 
\end{proof}

Similarly,  the singularity of $\bP(1,8,13)$ at $(0,13)$ is modelled on $\De_{5,8}$ and has a nearby open set  of points that are not probe-displaceable.
 These arguments generalize to show that typically the resolution of 
 $\bP(1,q,p)$ has an open set of points with unknown properties.

%%%%%%%%%%%%%%%%%%%%%%%

%%%%%%%%%%%%%%%%%%%%%%%%%%%%%%%%%%%%
%%%%%%%%%%%%%%%%%%%%%%%%%%%%%%%%%%%%
%%%%%%%%%%%%%%%%%%%%%%%%%%%%%%%%%%%%
\section{Extended probes with flags: general case}\label{s:EPwFNP}

In this section we will generalize extended probes with flags, Definition~\ref{d:PEPwF} and 
Theorem~\ref{t:pdeflected}, to the case where the probe $P$ is not parallel to 
the base facet $F_{Q}$ of $Q$.

%%%%%%%%%%%%%%%%%%%%%%%%%%%%%
%%%%%%%%%%%%%%%%%%%%%%%%%%%%%
\subsection{Cautionary counterexample}\label{s:MCE}

Before diving into the more complicated notation for the non-parallel extended probes, let us first demonstrate that Theorem~\ref{t:pdeflected} is not valid as stated when the probe $P$ is not parallel to $F_{Q}$.
We will do this by showing that if it was valid, then we could displace the Clifford torus
in $\CP^{2}$, which is known to be nondisplaceable \cite{BEP04, Ch04}.
If the moment polytope for $\CP^{2}$ is given by
$$
	\De = \{x \in \R^{2} \mid x_{1} \geq 0\,,\,\, x_{2} \geq 0\,,\,\, -x_{1} -x_{2} + 6 \geq 0\}
$$
then the fiber $L_{u}$ over $u= (2,2)$ is the Clifford torus. 
What follows is similar to Remark~\ref{r:nohelp} for symmetric extended probes.

\begin{figure}[h]

	\begin{center} 
	\leavevmode 
	\includegraphics[width=2.5in]{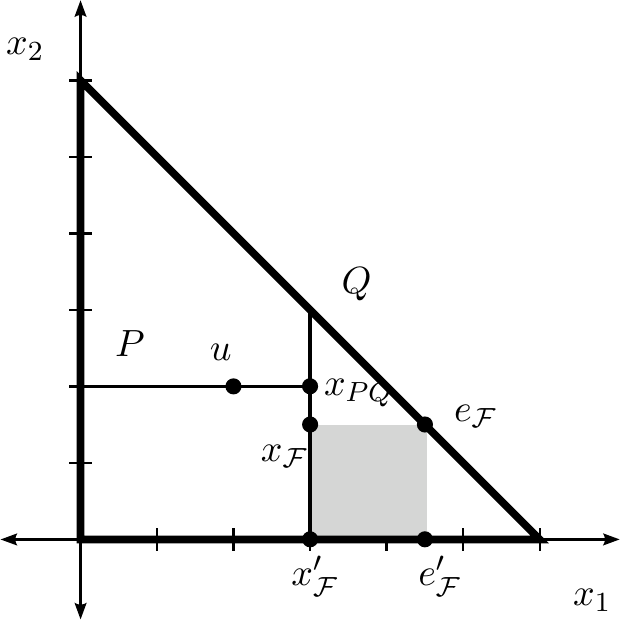}
	\end{center} 

	\caption{Illustration of Example~\ref{e:nohelpF}.  
	Why Theorem~\ref{t:pdeflected} is not valid if $P$ is not parallel to $F_{Q}$,
	the facet on which the deflecting probe is based.}
	\label{f:nohelpF}
\end{figure}

\begin{example}\label{e:nohelpF}
	Let $P$ and $Q$ be the probes where
	$$
		b_{P} = (0, 2), \quad v_{P} = (1, 0), \quad b_{Q} = (3, 3), 
		\quad v_{Q} = (0,-1).
	$$
	Form the `parallel' extended probe with flag $\cDP = P \cup Q \cup \cF$ where
	$$
		x_{PQ} = (3, 2)\,, \quad x_{\cF} = (3, \tfrac{3}{2})\,,
		\quad x_{\cF}' = (3, 0)\,, \quad \ell_{\cF} = \tfrac{3}{2}.
	$$
	We have that $P = [b_{P}, x_{PQ}]$ has length $\ell(P) = 3$ and passes through $u$.
	Since 
	$$d_{\aff}(u, F_{P}) = 2 \quad\mbox{and}\quad \ell(\cDP) = \tfrac{9}{2}$$ 
	if Theorem~\ref{t:pdeflected} applied then $\cDP$ would displace the fiber $L_{u}$.
	Of course it does not apply since $P$ is not parallel to $F_{Q}$.
\end{example}

%%%%%%%%%%%%%%%%%%%%%%%%%%%%%
%%%%%%%%%%%%%%%%%%%%%%%%%%%%%

%%%%%%%%%%%%%%%%%%%%%%%%%%%%%
%%%%%%%%%%%%%%%%%%%%%%%%%%%%%
\subsection{The definition and the displaceability method}

Despite the above failure, the parallel condition in Theorem~\ref{t:pdeflected} can be restrictive and in trying to relax it we are led to the following general notion of \emph{extended probes with flags}.

\begin{figure}[h]

	\begin{center} 
	\leavevmode 
	\includegraphics[width=3in]{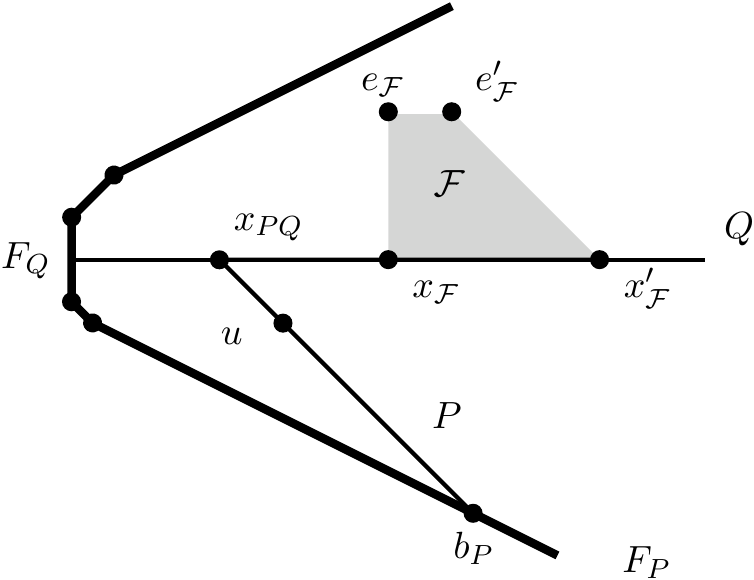}
	\end{center} 

	\caption{An extended probe with flag $\cDP = P \cup Q \cup \cF_{\mu}$,
	where the flag parameter $\mu = 0$.}
	\label{f:epNP}
\end{figure}

\begin{defn}\label{d:EPwF}
Let $P$ and $Q$ be probes in a rational polytope $\De \subset \R^{n}$ where
the probe $P$ ends at the point $x_{PQ}$ on $Q$.
The {\bf extended probe with flag} $\cDP$ formed by deflecting $P$ with $Q$ is the subset
$$
	\cP = P \cup Q \cup \cF_{\mu} \subset \De
$$
where the flag $\cF_{\mu}$ is the convex hull of the points 
$\{x_{\cF},\, x_{\cF}',\,e_{\cF},\, e'_{\cF}\}$ in $\De$.  The points $x_{\cF}$ and $x_{\cF}'$
are on $Q$, while
\begin{equation*} 
\begin{aligned}
		e_{\cF} &= x_{\cF} + \ell(\cF)\,v_{\cF_{\mu}} &\quad\mbox{where}\quad& 
		v_{\cF_{\mu}} = v_{P} - (1+\mu)\ip{\eta_{Q}, v_{P}}\,v_{Q}\\
		e'_{\cF} &= x_{\cF}' + \ell(\cF)\,v_{\cF_{\mu}}' &\quad\mbox{where}\quad& 
		v_{\cF_{\mu}}' = v_{P} - \mu\,\ip{\eta_{Q}, v_{P}}\,v_{Q}
\end{aligned}
\end{equation*}
The parameter $\mu \in [0,1]$ affects the shape of the flag, and the length of the flag
$\ell(\cF) \geq 0$ must be small enough so that $e_{\cF}$ and $e_{\cF}'$ stay in $\De$.

The {\bf  length} of the extended probe with flag $\cP$ is 
$\ell(\cP) = \ell(P) + \ell(\cF)$.
We also assume that the line segment $[x_{\cF}, e_{\cF}]$ 
does not cross the line segment $[x_{\cF}', e'_{\cF}]$, so that they are boundaries of the flag
as in Figure \ref{f:epNP}.
\end{defn}

\begin{remark}\label{r:flag}
(i) If $v_{P}$ is parallel to the facet $F_{Q}$ where
$v_{Q}$ is based, then $\ip{\eta_{Q}, v_{P}} = 0$ and hence by \eqref{e:PFL}, the shape of the flag 
$\cF$ is independent of the parameter $\mu$.  In this case we recover the definition of a parallel extended probe with flag, Definition~\ref{d:PEPwF}.

(ii) When $v_{P}$ is not parallel to the facet $F_{Q}$ the parameter $\mu$ affects the shape of the flag.
For the hyperplane $H_{Q} = \{x \in \R^{n} \mid \ip{\eta_{Q}, x} = 0\}$,
the projection $\pi_{Q}\co\R^{n} \to H_{Q}$ along $v_{Q}$ is 
$$
\pi_{Q}(w) = w- \ip{\eta_{Q}, w}\,v_{Q}
$$
and the reflection $r_{Q}\co\R^{n} \to \R^{n}$ across $H_{Q}$ via $v_{Q}$ is
$$
r_{Q}(w) = w- 2\ip{\eta_{Q}, w}\,v_{Q}.
$$
So as $\mu \in [0,1]$ varies the shape of the flag $\cF_{\mu}$ linearly interpolates between
\begin{align*}
	\cF_{0}:& \mbox{ where $v_{\cF_{0}} = \pi_{Q}(v_{P})$ and $v_{\cF_{0}}' = v_{P}$}\,,\quad\mbox{and}\\
	\cF_{1}:& \mbox{ where $v_{\cF_{1}} = r_{Q}(v_{P})$ and $v_{\cF_{1}}' = \pi_{Q}(v_{P})$}.
\end{align*}
\end{remark}

\begin{figure}[h]

	\begin{center} 
	\leavevmode 
	\includegraphics[width=5in]{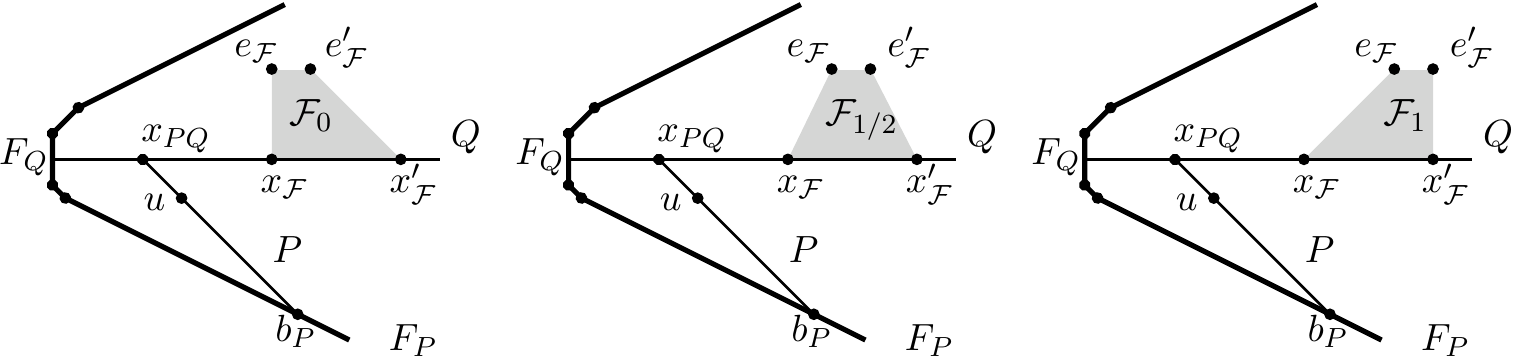}
	\end{center} 

	\caption{Extended probes with flags $\cP = P \cup Q \cup \cF_{\mu}$
	for varying flag parameters, as per Remark~\ref{r:flag}(ii).}
	\label{f:FlagParameter}
\end{figure}

The following theorem explains how one can use extended probes with flags to displace Lagrangian torus fibers. 
 
\begin{thm}\label{t:deflected}
	Let $\cP = P \cup Q \cup \cF$ be an extended probe with a flag constructed from probes 
	$P$ and $Q$ as above,
	in a moment polytope $\De = \Phi(M)$ for a toric symplectic orbifold 
	$(M^{2n}, \w, \T, \Phi)$.  
	
	For a point $u \in \De$, if $u$ is in the interior of $P$, the affine distance from $u$ to the facet $F_{P}$
	satisfies
	$$
		d_{\aff}(u, F_{P}) < \tfrac{1}{2}\,\ell(\cP)\,
	$$ 
	and the flag $\cF$ satisfies both inequalities
	\begin{equation}\label{e:flag}
	d_{\aff}(x_{PQ}, F_{Q}) < d_{\aff}(x_{\cF}, x_{\cF}') \quad\mbox{and}\quad 
	\ell(\cF) < d_{v_{P}}(x_{PQ}, F_{Q}),
	\end{equation}
	then the Lagrangian fiber $L_{u} = \Phi^{-1}(u)$	is displaceable. 
\end{thm} 

The second condition in \eqref{e:flag} did not appear in Theorem~\ref{t:pdeflected}, for there it is
trivially satisfied since if $v_{P}$ is parallel to $F_{Q}$, then $d_{v_{P}}(x_{PQ}, F_{Q}) = \infty$.
Compare the following remark with Remark~\ref{r:nohelp}.

\begin{remark}\label{r:extprob}
	(i) If $F_{P}$ and $F_{Q}$ were the only facets in the polytope, then $d_{v_{P}}(b_{P}, F_{Q})$ 
	represents the maximum length $P$ could be extended to before it hit the facet $F_{Q}$.
	The second condition in \eqref{e:flag} implies that this maximum length is an {\it a priori} upper bound
	\begin{equation}\label{e:aub}
		\ell(\cP) < d_{v_{P}}(b_{P}, F_{Q}).
	\end{equation}
	on the length $\ell(\cP)$ for an extended probe with flag 
	$\cP$ formed with probes $P$ and $Q$.
	
	(ii) This a priori upper bound \eqref{e:aub} has the following consequence: 
	Suppose a probe $P$ exits the polytope $\De$ through the facet
	$F$.  Then the displaceability results given by using $P$ as a standard probe cannot be improved on by
	using Theorem~\ref{t:deflected} with a deflecting probe $Q$ based on $F$.
	
	(iii) In the counter-example in Section~\ref{s:MCE}, the deflecting probe $Q$ is based on
	the facet through which the probe $P$ would exit the polytope.  So in this example,
	\eqref{e:aub} is violated and hence the second condition in \eqref{e:flag} is as well.
\end{remark}
%%%%%%%%%%%%%%%%%%%%%%%%%%%%%
%%%%%%%%%%%%%%%%%%%%%%%%%%%%%

%%%%%%%%%%%%%%%%%%%%%%%%%%%%%
%%%%%%%%%%%%%%%%%%%%%%%%%%%%%
\subsection{Resolution of a finite volume $A_{n}$-singularity}\label{fvAnR}

Let us consider a minimal resolution of an $A_{2}$-singularity that now has finite volume, so the moment polytope is given by
\begin{equation}\label{e:fvAnR}
	\Tilde{\De}_{2,3}(\ka) = \{x \in \R^{2} \mid \ell^{v}(x) \geq 0,\,
	\ell^{s}(x) \geq 0,\, \ell_{\infty}(x) > 0,\, 
	\ell_{1}^{\ka_{1}}(x) \geq 0,\, \ell_{2}^{\ka_{2}} \geq 0\} 
\end{equation}
where the finite volume $A_{2}$-singularity is defined by
$$
	\ell^{v}(x):=x_{1}\,,\,\, \ell^{s}(x):= -2x_{1} + 3x_{2}\,,\,\, \ell_{\infty}(x) := -x_{2} + 2
$$
and the minimal resolution at the origin uses
$$ 
	\ell_{1}^{\ka_{1}}(x):= x_{2} - \ka_{1}\,,\,\, \ell_{2}^{\ka_{2}}(x):= -x_{1} + 2x_{2} - \ka_{2}.
$$ 
Figure~\ref{f:fvAnR} depicts the polytope $\Tilde{\De}_{2,3}(\ka)$ when
$$ 
0 < \ka_{1} < 2\,,\quad 0 < \ka_{2} < 1\,,\quad \ka_{2} < 2\ka_{1} \quad\mbox{and}\quad \ka_{1} < 2 \ka_{2}.
$$ 

In this example there remains an open region of unknown points, even after using extended probes.

\begin{figure}[h]
	\begin{center} 
	\leavevmode 
	\includegraphics[width=4in]{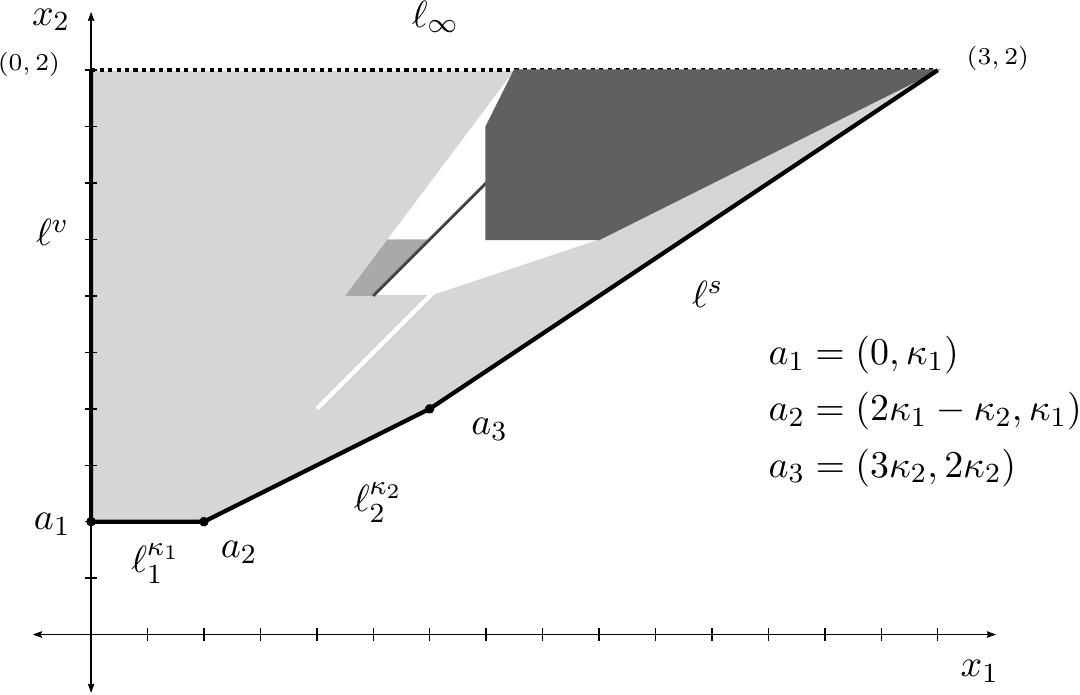}
	\end{center} 

	\caption{Displaceable and nondisplaceable fibers in $\Tilde{\De}_{2,3}(\ka)$.
	In light gray are points displaceable by probes, in medium gray are points displaceable
	by extended probes with trapezoidal flags, and in dark grey are nondisplaceable points.}	
	\label{f:fvAnR}
\end{figure}

\begin{prop}\label{p:fvAnR}
	In the resolution of the finite volume $A_{2}$-singularity given in \eqref{e:fvAnR}, then in
	the polytope $\Tilde{\De}_{2,3}(\ka)$:
	\begin{enumerate}
	\item[(i)] The Lagrangian fiber $L_{x}$ is displaceable by probes if $x$ is in one of
	the regions
	$$
		\{ 4x_{1} - 3x_{2} < 0\}\,,\,\, \{ x_{2} < 1 + \ka_{1}/2\}\,,\,\,
		\{-x_{1} + 3x_{2} < 2+ \ka_{2}\}\,,\,\,
		\{-x_{1} + 2x_{2} < 2\} 
	$$
	and not on the line segment $\{x_{1} - x_{2} = \ka_{1} - \ka_{2}\,,\,\,x_{2} \geq 2 \ka_{2}\}$.  
	\item[(ii)]  The Lagrangian fiber $L_{x}$ is displaceable by  
	an extended probe
	with trapezoidal flag if $x$ is in the region
	$$
		\{x_{2} < 1+ \ka_{2}\,,\,\, x_{1} - x_{2} < \ka_{2}/2\}.
	$$

	\item[(iii)] The Lagrangian fiber $L_{x}$ is nondisplaceable if $x$ is in the region
	$$
		\{ 2x_{1} - x_{2} \geq 1\,,\,\, x_{1} \geq 1+\ka_{1}\,,\,\, x_{2} \geq 1 + \ka_{2}\,,\,\,
		-x_{1} + 2x_{2} \geq 1\}
	$$
	or on the line segment $\{x_{1} - x_{2} = \ka_{2}/2\,,\,\, x_{2} \geq 1+\ka_{2}/2\}$.
	\end{enumerate}
\end{prop}
\begin{proof}
	Part (i) is straightforward with using probes with direction $(1,0)$ on $\{\ell^{v} = 0\}$,
	direction $(1,1)$ on $\{\ell_{1}^{\ka_{1}} = 0\}$ and $\{\ell_{2}^{\ka_{2}} = 0\}$,
	direction $(-1,0)$ on $\{\ell_{2}^{\ka_{2}} = 0\}$
	and direction $(-2, -1)$ on $\{\ell^{s} = 0\}$.
	
	Part (iii) is also straightforward.  For the points in the region one uses the ghost
	facets $\{\ell_{\infty}^{\ka_{\infty}}(x):= -x_{2}+\ka_{\infty}\geq 0\}$ and 
	$\{\ell_{g}^{\ka_{g}}(x):= -x_{1} + x_{2} + \ka_{g} \geq 0\}$ for varying $\ka_{\infty} \geq 2$
	and $\ka_{g} \geq 1$.
	
\begin{figure}[h]

	\begin{center} 
	\leavevmode 
	\includegraphics[width=4in]{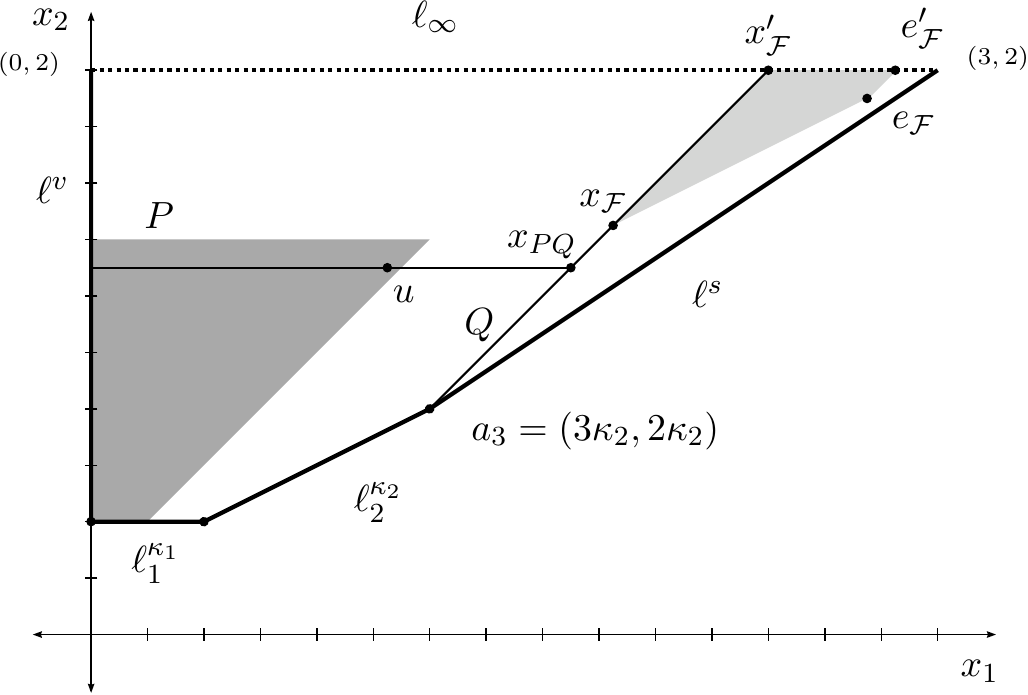}
	\end{center} 

	\caption{The non-parallel extended probe with flag from part (ii) of Proposition~\ref{p:fvAnR}
	displacing the point $u$.
	Points in the dark grey region can be displaced with these types of extended probes.}	
	\label{f:fvAnREP}
\end{figure}
	
	For part (ii), let $Q$ be a probe with direction $v_{Q} = (1,1)$ based on
	$\{\ell_{2}^{\ka_{2}} = 0\}$ arbitrary close to $a_{3} = (3\ka_{2}, 2\ka_{2})$.
	Let $P$ be based at $b_{P} = (0, \l)$ for $\ka_{1} < \l < 1 + \ka_{2}$, with direction $v_{P} = (1,0)$,
	and form the extended probe with flag $\cP = P \cup Q \cup \cF_{0}$ where
	$\mu= 0$ is the flag parameter.  The flag is given by
	\begin{align*}
		x_{PQ}&= (\l+\ka_{2}, \l) && \ell(P) = \l + \ka_{2}\\
		x_{\cF}&= (2-\l + 3\ka_{2}, 2+2\ka_{2} - \l) && x_{\cF}'= (2+\ka_{2}, 2)\\
		e_{\cF}&= x_{\cF} + \ell(\cF)(2, 1) && e_{\cF}' = x_{\cF}' + \ell(\cF)(1,0)\\
		\ell(\cF)&< \l - 2\ka_{2} && \ell(\cP) < 2\l - \ka_{2},
	\end{align*}
	where the upper bound on $\ell(\cF)$ comes from the second condition in
	\eqref{e:flag} and $x_{\cF}$ should be moved slightly closer to $x_{PQ}$ 
	so that the first condition in \eqref{e:flag} is satisfied.
	By Theorem~\ref{t:deflected}, this extended probe displaces everything
	on $P$ between $b_{P} = (0,\l)$ and $(\l - \ka_{2}, \l)$.  As $\l$ ranges
	over $\ka_{1} < \l < 1 + \ka_{2}$, these extended probes displace precisely the region in (ii).	
\end{proof}

%%%%%%%%%%%%%%%%%%%%%%%%%%%%%
%%%%%%%%%%%%%%%%%%%%%%%%%%%%%

%%%%%%%%%%%%%%%%%%%%%%%%%%%%%%%%%%%%
%%%%%%%%%%%%%%%%%%%%%%%%%%%%%%%%%%%%
%%%%%%%%%%%%%%%%%%%%%%%%%%%%%%%%%%%%

%%%%%%%%%%%%%%%%%%%%%%%%%%%%%%%%%%%%
%%%%%%%%%%%%%%%%%%%%%%%%%%%%%%%%%%%%
%%%%%%%%%%%%%%%%%%%%%%%%%%%%%%%%%%%%
\section{Proofs of results about probes}
\label{s:proofs}

Consider the symplectic form $\w_{0} = \tfrac{1}{\pi}\, dx \wedge dy$ on $\C$.  For this
symplectic form, the standard $\T^{1} = S^{1} = \R/\Z$ action on $\C$ by $t \cdot z = e^{2\pi i t}z$
is given by the moment map $\Phi_{0}\co\C \to \R$ where $\Phi_{0}(z) = \abs{z}^{2}$.
The symplectic form $\w_{0}$ is also normalized so that 
$$
	\int_{\bD(a)} \om_0= a
$$
where $\bD(a)$ is the disc
$$
	\bD(a): = \{z \in \C : \abs{z}^{2} \leq a\} \subset (\C, \w_{0}).
$$
We will denote its boundary by
$$
	S^{1}(a) = \partial \bD(a) = \{z \in \C : \abs{z}^{2} = a\},
$$
and the annulus by
$$ 
	\bA(b,c) = \bD(c) \setminus \Int \bD(b) \quad\mbox{for $0 \leq b < c$}.
$$ 

For each theorem we have an extended probe $\cP$ in a toric symplectic orbifold $(M^{2n}, \w, \T, \Phi)$
and our goal is to displace the Lagrangian torus fiber $L_{u} = \Phi^{-1}(u)$.  To displace $L_{u}$, 
it suffices to build an embedding
\begin{equation}\label{e:FE}
	\psi_{\cP}\co\bD(\ell) \times \T^{n-1} \to (M^{2n}, \w)
\end{equation}
such that for some $a < \tfrac{1}{2}\, \ell$
\begin{equation}\label{e:EC}
	\psi_{\cP}^{*}\,\w = \pi^{*}\w_{0} \quad\mbox{and}\quad \psi_{\cP}(S^{1}(a) \times \T^{n-1}) = L_{u} 
\end{equation}
where $\pi\co\bD(\ell) \times \T^{n-1} \to \bD(\ell)$ is the projection.  Since $a < \tfrac{1}{2}\,\ell$ there is
a Hamiltonian isotopy of $\bD(\ell)$ supported away from the boundary that displaces $S^{1}(a)$,
and therefore the embedding can be used to extend this to a Hamiltonian isotopy of $(M, \w)$ that
displaces $L_{u}$.

While the precise details for building $\psi_{\cP}$ vary depending on the type of extended probe,
the following outline describes the general process.  Here $P, Q, \cE$ are the three parts
of the extended probe $\cP$, where $\cE$ is either 
$P'$, $\cF$, or $\cF_{\mu}$ depending on the type of extended probe.
\begin{description}
\item[Stage 1] For $\ell_{P} = \ell(P)$, produce an embedding
\begin{equation}\label{e:psiP}
	\psi_{P}\co\bD(\ell_{P}) \times \T^{n-1} \to (M^{2n}, \w) 
	\quad\mbox{with  } \psi_{P}^{*}\,\w = \pi^{*}\w_{0}
	\mbox{  and  } \im(\psi_{P}) \subset \Phi^{-1}(P).
\end{equation}
Except for the second version of Theorem~\ref{t:noflag}, 
we have $u \in \Int P$ and we will show
$$ 
	\psi_{P}(S^{1}(a) \times \T^{n-1}) = L_{u}
$$ 
for $a = d_{\aff}(u, F_{P})$ where by assumption $a < \tfrac{1}{2}\, \ell$.

\item[Stage 2]
For $\ell = \ell(\cP)$, produce an embedding
\begin{equation}\label{e:psiE}
	\psi_{\cE}\co\bA(\ell_{P}, \ell) \times \T^{n-1} \to (M^{2n}, \w) 
	\quad\mbox{with  } \psi_{\cE}^{*}\,\w = \pi^{*}\w_{0}
	\mbox{  and  } \im(\psi_{\cE}) \subset \Phi^{-1}(\cE).
\end{equation}
In the second version of Theorem~\ref{t:noflag},
we have  $u \in \Int P'$ and we will show
$$ 
	\psi_{P'}(S^{1}(a) \times \T^{n-1}) = L_{u}
$$ 
for $a = d_{\aff}(u, x_{PQ}') + \ell_{P}$ where by assumption $a < \tfrac{1}{2}\, \ell$.

\item[Stage 3]
Use the deflecting probe $Q$ to build a symplectomorphism $\Psi$ of $(M^{2n}, \w)$ 
such that $\Psi \circ \psi_{P}$ and $\psi_{\cE}$ glue together to form an embedding
$$ 
	\psi_{\cP}\co\bD(\ell) \times \T^{n-1} \to (M^{2n}, \w)
$$ 
that satisfies \eqref{e:EC}.  Since the fiber $L_{u}$ is disjoint from $\Phi^{-1}(Q)$,
to ensure $\psi_{\cP}$ satisfies the second condition in
\eqref{e:EC} it suffices to prove that $\Psi$ can be built to be supported
in any given neighborhood of $\Phi^{-1}(Q) \subset M$.
\end{description}

%%%%%%%%%%%%%%%%%%%%%%%%%%%%%
%%%%%%%%%%%%%%%%%%%%%%%%%%%%%
\subsection{Action-angle coordinates}

The canonical symplectic form $d\l$ on $T^{*}\T = \ft^{*} \times \T$ is
\begin{equation}\label{e:CSF}
d\l\big((v,\eta), (v', \eta')\big) = \ip{\eta', v} - \ip{\eta, v'}.
\end{equation}
If $\{\eta_{1}, \dots, \eta_{n}\} \subset \ft_{\Z}$ and $\{v_{1}, \dots, v_{n}\} \subset \ft_{\Z}^{*}$
are dual bases, then in the associated coordinates 
$(x,\th) = (x_{1}, \dots, x_{n}, \th_{1}, \dots, \th_{n}) \in \ft^{*} \times \T$
the canonical  symplectic form is $d\l = dx \wedge d\th$ and it is clear
that the projection $\pi_{\ft^{*}}\co\ft^{*} \times \T \to \ft^{*}$ is the moment map for the obvious $\T$-action.
Now let $\De \subset \ft^{*}$ be the moment polytope
for a symplectic toric orbifold $(M^{2n}, \w, \T, \Phi)$. Then $(M^{2n}, \w, \T, \Phi)$ can be modeled by
$$
	(\De \times \T, d\l, \T, \pi_{\ft^{*}})
$$
by performing a symplectic cut along $F \times \T \subset \De \times \T$ for each facet $F \subset \De$.   
If $f \in \Int(F)$, then this amounts to replacing $f \times \T$ with $f \times \T/\T^{1}_{\eta_{F}}$
where $\T^{1}_{\eta_{F}} \subset \T$ is the circle generated by $F$'s
primitive interior conormal $\eta_{F} \in \ft_{\Z}$.  In this way we can consider 
the action-angle coordinates $(x, \th) \in \De \times \T$ as a global
coordinate system on $(M, \w, \T, \Phi)$.

For example the disk $(\bD(a), \w_{0})$ has action-angle coordinates
$(s, \phi) \in [0, a] \times \T^{1}$ where the circle $\{(0, \phi) \mid \phi \in \T^{1}\}$ is collapsed to a point.
The explicit identification $(s, \phi) \mapsto e^{2\pi i \phi} \sqrt{s}$ pulls 
$\w_{0}$ back to $ds \wedge d\phi$.  Likewise the annulus $(\bA(b,c), \w_{0})$ has
action-angle coordinates $(s, \phi) \in [b,c] \times \T^{1}$ with $\w_{0} = ds \wedge d\phi$.

%%%%%%%%%%%%%%%%%%%%%%%
\subsubsection{Coisotropic embeddings from probes}

Let $P$ be a probe with direction $v_{P} \in \ft_{\Z}^{*}$, length $\ell_{P}$, and based at the point $b_{P}$
on the interior of a facet $F_{P}$, which has primitive interior conormal $\eta_{F_{P}} \in \ft_{\Z}$.
Since $v_{P}$ is integrally transverse to $\eta_{F_{P}}$ and inward pointing, there is a lattice basis
for $\ft_{\Z}$ of the form $\{\eta_{1}' = \eta_{F_{P}}, \eta_{2}', \dots, \eta_{n}'\}$ where
\begin{equation}\label{e:LB}
	\ip{\eta_{1}', v_{P}}=1 \quad\mbox{and}\quad \ip{\eta_{k}', v_{P}} = 0\quad\mbox{for $k\geq 2$}.
\end{equation}
Using the model $(\De \times \T, d\l, \T, \pi_{\ft^{*}})$ for $(M, \w, \T, \Phi)$ define the embedding
$$ 
	\psi_{P}\co\bD(\ell_{P}) \times \T^{n-1} \to (M^{2n}, \w)
$$ 
where if $(s, \phi_{1}, \dots, \phi_{n})$ are coordinates for $\bD(\ell_{P}) \times \T^{n-1}$
such that $(s, \phi_{1})$ are action-angle coordinates for $(\bD(\ell_{P}), \w_{0})$, then $\psi_{P}$ is given by
\begin{equation}\label{e:EAAC}
	\psi_{P}(s, \phi_{1}, \phi_{2}, \dots, \phi_{n}) =
	(b_{P}+s\,v_{P}\,,\,\, \phi_{1}\eta_{F_{P}} + \phi_{2}\eta_{2}' + \dots + \phi_{n}\eta_{n}')
	 \in \De \times \T.
\end{equation}
This map is well-defined since for fixed $(\phi_{2}, \dots, \phi_{n}) \in \T^{n-1}$, the image of the map
$$
	\phi_{1} \mapsto \psi_{P}(0, \phi_{1}, \phi_{2}, \dots, \phi_{n}) = 
	(b_{P}, \phi_{1}\eta_{F_{P}} + \phi_{2}\eta_{2}' + \dots + \phi_{n}\eta_{n}')
$$ 
lies in $F_{P} \times \T$, which in $(M, \w)$ is replaced with $F_{P} \times \T/\T^{1}_{\eta_{F_{P}}}$.
By design
$$ 
	\psi_{P}^{*}\,\w = \pi^{*}\w_{0} \quad\mbox{and}\quad
	\psi_{P}(S^{1}(a) \times \T^{n-1}) = L_{u(a)}
$$ 
as in \eqref{e:EC}, where $u(a) = b_{P} + a\,v_{P}$ is the point on $P$ such that $d_{\aff}(u(a), b_{P}) = a$.
This embedding $\psi_{P}$ will serve as \eqref{e:psiP} in Stage 1 for all the extended probe theorems.
%%%%%%%%%%%%%%%%%%%%%%%

%%%%%%%%%%%%%%%%%%%%%%%
\subsubsection{Coisotropic embeddings from rational line segments}

Let $S \subset \De$ be a rational line segment starting at $b_{S}$ and ending at $e_{S}$ in $\Int \De$,
with length $\ell_{S} = d_{\aff}(b_{S}, e_{S})$ and direction $v_{S} \in \ft_{\Z}^{*}$.  Let
$\{\eta_{1}', \eta_{2}', \dots, \eta_{n}'\}$ be an integral basis for $\ft_{\Z}$ that satisfies
\eqref{e:LB} with respect to $v_{S}$.  Then similarly to the case of a probe,
using the model $(\De \times \T^{n}, d\l, \T, \pi_{\ft^{*}})$ for $(M, \w, \T, \Phi)$ we can define an embedding
$$
	\psi_{S}\co\bA(b, b+ \ell_{S}) \times \T^{n-1} \to (M, \w)
$$
for any $b > 0$, such that
\begin{equation}\label{e:EAAAC}
	\psi_{S}(s, \phi_{1}, \phi_{2}, \dots, \phi_{n}) = (b_{S} + (s-b)v_{S}\,,\,\, 
	\phi_{1}\eta_{1}' + \phi_{2}\eta_{2}' + \dots + \phi_{n}\eta_{n}')
	 \in \De \times \T
\end{equation}
where $(s, \phi_{1}) \in [b, b+\ell_{S}] \times \T^{1}$ are action-angle coordinates for 
$(\bA(b, b+\ell_{P}), \w_{0})$.  Again we have
$$ 
	\psi_{S}^{*}\,\w = \pi^{*}\w_{0} \quad\mbox{and}\quad
	\psi_{S}(S^{1}(a) \times \T^{n-1}) = L_{u(a)}
$$ 
where $u(a) = b_{S} + (a-b)\,v_{S}$ is the point on $S$ such that $d_{\aff}(u(a), b_{S}) = a-b$.
For $S = P'$, this embedding $\psi_{P'}$ will serve as \eqref{e:psiE} in Stage 2 for Theorem~\ref{t:noflag}.
%%%%%%%%%%%%%%%%%%%%%%%

%%%%%%%%%%%%%%%%%%%%%%%%%%%%%
%%%%%%%%%%%%%%%%%%%%%%%%%%%%%

%%%%%%%%%%%%%%%%%%%%%%%%%%%%%
%%%%%%%%%%%%%%%%%%%%%%%%%%%%%
\subsection{Proving Theorem~\ref{t:noflag}: Symmetric extended probes}

Let $\cP = P \cup Q \cup P'$ be a symmetric extended probe.
Let $A_{Q}, \widehat{A}_{Q}\co\ft^{*} \to \ft^{*}$ be the affine and linear reflections from
\eqref{e:Areflect} and \eqref{e:reflect} associated to the symmetric probe $Q$.

%%%%%%%%%%%%%%%%%%%%%%%
\subsubsection{Stage 1 and 2 for Theorem~\ref{t:noflag}}\label{s:S12NF}

The probe $P$ has length $\ell_{P}$, direction $v_{P} \in \ft_{\Z}^{*}$, starts at $b_{P} \in \Int F_{P}$, ends at 
the point $x_{PQ} = b_{P} + \ell_{P}v_{P} \in \Int Q$.  For a choice of lattice basis 
$\{\eta_{F_{P}}, \eta_{2}', \dots, \eta_{n}'\}$ for $\ft_{\Z}$ that satisfies \eqref{e:LB}, define the embedding for Stage 1
$$ 
	\psi_{P}\co\bD(\ell_{P}) \times \T^{n-1} \to (M, \w)
$$ 
so that for $(s, \phi_{1}) \in [0, \ell_{P}] \times \T^{1}$,
\begin{equation}\label{e:NFS1C}
	\psi_{P}(s, \phi_{1}, \phi_{2}, \dots, \phi_{n}) =
	\left(b_{P}+s\,v_{P}\,,\,\, \phi_{1}\eta_{F_{P}} + \sum_{k=2}^{n}\phi_{k}\eta_{k}'\right)
\end{equation}
as in \eqref{e:EAAC}.

The rational line segment $P'$ has length $\ell_{P'}$, direction 
$v_{P'} = \widehat{A}_{Q}(v_{P}) \in \ft_{\Z}^{*}$, starts
at the point $x_{PQ}' = A_{Q}(x_{PQ}) \in \Int Q$, and ends at $e_{P'} = x_{PQ}' + \ell_{P'}v_{P'} \in \Int \De$.  Since $\widehat{A}_{Q}$ is an element of $GL(\ft^{*}_{\Z})$, the image under 
$\widehat{A}_{Q}^{*}$, the dual of $\widehat{A}_{Q}$, of the lattice basis of $\ft_{\Z}$ used for $\psi_{P}$:
$$
\big\{\widehat{A}_{Q}^{*}(\eta_{F_{P}}), \widehat{A}_{Q}^{*}(\eta'_{2}), 
\dots, \widehat{A}_{Q}^{*}(\eta'_{n})\big\} \in \ft_{\Z}
$$
is still a lattice basis. This new basis satisfies \eqref{e:LB} with respect to 
$v_{P'} =\widehat{A}_{Q}(v_{P})$ since $(\widehat{A}_{Q})^{2} = \id$. 
Define the embedding for Stage 2, where $\ell = \ell_{P} + \ell_{P'}$,
$$ 
	\psi_{P'}\co\bA(\ell_{P}, \ell) \times \T^{n-1} \to (M,\w)
$$ 
so that for $(s, \phi_{1}) \in [\ell_{P}, \ell] \times \T^{1}$,
\begin{equation}\label{e:NFS2C}
\begin{split}
	\psi_{P'}(s, \phi_{1}, \phi_{2}, \dots, \phi_{n}) &= \left(x_{PQ}' + (s-\ell_{P})v_{P'}\,,\,\, 
	\phi_{1}\widehat{A}_{Q}^{*}(\eta_{F_{P}}) + \sum_{k=2}^{n}\phi_{k}\widehat{A}_{Q}^{*}(\eta'_{k})\right)\\
	&= \left(A_{Q}(x_{PQ} + (s-\ell_{P})v_{P})\,,\,\, \widehat{A}_{Q}^{*}
	\big(\phi_{1}\eta_{F_{P}} + \sum_{k=2}^{n} \phi_{k}\eta_{k}'\big) \right)
\end{split}
\end{equation}
as in \eqref{e:EAAAC}. 
%%%%%%%%%%%%%%%%%%%%%%%

%%%%%%%%%%%%%%%%%%%%%%%
\subsubsection{Stage 3 for Theorem~\ref{t:noflag}}

Recall that for our symmetric probe $Q \subset \De$, we have the affine reflection
$A_{Q}\co\ft^{*} \to \ft^{*}$ and the linear version $\widehat{A}_{Q}:\ft^{*} \to \ft^{*}$ from
\eqref{e:Areflect} and \eqref{e:reflect}.  Observe that since $(\widehat{A}_{Q})^{2} = \id$ it follows that
\begin{equation}\label{e:PsiLNF}
	\Psi_{A_{Q}}\co(\ft^{*} \times \T, d\l) \to (\ft^{*} \times \T, d\l) \quad\mbox{by}\quad
	\Psi_{A_{Q}}(v, \eta) = (A_{Q}(v), \widehat{A}_{Q}^{*}(\eta))
\end{equation}
is a symplectomorphism with respect to the canonical symplectic form \eqref{e:CSF}.  Comparing
\eqref{e:NFS1C} and \eqref{e:NFS2C} we see that to establish Stage 3 it suffices to prove
the following proposition, which can be seen as a local version of \cite[Proposition 5.5]{MT10}.

\begin{prop}\label{p:LHI}
	Let $Q \subset \De$ be a symmetric probe in the moment polytope for a
	symplectic toric orbifold $(M^{2n}, \w, \T, \Phi)$.  Then for any neighborhood
	of $\cN$ of $\Phi^{-1}(Q) \subset M$, there is a Hamiltonian isotopy of $(M, \w)$ supported
	in $\cN$ with time one map $\Psi$ such that 
	$$
		\Psi|_{\cU} = \Psi_{A_{Q}}|_{\cU}
	$$
	for a smaller neighborhood $\cU \subset \cN$ of $\Phi^{-1}(Q)$, where $\Psi_{A_{Q}}$
	is given by \eqref{e:PsiLNF}.
\end{prop}
\begin{proof}[Proof of special case of Proposition~\ref{p:LHI}]
	Consider the special case where 
%	$(M^{2n}, \w, \T, \Phi)$ is a symplectic toric manifold
%	and 
	$v_{Q}$ is parallel to every facet except $F_{Q}$ and $F_{Q}'$, meaning $\ip{\eta_{F}, v_{Q}} = 0$
	for all other interior conormals $\eta_{F} \in \ft_{\Z}$.  We have that
	$$
		\De = \bigcap_{j=1}^{N}\{x \in \ft^{*} \mid \ip{\eta_{j}, x} + \kappa_{j} \geq 0\}
	$$
	where without loss of generality $\{\eta_{1} = \eta_{F_{Q}}, \eta_{2}, \dots, \eta_{n}\}$
	is a lattice basis for $\ft_{\Z}$ and $\eta_{n+1} = \eta_{F_{Q}}'$.
	Let $\{v_{1} = v_{Q}, v_{2}, \dots, v_{n}\}$ be a dual basis for $\ft_{\Z}^{*}$.
	We have that
	$$
		\eta_{n+1} = -\eta_{1} + \sum_{i=2}^{n}a_{n+1}^{i}\, \eta_{i}
		\quad\mbox{and}\quad
		\eta_{n+j} = \sum_{i=2}^{n} a_{n+j}^{i}\,\eta_{i} \quad\mbox{for $j \geq 2$}.
	$$
	since $\ip{\eta_{1}, v_{Q}} = 1 = \ip{-\eta_{n+1}, v_{Q}}$ and $\ip{\eta_{j}, v_{Q}} = 0$ otherwise.
	
	We can identify $(M^{2n}, \w, \T, \Phi)$ with $(M_{\De}, \w_{\De}, \T, \Phi_{\De})$,
	which is built by performing symplectic reduction on the standard
	$(\C^{N}, \w_{0}, (S^{1})^{N}, \Phi_{0})$.
	In particular it has the form
	$(M_{\De}, \w_{\De}) = (Z/K, \overline{\w}_{0})$ where the level set
	\begin{equation*}
		Z = \{H_{n+1}(z) = c_{n+1}\,, \dots, H_{N}(z) = c_{N}\} \subset (\C^{N}, \w_{0})
	\end{equation*}
	for the Hamiltonians
	\begin{align}\notag
		H_{n+1}(z) &:= \abs{z_{1}}^{2} + \abs{z_{n+1}}^{2} - \sum_{i=2}^{n} a^{i}_{n+1}\, \abs{z_{i}}^{2}\,;
		&c_{n+1} := \kappa_{1} + \kappa_{n+1} - \sum_{i=2}^{n}a_{n+1}^{i}\,\kappa_{i} %\quad \mbox{and}
		\label{e:Hn+1}\\
		H_{n+j}(z) &:= \abs{z_{n+j}}^{2} - \sum_{i=2}^{n} a_{n+j}^{i}\, \abs{z_{i}}^{2}\,;
		&c_{n+j} := \kappa_{n+j} - \sum_{i=2}^{n} a_{n+j}^{i}\, \kappa_{i} \quad\mbox{for $j \geq 2$}
	\end{align}
	is symplectically reduced using the action of the $N-n$ dimensional subtorus 
	$K \subset (S^{1})^{N}$ whose action is
	given by the Hamiltonians $H_{n+1}, \dots, H_{N}$ on $\C^{N}$.
	The moment map for the action of $\T$ on $Z/K$ is given by
	\begin{equation*}
		\Phi_{\De}(z) = \sum_{j=1}^{n} (\abs{z_{j}}^{2} - \kappa_{j})\, v_{j} \in \ft^{*}.
	\end{equation*}
	For a point $q$ on the probe $Q$, since $v_{1} = v_{Q}$ and $\ip{\eta_{j}, v_{Q}} = 0$
	for $j=2, \dots, n$, it follows that
	$$
		Q =\left\{\R v_{1} + \sum_{k=2}^{n} \ip{\eta_{k}, q} v_{k}\right\} \cap \De \subset \ft^{*}
	$$
	and hence $\Phi^{-1}_{\De}(Q)$ is
	\begin{equation}\label{e:QLSNF}
%		\Phi^{-1}_{\De}(Q) = 
		\left\{z \in Z : \abs{z_{2}}^{2} = \ip{\eta_{2}, q} + \kappa_{2}, \dots,
		\abs{z_{n}}^{2} = \ip{\eta_{n}, q} + \kappa_{n}\right\}/K \subset M_{\De}.
	\end{equation}
	
	Now the standard Hamiltonian $U(2)$ action on $\C e_{1} \times \C e_{n+1} \subset \C^{N}$
	preserves the level set $Z$ and commutes with the action of $K$, so it descends to
	a Hamiltonian $U(2)$ action on $M_{\De} = Z/K$.
	Consider the element $B \in U(2)$ so that $B(z_{1}, z_{n+1}) = (z_{n+1}, z_{1})$.
	Since
	$$
		A_{Q}(x) = x + \left\langle-2\eta_{1} + 
		\sum_{i=2}^{n} a_{n+1}^{i}\,\eta_{i}\,,\, x\right\rangle v_{1} + (\ka_{n+1} - \ka_{1})\,v_{1} 
	$$
	it follows for $z \in Z \subset \C^{N}$ that we have
	\begin{align*}
		(A_{Q} \circ \Phi_{\De})(z) &= A_{Q}\left(\sum_{j=1}^{n} (\abs{z_{j}}^{2} - \kappa_{j})\, v_{j}\right)\\
		&= \left(-\abs{z_{1}}^{2} + \ka_{n+1} + \sum_{i=2}^{n}a_{n+1}^{i}(\abs{z_{i}}^{2} - \ka_{i})\right)v_{1}
		+ \sum_{j=2}^{n} (\abs{z_{j}}^{2} - \ka_{j})\, v_{j}\\
		&= (\abs{z_{n+1}}^{2} - \ka_{1})\,v_{1} + \sum_{j=2}^{n} (\abs{z_{j}}^{2} - \ka_{j})\, v_{j}
		= (\Phi_{\De} \circ B)(z)
	\end{align*}
	where the second to last equality uses that $H_{n+1}(z) = c_{n+1}$ from 
	\eqref{e:Hn+1} for
	points on $Z$.
	Therefore up to applying a uniform rotation using the toric action, we have that $B \in U(2)$
	acts on $(M_{\De}, \w_{\De})$ as the Hamiltonian
	diffeomorphism $\Psi_{A_{Q}} \in \Ham(M_{\De}, \w_{\De})$ from \eqref{e:PsiLNF}.
	
	Now let $X \in \mathfrak{u}(2)$ be such that $\exp(X) = B$ and let $H(z_{1}, z_{n+1})$
	be the autonomous Hamiltonian whose corresponding Hamiltonian flow $\vp^{H}_{t}$ on $(\C^{N}, \w_{0})$
	is the action of $\exp(tX) \in U(2)$.  Since $H$ Poisson commutes with $\abs{z_{j}}^{2}$ for
	$j=2, \dots, n$, it follows that $\vp^{H}_{t}$ in $\Ham(M_{\De}, \w_{\De})$ preserves
	level sets of the form
	$$
		\left\{z \in Z : \abs{z_{2}}^{2} = b_{2}\,,\, \dots\,,\, \abs{z_{n}}^{2} = b_{n}\right\}/K \subset M_{\De}
	$$
	in particular $\Phi_{\De}^{-1}(Q)$ is preserved.  Now let $\cN \subset M_{\De}$ be any neighborhood
	of $\Phi_{\De}^{-1}(Q)$ and let $\r = \r(\abs{z_{2}}^{2}, \dots, \abs{z_{n}}^{2})$ be a bump function that 
	is a constant $1$ near the level set \eqref{e:QLSNF} and $\supp(\r) \subset \cN$.  Then the time 
	one flow $\Psi = \vp^{\r H}_{1}$ for the Hamiltonian $\r H\co M_{\De} \to \R$ is the desired
	element of $\Ham(M_{\De}, \w_{\De})$.
\end{proof}
\begin{proof}[Proof of general case of Proposition~\ref{p:LHI}]
	Let $\{\eta_{1}' = \eta_{F_{Q}}, \eta_{2}', \dots, \eta_{n}'\} \in \ft_{\Z}$ be a lattice basis
	satisfying \eqref{e:LB} with respect to $v_{Q}$.  	
	If $q$ is a point on $Q$, then for $\e > 0$ define 
	the rational half-spaces 
	$$
		\cH_{2k-1} = \{x \in \ft^{*} \mid \ip{\eta_{k}', x - q} + \eps \geq 0\}
		\quad\mbox{and}\quad
		\cH_{2k} = \{x \in \ft^{*} \mid \ip{-\eta_{k}', x - q} + \eps \geq 0\}
	$$
	for $k = 2, \dots, n$.  Since $Q$ starts and ends at points in the interior of the facets $F_{Q}$
	and $F_{Q}'$ of $\De$, respectively, for $\e > 0$ sufficiently small
	$$
		\De_{Q,\e} =  \{x \in \ft^{*} \mid \ip{\eta_{F_{Q}}, x} + \kappa \geq 0\}
		\cap \{x \in \ft^{*} \mid \ip{\eta_{F_{Q}'}, x} + \kappa' \geq 0\}
		\cap \bigcap_{j=3}^{2n} \cH_{j} \subset \De
	$$
	is a neighborhood of $\De_{Q,0} = Q \subset \De$.  Furthermore $\De_{Q, \e}$ is the moment
	polytope for a symplectic toric manifold $(Y^{2n}_{\e}, \w_{\e}, \T, \Phi)$, that satisfies the 
	special condition that all facets except $F_{Q}$ and $F_{Q}'$ are parallel to $v_{Q}$.
	By the special case of Propositon~\ref{p:LHI} we can build the desired Hamiltonian isotopy $\Psi$
	in $\Ham(Y_{\e}^{2n}, \w_{\e})$ that is generated by an autonomous Hamiltonian
	supported in $\cN_{\e} = \Phi^{-1}(\De_{Q, \e/2}) \subset Y^{2n}_{\e}$.  Since $\cN_{\e}$ 
	canonically embeds into $(M, \w)$, preserving the toric structure, we can see $\Psi \in \Ham(M, \w)$
	as our desired Hamiltonian isotopy.
\end{proof}
%%%%%%%%%%%%%%%%%%%%%%%

%%%%%%%%%%%%%%%%%%%%%%%%%%%%%
%%%%%%%%%%%%%%%%%%%%%%%%%%%%%

%%%%%%%%%%%%%%%%%%%%%%%%%%%%%
%%%%%%%%%%%%%%%%%%%%%%%%%%%%%
\subsection{Proving Theorem~\ref{t:pdeflected}: Parallel extended probes with flags}

Let $\cP = P \cup Q \cup \cF$ be a parallel extended probe with flag. 
Since the direction $v_{P}$ of $P$ is parallel to the facet $F_{Q}$,
we can pick dual lattice bases $\{\eta_{1}, \dots, \eta_{n}\}$ for $\ft_{\Z}$ and  
$\{e_{1}, \dots, e_{n}\}$ for $\ft_{\Z}^{*}$ so that
$$ 
	v_{P} = e_{1}\,,\quad v_{Q} = e_{2}\,,\quad 
	\eta_{F_{P}} = \eta_{1} + \sum_{k=2}^{n}a_{k}\eta_{k}\,,\quad
	\eta_{F_{Q}} = \eta_{2}.
$$ 
Picking action-angle coordinates $(x,\th)$ on $\ft^{*} \times \T$ with respect to these bases, we can let 
the points on the probe $P \subset \ft^{*}$ have coordinates
$$ 
	b_{P} = (0, r_{PQ}, b_{3}, \dots, b_{n})\,, \quad
	u = (a, r_{PQ}, b_{3}, \dots, b_{n})\,, \quad 
	x_{PQ} = (\ell_{P}, r_{PQ}, b_{3}, \dots, b_{n})
$$ 
where $\ell_{P}$ is the length of $P$ and $0 < a < \ell_{P}$ is $a = d_{\aff}(u, F_{P}) = d_{v_{P}}(u, b_{P})$.
Besides $x_{PQ}$, let the other points on $Q \subset \ft^{*}$ be
\begin{equation}\label{e:ptQ}
	b_{Q} = (\ell_{P}, 0, b_{3}, \dots, b_{n})\,,\,\,
	x_{\cF} = (\ell_{P}, r_{\cF}, b_{3}, \dots, b_{n})\,,\,\,
	x_{\cF}' = (\ell_{P}, r_{\cF}', b_{3}, \dots, b_{n})
\end{equation}	
where $r_{PQ}, r_{\cF}, r_{\cF}'$ are positive and $r_{\cF} < r_{\cF}'$.
By \eqref{e:ptQ}, the end points of $\cF$ are
$$ 
	e_{\cF} = (\ell, r_{\cF}, b_{3}, \dots, b_{n}) \quad\mbox{and}\quad 
	e_{\cF}' = (\ell, r_{\cF}', b_{3}, \dots, b_{n})
$$
where $\ell = \ell_{P} + \ell_{\cF}$ is the length of the extended probe $\cP$ and $\ell_{\cF}$ is the length of the flag $\cF$.
In our coordinates for $\ft^{*}$ the flag $\cF$ is given by
\begin{equation*}
	\cF = \big\{(x_{1}, x_{2}, b_{3}, \dots, b_{n}) \in \ft^{*} \mid 
	\ell_{P} \leq x_{1} \leq \ell\,,\,\, r_{\cF} \leq x_{2} \leq r_{\cF}'\big\}.
\end{equation*}

%%%%%%%%%%%%%%%%%%%%%%%
\subsubsection{Stage 1 for Theorem~\ref{t:pdeflected}}

In the action-angle coordinates $(x, \th)$ for $(M, \w)$, the embedding for the 
probe $P$
$$ 
	\psi_{P}\co\bD(\ell_{P}) \times \T^{n-1} \to (M^{2n}, \w)
$$ 
from \eqref{e:EAAC} has the form
\begin{equation}\label{e:ECpsiP}
	\psi_{P}(s, \phi_{1}, \phi_{2}, \dots, \phi_{n}) = 
%	\begin{pmatrix}
%	x\\ \th
%	\end{pmatrix}
%	=
	\begin{pmatrix}
	s & r_{PQ} & b_{3} & \cdots & b_{n}\\
	\phi_{1} & a_{2}\phi_{1} + \phi_{2}& a_{3}\phi_{1} + \phi_{3}& \cdots &	a_{n}\phi_{1} + \phi_{n}
	\end{pmatrix}.
\end{equation}
%%%%%%%%%%%%%%%%%%%%%%%

%%%%%%%%%%%%%%%%%%%%%%%
\subsubsection{Stage 2 for Theorem~\ref{t:pdeflected}}

It follows from \eqref{e:pflag} that 
$r_{PQ} < r_{\cF}' - r_{\cF}$ and therefore by Lemma~\ref{l:APD} there is a Hamiltonian isotopy $\r$ of 
$(\bD(r_{\cF}'), \w_{0})$ supported in the interior so that
\begin{equation}\label{e:rho}
	\r\big(S^{1}(r_{PQ})\big) \subset \bA(r_{\cF}, r_{\cF}').
\end{equation}
If $(x_{2},\, \th_{2}) \in [0, r_{\cF}'] \times \T^{1}$ are action-angle coordinates on $(\bD(r_{\cF}'), \w_{0})$,
then write this Hamiltonian diffeomorphisms as $\r = (\r_{x_{2}},\, \r_{\th_{2}})$.

Now let $(s, \phi_{1}) \in [\ell_{P}, \ell] \times \T^{1}$ be action-angle coordinates on 
$(\bA(\ell_{P}, \ell), \w_{0})$ and using the action-angle coordinates $(x, \th)$ on $(M, \w)$ define the embedding
\begin{equation}\label{e:EpsiF}
	\psi_{\cF}\co\bA(\ell_{P}, \ell) \times \T^{n-1} \to (M^{2n}, \w)
\end{equation}
$$
	\psi_{\cF}(s, \phi_{1}, \phi_{2}, \dots, \phi_{n}) =  
%	\begin{pmatrix}
%	x\\ \th
%	\end{pmatrix}
%	=
	\begin{pmatrix}
	s & \r_{x_{2}}(r_{PQ}\,,a_{2}\phi_{1} + \phi_{2}) & b_{3} 
	& \cdots & b_{n}\\
	\phi_{1} & \r_{\th_{2}}(r_{PQ}\,,a_{2}\phi_{1} + \phi_{2}) & a_{3}\phi_{1} + \phi_{3}
	& \cdots & a_{n}\phi_{1} + \phi_{n}
	\end{pmatrix}.
$$
Observe that the formula for $\psi_{\cF}$ is just the result of applying $\r$ to the $(x_{2}, \th_{2})$
coordinates in the formula \eqref{e:ECpsiP} for $\psi_{P}$.  It is straightforward to check that this embedding $\psi_{\cF}$ satisfies the conditions for \eqref{e:psiE}.
%%%%%%%%%%%%%%%%%%%%%%%

%%%%%%%%%%%%%%%%%%%%%%%
\subsubsection{Stage 3 for Theorem~\ref{t:pdeflected}}

In the standard toric structure $(\C^{n}, \w_{0}, (S^{1})^{n}, \Phi_{0})$, consider the probe 
$Q_{0} \subset \R^{n}_{+}$ given by 
$$
Q_{0} = \{(b_{1}, x_{2}, b_{3}, \dots, b_{n}) \mid x_{2} \in [0, q]\}
$$ 
where $b_{k}$ are positive, then
\begin{equation}\label{e:diskmodel}
	\Phi_{0}^{-1}(Q_{0}) = S^{1}(b_{1}) \times \bD(q) \times \prod_{k\ge 3} (S^{1}(b_k)) 
	\subset \C \times \C \times \C^{n-2}.
\end{equation}
For small $\eps\gg\de\ge0$ we define the subsets of $\R^{n}_{+}$
\begin{gather*}
\cN^{\eps,\de}: = [b_{1}-\eps, b_{1}+ \eps] \times [0,q-\de] \times \prod_{k\ge 3} [b_k-\eps,b_k+\eps],\\
X^\eps: = [b_1-\eps,b_1+\eps] \times [0,q] \times (b_3,\dots,b_n).
\end{gather*}
\begin{lem}\label{l:deflecting} 
	Let $\r$ be any area preserving diffeomorphism of $\bD(q)$ that is the identity near
	$\partial \bD(q)$.
	If $\cN_0: = \cN^{\eps,\de}$ for $\de > 0$ sufficiently small, then
	there exists a Hamiltonian isotopy of $\C^{n}$ supported in 
	${\rm int\,}\bigl(\Phi_{0}^{-1}(\cN_0)\bigr)$ such that
	the time one map of the isotopy $\Psi$ in 
	a small neighborhood $\mathcal{U}$ of $\Phi_{0}^{-1}(Q_{0}) \subset \C^{n}$ 
	is given by
	\begin{equation}\label{e:LFPSI}
		(\Psi^{\r})|_{\mathcal{U}} = (\id_{\C} \times \r \times \id_{\C^{n-2}})|_{\mathcal{U}}
	\end{equation}
	in terms of the decomposition in \eqref{e:diskmodel} and in particular
	%it also satisfies 
	$\Psi^{\r}(\Phi_{0}^{-1}(X^\eps)) = \Phi_{0}^{-1}(X^\eps)$.
\end{lem}
\begin{proof}
	The map $\r$ is the time one map of a Hamiltonian isotopy,
	generated by some time-dependent Hamiltonian $H$ with support in $\Int \bD(q)$.
	Simply multiply $H$ by a cut-off function $\alpha\co\C^{n} \to [0,1]$ 
	that is a function of the variable $y: = \sum_{k\not= 2}(x_k - b_{k})^2$
	and $\a(y) \equiv 1$ near $y=0$.
	The time one map $\Psi^{\r}$ of the Hamiltonian isotopy generated by $\alpha H$ has the 
	desired properties.
\end{proof}

Take now $Q_{0} \subset \R^{n}_{+}$ with $b_{1} = \ell_{P}$ and $q = r_{\cF}'$ 
to be a local model for our probe $Q \subset \De$.  
By applying Lemma~\ref{l:deflecting} to the Hamiltonian diffeomorphism $\r$ of $\bD(r_{\cF}')$ from \eqref{e:rho} in Stage~2, the resulting Hamiltonian diffeomorphism $\Psi^{\r}$ can be extended by the identity
outside its support to be an element of $\Ham(M, \w)$.  It follows from \eqref{e:LFPSI} that
near $\Phi^{-1}(Q)$ the Hamiltonian diffeomorphism $\Psi^{\r}$ has the form
\begin{equation}\label{e:PsiD}
	\Psi^{\r}\begin{pmatrix} x \\ \th \end{pmatrix} =
	\begin{pmatrix}
	x_{1} & \r_{x_{2}}(x_{2}\,,\th_{2}) & x_{3} & \cdots & x_{n}\\
	\th_{1} & \r_{\th_{2}}(x_{2}\,,\th_{2}) & \th_{3} & \cdots & \th_{n}
	\end{pmatrix}
\end{equation}
in our action-angle coordinates $(x, \th)$.  

Comparing \eqref{e:ECpsiP} and \eqref{e:EpsiF}, it is clear that
$$
	\Psi^{\r} \circ \psi_{P}\co\bD(\ell_{P}) \times \T^{n-1} \to (M, \w)
	\quad\mbox{and}\quad
	\psi_{\cF} \co\bA(\ell_{P}, \ell) \times \T^{n-1} \to (M, \w)
$$
glue together to form an embedding $\psi_{\cP}$ as in \eqref{e:FE}.  
%%%%%%%%%%%%%%%%%%%%%%%

%%%%%%%%%%%%%%%%%%%%%%%%%%%%%
%%%%%%%%%%%%%%%%%%%%%%%%%%%%%

%%%%%%%%%%%%%%%%%%%%%%%%%%%%%
%%%%%%%%%%%%%%%%%%%%%%%%%%%%%
\subsection{Proving Theorem~\ref{t:deflected}: Extended probes with flags}

Let $\cP = P \cup Q \cup \cF_{\mu}$ be a extended probe with flag. 
We can pick dual lattice bases $\{\eta_{1}, \dots, \eta_{n}\}$ for $\ft_{\Z}$ and  
$\{e_{1}, \dots, e_{n}\}$ for $\ft_{\Z}^{*}$ so that
$$ 
	v_{P} = c_{1}e_{1} + c_{2}e_{2}\,,\quad v_{Q} = e_{2}\,,\quad 
	\eta_{F_{P}} = \sum_{k=1}^{n}a_{k}\eta_{k}\,,\quad
	\eta_{F_{Q}} = \eta_{2}
$$ 
where $c_{1}, c_{2} \in \Z$ are relatively prime and without loss of generality $c_{1} > 0$.
Picking action-angle coordinates $(x,\th)$ on $\ft^{*} \times \T$ with respect to these bases, we can let 
the points on the probe $P \subset \ft^{*}$ have coordinates
\begin{equation*}
\begin{split}	
	b_{P} = (0, b_{2}, b_{3}, \dots, b_{n})\,, \quad
	u = b_{P} + a(c_{1}, c_{2}, 0, \dots, 0)\,, \quad \\
	x_{PQ} = b_{P} + \ell_{P}(c_{1}, c_{2}, 0, \dots, 0) = (\ell_{P}c_{1}, r_{PQ}, b_{3}, \dots, b_{n})
\end{split}
\end{equation*}
where $\ell_{P}$ is the length of the probe $P$.
The points on $Q \subset \ft^{*}$ are
\begin{equation}\label{e:ptQnp}
\begin{split}
	b_{Q} = (\ell_{P}c_{1}, 0, b_{3}, \dots, b_{n})\,,\quad\quad
	x_{PQ} = (\ell_{P}c_{1}, r_{PQ}, b_{3}, \dots, b_{n})\\
	x_{\cF} = (\ell_{P}c_{1}, r_{\cF}, b_{3}, \dots, b_{n})\,,\quad\quad
	x_{\cF}' = (\ell_{P}c_{1}, r_{\cF}', b_{3}, \dots, b_{n})
\end{split}
\end{equation}	
where $r_{PQ}, r_{\cF}, r_{\cF}'$ are positive and $r_{\cF} < r_{\cF}'$.
By \eqref{e:ptQnp}, the end points of $\cF_{\mu}$ are
\begin{align*}
	e_{\cF_{\mu}} = (\ell c_{1}\,,\, r_{\cF} -\ell_{\cF} \mu c_{2}\,,\, b_{3}, \dots, b_{n}) \quad\mbox{and}\quad \\
	e_{\cF_{\mu}}' = (\ell c_{1}\,,\, r_{\cF}' + \ell_{\cF} (1-\mu)c_{2}\,,\, b_{3}, \dots, b_{n})
\end{align*}
where $\ell = \ell_{P} + \ell_{\cF}$ is the length of the extended probe $\cP$
and $\ell_{\cF}$ is the length of the flag $\cF$.

%%%%%%%%%%%%%%%%%%%%%%%
\subsubsection{Stage 1 for Theorem~\ref{t:deflected}}

In action-angle coordinates $(x, \th)$ for $(M, \w)$ the embedding associated to the 
probe $P$
$$ 
	\psi_{P}\co\bD(\ell_{P}) \times \T^{n-1} \to (M^{2n}, \w)
$$ 
from \eqref{e:EAAC} has the form
\begin{equation}\label{e:ECpsiPmu}
	\psi_{P}(s, \phi_{1}, \phi_{2}, \dots, \phi_{n}) = 
	\begin{pmatrix}
	x\\ \th
	\end{pmatrix}
	=
	\begin{pmatrix}
	s\,c_{1} & b_{2} + s\,c_{2} & b_{3} & \cdots & b_{n}\\
	\th_{1}(\phi) & \th_{2}(\phi) & \th_{3}(\phi)& \cdots &	\th_{n}(\phi)
	\end{pmatrix}.
\end{equation}
%%%%%%%%%%%%%%%%%%%%%%%

%%%%%%%%%%%%%%%%%%%%%%%
\subsubsection{Stage 2 for Theorem~\ref{t:deflected}}

Assume now that $c_{2} \leq 0$, so that $v_{P}$ points towards $F_{Q}$.
By Lemma~\ref{l:APD} there is a 
compactly supported Hamiltonian isotopy $\r$ of $(\bD(r_{\cF}'), \w_{0})$ such that
$$
	\r(\bD(t)) \subset \bA(a_{\mu}(t), a_{\mu}(t) + t + \e) \quad\mbox{for $\e \leq t \leq r_{\cF}' - r_{\cF}-\e$}
$$
where 
$$
	a_{\mu}(t) = r_{\cF} - \mu t + \mu(r_{\cF}' - r_{\cF}).
$$
By \eqref{e:flag}, we have that $r_{PQ} < r_{\cF}' - r_{\cF}$ and hence
\begin{equation}\label{e:fitflag}
	\r(S^{1}(r_{PQ} + \l c_{2})) \subset \bA\big(r_{\cF} - \l c_{2}\mu\,,\,\, r_{\cF}' + \l (1-\mu)c_{2}\big)
	\quad\mbox{for $0 \leq \l \leq \ell_{\cF}$}.
\end{equation}

Assume that $c_{2} \geq 0$, so that $v_{P}$ points away from $F_{Q}$.
By Lemma~\ref{l:APD} there is a compactly supported Hamiltonian isotopy $\r$ of 
$(\bD(r_{\cF}'+ \ell_{\cF}(1-\mu)c_{2}), \w_{0})$ such that
$$
	\r(\bD(t)) \subset \bA(a_{\mu}(t), a_{\mu}(t) + t + \e) \quad\mbox{for $\e \leq t \leq 
	r_{\cF}' - r_{\cF} + \ell_{\cF}c_{2}-\e$}
$$
where 
$$
	a_{\mu}(t) = r_{\cF} - \mu t + \mu(r_{\cF}' - r_{\cF}).
$$
By \eqref{e:flag}, we have that $r_{PQ} < r_{\cF}' - r_{\cF}$ and hence
\begin{equation}\label{e:fitflag+}
	\r(S^{1}(r_{PQ} + \l c_{2})) \subset \bA\big(r_{\cF} - \l c_{2}\mu\,,\,\, r_{\cF}' + \l (1-\mu)c_{2}\big)
	\quad\mbox{for $0 \leq \l \leq \ell_{\cF}$}.
\end{equation}

Now let $(s, \phi_{1}) \in [\ell_{P}, \ell] \times \T^{1}$ be action-angle coordinates on 
$(\bA(\ell_{P}, \ell), \w_{0})$ and using the action-angle coordinates $(x, \th)$ on $(M, \w)$ define the embedding
\begin{equation*} 
\begin{split}
	\psi_{\cF_{\mu}}\co\bA(\ell_{P}, \ell) \times \T^{n-1} \to (M^{2n}, \w) \hspace{2in}\\
	\psi_{\cF_{\mu}}(s, \phi_{1}, \phi_{2}, \dots, \phi_{n}) =  
%	\begin{pmatrix}
%	x\\ \th
%	\end{pmatrix}
%	=
	\begin{pmatrix}
	s\,c_{1} & \r_{x_{2}}(b_{2} + s\,c_{2}\,, \th_{2}(\phi)) & b_{3} 
	& \cdots & b_{n}\\
	\th_{1}(\phi) & \r_{\th_{2}}(b_{2} + s\,c_{2}\,, \th_{2}(\phi)) & \th_{3}(\phi)
	& \cdots & \th_{n}(\phi)
	\end{pmatrix}.
\end{split}
\end{equation*}
Observe that the formula for $\psi_{\cF_{\mu}}$ is just the result of applying $\r$ to the $(x_{2}, \th_{2})$
coordinates in the formula \eqref{e:ECpsiPmu} for $\psi_{P}$.  It is straightforward to check that this embedding $\psi_{\cF_{\mu}}$ satisfies the conditions for \eqref{e:psiE}, in particular
$\im (\psi_{\cF_{\mu}}) \subset \Phi^{-1}(\cF_{\mu})$ follows from \eqref{e:fitflag} and \eqref{e:fitflag+}.

\begin{remark}
	When $c_{2} < 0$, note that 
	$-r_{PQ}/c_{2} = d_{v_{P}}(x_{PQ}, F_{Q})$.  So the second assumption in 
	\eqref{e:flag}, i.e. $\ell_{\cF} < d_{v_{P}}(x_{PQ}, F_{Q})$, ensures that 
	$S^{1}(r_{PQ} +\l c_{2})$ in \eqref{e:fitflag} does not collapse to a point.
	If it did collapse to a point, then $\psi_{\cF_{\mu}}$ would no longer be an embedding
	and this is necessary for our proof.	
\end{remark}
%%%%%%%%%%%%%%%%%%%%%%%

%%%%%%%%%%%%%%%%%%%%%%%
\subsubsection{Stage 3 for Theorem~\ref{t:deflected}}

The rest of the proof is now the same as in the parallel case. 
By applying Lemma~\ref{l:deflecting} to the Hamiltonian diffeomorphism $\r$ of $\bD(r_{\cF}')$ from
in Stage 2, the resulting Hamiltonian diffeomorphism $\Psi^{\r}$ can be extended by the identity
outside its support to be an element of $\Ham(M, \w)$.  It follows from \eqref{e:LFPSI} that
near $\Phi^{-1}(Q)$ the Hamiltonian diffeomorphism $\Psi^{\r}$ has the form
$$ 
	\Psi^{\r}\begin{pmatrix} x \\ \th \end{pmatrix} =
	\begin{pmatrix}
	x_{1} & \r_{x_{2}}(x_{2}\,,\th_{2}) & x_{3} & \cdots & x_{n}\\
	\th_{1} & \r_{\th_{2}}(x_{2}\,,\th_{2}) & \th_{3} & \cdots & \th_{n}
	\end{pmatrix}
$$ 
in our action-angle coordinates $(x, \th)$.  

Comparing \eqref{e:ECpsiP} and \eqref{e:EpsiF}, it is clear that
$$
	\Psi^{\r} \circ \psi_{P}\co\bD(\ell_{P}) \times \T^{n-1} \to (M, \w)
	\quad\mbox{and}\quad
	\psi_{\cF} \co\bA(\ell_{P}, \ell) \times \T^{n-1} \to (M, \w)
$$
glue together to form an embedding $\psi_{\cP}$ as in \eqref{e:FE}. 
%%%%%%%%%%%%%%%%%%%%%%%

%%%%%%%%%%%%%%%%%%%%%%%
\subsubsection{Hamiltonian diffeomorphisms of the disk and the associated flags}

The area preserving diffeomorphisms of a disk to which we applied Lemma~\ref{l:deflecting}
come from the following lemma.

For real numbers $0<A < B$, pick a smooth function $a\co [0, B-A] \to [A, B]$
that is non-increasing, is such that
$$
A \leq a(s) + s \leq B\,,\,\, a(B-A) = A\,,
$$
and the function $b(s) := a(s) + s$ is non-decreasing for $s \in [0, B-A]$.
Above we picked $a$ to have the form
$$ 
	a_{\mu}(s) = A -\mu s + \mu (B-A) \quad\mbox{for $\mu \in [0,1]$}
$$
where the parameter $\mu$ corresponds with the flag parameter.

\begin{figure}[h]

	\begin{center} 
	\leavevmode 
	\includegraphics[width=4in]{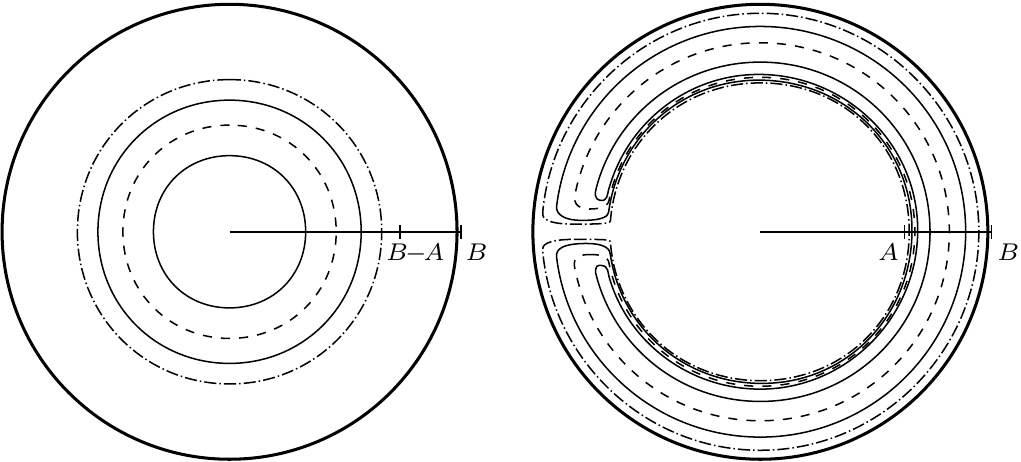}
	\end{center} 
	
	\caption{The family of circles that build the Hamiltonian diffeomorphism $\r$ associated to
	$a_{0}$, drawn with $A = 4$ and $B=9$.}	
	\label{f:APD}
\end{figure}

\begin{lem}\label{l:APD}
For any function $a(s)$ as above and any $\eps>0$, there is a 
compactly supported Hamiltonian diffeomorphism 
$\r:\Int\bD(B)\to \Int\bD(B)$
such that 
$$
	\r(\bD(s))\subset \bA(a(s),\, a(s)+s+\eps) \quad\mbox{for all $\eps \leq s \leq B-A-\eps$.}
$$
\end{lem}
\begin{proof}
Choose a smooth family of disjoint, contractible closed curves 
$$
	\g_s \subset \bA(a(s),\, a(s)+s+\eps) \quad s \in [\eps, B-A-\eps]
$$
that each enclose a region of area $s$.  This is possible because $\eps>0$.  
Next pick a compactly supported diffeomorphism $\psi$ of $\Int\bD(B)$ 
such that each circle $\d \bD(s)$ is mapped to $\g_s$ for all $s\in (\eps, B-A-\eps)$.
Finally isotope $\psi$ to an area preserving $\r$ via Moser's method.
Using that $\g_s$ encloses the same amount of area as $\d \bD(s)$,
it is not hard to check that the isotopy is given by flowing along a vector field $X_t$ that at each time $t$ is tangent to the curves $\g_s$.
\end{proof}
%%%%%%%%%%%%%%%%%%%%%%%

%%%%%%%%%%%%%%%%%%%%%%%%%%%%%
%%%%%%%%%%%%%%%%%%%%%%%%%%%%%

%%%%%%%%%%%%%%%%%%%%%%%%%%%%%%%%%%%%
%%%%%%%%%%%%%%%%%%%%%%%%%%%%%%%%%%%%
%%%%%%%%%%%%%%%%%%%%%%%%%%%%%%%%%%%%

%%%%%%%%%%%%%%%%%%%%%%%%%%%
%\bibliographystyle{alpha}

%%%%%%%%%%%%%%%%%%%%%%%%%%%

\end{document}